\documentclass[11pt]{amsart}
\usepackage{amssymb,euscript,tikz,units}
\usepackage[colorlinks,citecolor=blue,linkcolor=red]{hyperref}
\usepackage{verbatim}
\usepackage[nameinlink]{cleveref}
\usepackage{colonequals}
\usepackage{tikz}
\usepackage{enumitem}
\usepackage[justification=centering]{caption}
\usepackage{amsfonts,amssymb,amsmath,amsthm,amscd}
\allowdisplaybreaks

\newcommand{\bracedcdot}{ }
\newcommand{\N}{\mathbb{N}}
\newcommand{\calT}{\mathcal{T}}
\newcommand{\T}{\calT}
\newcommand{\M}{\mathcal{M}}
\newcommand{\sM}{\mathcal{M}}

\newcommand{\sL}{\mathcal{L}}
\newcommand{\sN}{\mathcal{N}}
\newcommand{\R}{\mathbb{R}}
\newcommand{\Z}{\mathbb{Z}}

\newcommand{\wep}{Weil-Petersson}

\newcommand{\sbs}{\subset}
\newcommand{\limg}{\lim_{g\rightarrow\infty}}
\newcommand{\sT}{\mathcal{T}}
\newcommand{\dvol}{d\mathrm{vol}}
\newcommand{\E}{\mathbb{E}_{\rm WP}^g}

\newcommand{\ve}{\boldsymbol}

\newcommand{\eg}{\textit{e.g.\@ }}
\newcommand{\cf}{\textit{c.f.\@ }}
\newcommand{\ie}{\textit{i.e.\@ }}

\def\sys{\mathop{\rm sys}}
\def\area{\mathop{\rm Area}}
\def\arcsinh{\mathop{\rm arcsinh}}

\def\Vol{\mathop{\rm Vol}}
\def\dist{\mathop{\rm dist}}

\def\Prob{\mathop{\rm Prob}\nolimits_{\rm WP}^g}
\def\Mod{\mathop{\rm Mod}}
\def\Sym{\mathop{\rm Sym}}

\newcommand{\ls}{\ell_{\sys}}
\newcommand{\lss}{\ell_{\sys}^{\rm sep}}
\newcommand{\lns}{\ell_{\sys}^{\rm ns}}
\newcommand{\exl}{\rm Ext}

\newcommand{\exlss}{\rm Ext_{\sys}^{\rm sep}}

\DeclareMathOperator{\Diff}{Diff}
\DeclareMathOperator{\Teich}{Teich}
\DeclareMathOperator{\WP}{WP}

\theoremstyle{plain}
\newtheorem{theorem}{Theorem}
\newtheorem{corollary}[theorem]{Corollary}
\newtheorem{proposition}[theorem]{Proposition}
\newtheorem{lemma}[theorem]{Lemma}

\newtheorem{question}[theorem]{Question}
\newtheorem{remark}[theorem]{Remark}

\newtheorem*{thm*}{Theorem}

\newcommand{\be}{\begin{equation}}
\newcommand{\ene}{\end{equation}}
\newcommand{\br}{\begin{remark}}
\newcommand{\er}{\end{remark}}
\newcommand{\bl}{\begin{lem}}
\newcommand{\el}{\end{lem}}
\newcommand{\bcor}{\begin{cor}}
\newcommand{\ecor}{\end{cor}}
\newcommand{\bpro}{\begin{pro}}
\newcommand{\epro}{\end{pro}}
\newcommand{\ben}{\begin{enumerate}}
\newcommand{\een}{\end{enumerate}}
\newcommand{\bp}{\begin{proof}}
\newcommand{\ep}{\end{proof}}
\newcommand{\bpo}{\begin{pro}}
\newcommand{\epo}{\end{pro}}
\newcommand{\beq}{\begin{equation*}}
\newcommand{\eeq}{\end{equation*}}
\newcommand{\bear}{\begin{eqnarray}}
\newcommand{\eear}{\end{eqnarray}}
\newcommand{\beqar}{\begin{eqnarray*}}
\newcommand{\eeqar}{\end{eqnarray*}}
\newcommand{\bt}{\begin{theorem}}
\newcommand{\et}{\end{theorem}}
\newcommand{\bex}{\begin{excer}}
\newcommand{\eex}{\end{excer}}

\theoremstyle{definition}

\newtheorem{definition}[theorem]{Definition}

\theoremstyle{remark}

\newtheorem*{rem*}{Remark}
\newtheorem*{def*}{Definition}
\newtheorem*{con*}{Construction}
\newtheorem*{definition*}{Definition}

\begin{document}

\title[Random surfaces]{Large genus asymptotics for lengths of separating closed geodesics on random surfaces}
\author{Xin Nie, Yunhui Wu, and Yuhao Xue}

\address{Shing-Tung Yau Center of Southeast University, Nanjing, China}
\email[(X.~N.)]{nie.hsin@gmail.com}

\address{Department of Mathematics and Yau Mathematical Sciences Center, Tsinghua University, Beijing, China}
\email[(Y.~W.)]{yunhui\_wu@tsinghua.edu.cn}
\email[(Y.~X.)]{xueyh18@mails.tsinghua.edu.cn}

\date{}
\maketitle

\begin{abstract}
 In this paper, we investigate basic geometric quantities of a random hyperbolic surface of genus $g$ with respect to the Weil-Petersson measure on the moduli space $\M_g$. We show that as $g$ goes to infinity, a generic surface $X\in \M_g$ satisfies asymptotically:
\begin{enumerate}
\item the separating systole of $X$ is approximately $2\log g$; 

\item there is a half-collar of width approximately $\frac{\log g}{2}$ around any separating systolic curve on $X$;

\item the length of the shortest separating closed multi-geodesics on $X$ is approximately $2\log g$.
\end{enumerate}
As applications, we also discuss the asymptotic behavior of the extremal separating systole, the non-simple systole and the expected length of the shortest separating closed multi-geodesics as $g$ goes to infinity.
\end{abstract}

\section{Introduction}

The overall behavior of geometric quantities such as systole, diameter, eigenvalues of Laplacian, Cheeger constant, \textit{etc}., for all closed hyperbolic surfaces of a given genus $g$, is a classical object of study.
While there are many results and conjectures about the maximal/minimal values of these quantities, as functions on the moduli space $\M_g$,
Mirzakhani initiated a new approach in \cite{Mirz13} to the subject: based on her celebrated thesis works \cite{Mirz07,Mirz07-int}, she obtained asymptotic results on certain statistical information about these quantities, viewed as random variables with respect to the probability measure $\Prob$ on $\M_g$ given by the Weil-Petersson metric. One may see the book \cite{Wolpert-book} of Wolpert and the recent survey \cite{Wright-tour} of Wright for more details.

\subsection{Separating systole} 
It was shown by Mirzakhani \cite[Corollary 4.3]{Mirz13} that if we consider the \emph{systole} 
$$
\ls(X):=\min\big\{\ell_\gamma(X)\,;\, \text{$\gamma\subset X$ is a simple closed geodesic}\big\}
$$
as a function on $\sM_g$, where the variable $X\in\sM_g$ is a closed hyperbolic surface of genus $g$ and
$\ell_\gamma(X)$ is the length of $\gamma$, then
the expected value of $\frac{1}{\ls\bracedcdot}$ is bounded from above and below by two positive constants independent of $g$. Meanwhile, \cite{Mirz13} also contains results on the \emph{separating systole}
$$
\lss(X):=\min\big\{\ell_\gamma(X)\,;\, \text{$\gamma\subset X$ is a separating simple closed geodesic}\big\},
$$
implying that $\lss$ behaves drastically differently from $\ls$. The function $\lss$ is unbounded on $\M_g$: if $X$ carries a pants decomposition consisting of arbitrarily short non-separating closed geodesics, the classical Collar Lemma (\eg see \cite{Kee74}) implies that the length of any separating closed geodesic is arbitrarily large because it has nonempty intersection with certain curves in the pants decomposition. In this paper we study the asymptotic behavior of the shortest simple separating closed geodesics as the genus $g$ goes to infinity. First we recall the following result. One may also see \cite[Theorem 4.2]{Mirz10} of Mirzakhani's 2010 ICM report for a weaker version.
\begin{thm*}[{\bf{Mirzakhani}, \cite[Theorem 4.4]{Mirz13}}]
Let $0<a<2$. Then
\[\Prob\left(X\in \M_g\,;\,  \lss(X)<a \log g \right)=O\left((\log g)^3 g^{\frac{a}{2}-1} \right).\]
\end{thm*}

\noindent This result in particular implies that for any $\epsilon>0$,
$$
\lim \limits_{g\to \infty} \Prob\left(X\in \M_g\,;\,  \lss(X)> (2-\epsilon)\log g \right)=1.
$$ 

Let $\omega:\{2,3,\cdots\}\to\R^{>0}$ be any function satisfying 
\be \label{eq-omega}
\lim \limits_{g\to \infty}\omega(g)= +\infty \ \text{and} \ \lim \limits_{g\to \infty}\frac{\omega(g)}{\log\log g} = 0.
\ene

The main part of this article is to show

\begin{theorem}\label{main}
Let $\omega(g)$ be a function satisfying \eqref{eq-omega}. Consider the following two conditions defined for all $X\in\M_g$:
\begin{itemize}
\item[(a).]\label{item_main1} $|\ell_{\sys}^{\rm sep}(X)-(2\log g - 4\log \log g)| \leq \omega(g)$;

\item[(b).]\label{item_main2} $\ell_{\sys}^{\rm sep}(X)$ is achieved by a simple closed geodesic separating $X$ into $S_{1,1}\cup S_{g-1,1}$.
\end{itemize}
Then we have
$$
\lim \limits_{g\to \infty} \Prob\left(X\in \M_g\,;\,  \textit{$X$ satisfies $(a)$ and $(b)$} \right)=1.
$$ 
\end{theorem}

\noindent The result in particular implies that for any $\epsilon>0$,
$$
\lim \limits_{g\to \infty} \Prob\left(X\in \M_g\,;\, (2-\epsilon)\log g< \lss(X)<2\log g \right)=1.
$$ 

\noindent Although $\ell_{\sys}^{\rm sep}$ is unbounded on $\sM_g$ as introduced above, the very recent joint work \cite{PWX20} of  H.\ Parlier with the second and third named authors of this paper shows that the expected value $\E[\lss]\sim 2\log g$ as $g\to \infty$ (\cf Section \ref{section questions}).
\begin{rem*}
We remark that the seemingly cumbersome upper and lower bounds $2\log g-4\log \log g\pm\omega(g)$ of $\ell_{\sys}^{\rm sep}(X)$ in the theorem above is related to the expected number of multi-geodesics of length less than $L$ on $X\in\M_g$ bounding a one-holed torus or a three-holed sphere, which is roughly $\frac{L^2e^\frac{L}{2}}{g}$. One may see the remark following Lemma \ref{E[N]} for more details. 
\end{rem*}

In the subsequent subsections we discuss applications of Theorem \ref{main} or the proof of Theorem \ref{main}.

\subsection{Long half collar and extremal length} A \emph{collar} of a simple closed geodesic $\gamma$ in a hyperbolic surface $X$ is an embedded symmetric hyperbolic cylinder centered at $\gamma$, bounded by two equidistant curves from $\gamma$, whereas a \emph{half-collar} of $\gamma$ is an embedded hyperbolic cylinder bounded by one equidistant curve along with $\gamma$ itself. 
For fixed $g$, if $X\in \M_g$ has a very long separating systolic curve $\gamma$, then the width of the maximal collar of $\gamma$ is very small, because the area of $X$ is fixed. On the other hand,  as $g$ goes to infinity,
as an application of Theorem \ref{main}, we show that in an asymptotic sense, for a generic point $X\in \M_g$, there is an arbitrarily long half-collar around a separating systolic curve. More precisely,
\begin{theorem}\label{thm:half collar}
Given $\epsilon>0$, consider condition $(b)$ from Theorem \ref{main} and the following condition defined for all $X\in\sM_g$:
\begin{itemize}
\item[(c).] There is a half-collar around $\gamma$ in the $S_{g-1,1}$-part of $X$ with width $\frac{1}{2}\log g-\left(\frac{3}{2}+\epsilon\right)\log\log g$.
\end{itemize}
Then we have
$$
\lim \limits_{g\to \infty} \Prob\left(X\in \M_g\,;\, \textit{$X$ satisfies $(b)$ and $(c)$} \right)=1.
$$ 
\end{theorem}

\noindent Note that one cannot replace ``half-collar'' by ``collar'' in the above theorem. In fact, since a geodesic $\gamma\subset X$ realizing $\lss(X)$ is arbitrarily long and bounds a one-holed torus for a generic point $X\in \M_g$ by Theorem \ref{main}, the maximal embedded half-collar about $\gamma$ in the one-holed torus must be arbitrarily thin because the area of a one-holed torus is equal to $2\pi$.

The theory of \emph{extremal length} was developed by Ahlfors and Beurling (\eg see \cite[Chapter 4]{Ahlfors-ci}). One may also see \cite[Section 3]{Kerck80} of Kerckhoff for its deep connection to the geometry of Teichm\"uller space. Here we deduce from Theorem \ref{thm:half collar} a consequence about extremal lengths of separating curves. Let $\exl_\gamma(X)$ denote the extremal length of the family of curves homotopic to $\gamma$ (see Subsection \ref{subsec-el} for the precise definition) and $\exlss(X)$ denote the \emph{separating extremal length systole} of $X$, defined as the infimum of $\exl_\gamma(X)$ over all separating simple closed geodesics $\gamma$ on $X$. It is known by Maskit \cite{Maskit} that $\ell_\gamma(X)\leq\pi\exl_\gamma(X)$, hence 
$\lss(X)\leq\pi\exlss(X)$. Conversely, \cite{Maskit} also provided an upper bound for $\exl_\gamma(X)$ (see Lemma \ref{lemma:extremal}), implying that hyperbolic and extremal lengths are comparable for short curves, but not for long ones in general.
As an application of Theorem \ref{thm:half collar}, we show:
\begin{theorem}\label{cor:extremal}
Given any $\epsilon>0$, we have
$$
\lim \limits_{g\to \infty} \Prob\left(X\in \M_g\,;\, \frac{\exlss(X)}{\lss(X)}< \frac{2+\epsilon}{\pi} \right)=1,
$$
as a consequence of this result and Theorem \ref{main}, 
$$
\lim \limits_{g\to \infty} \Prob\left(X\in \M_g\,;\, \frac{(2-\epsilon)}{\pi}\log g< \exlss(X)< \frac{(4+\epsilon)}{\pi}\log {\textit{g}} \right)=1.
$$
\end{theorem}

\begin{rem*}
For $X\in \sM_g$, the \emph{extremal length systole} $\exl_{sys}(X)$ of $X$ is defined as the infimum of $\exl_\gamma(X)$ over all simple closed geodesics $\gamma$ on $X$. For any systolic curve $\gamma\subset X$, it is known that the maximal collar of $\gamma$ has width bounded from below by a uniform positive constant independent of $g$ (see \eg \cite[Lemma 4.6]{Wu19} if the systole length $\geq 1$ and use the classical Collar Lemma if the systole length $< 1$). Then by Maskit \cite{Maskit} (or see Lemma \ref{lemma:extremal}) it is not hard to see that $\exl_{sys}(X)$ is uniformly comparable to $\ell_{sys}(X)$. Thus, as the genus $g$ goes to infinity, the asymptotic behavior of $\exl_{sys}\bracedcdot$ on $\M_g$ is similar as the behavior of $\ell_{sys}\bracedcdot$ on $\M_g$, which has already been studied by Mirzakhani \cite{Mirz13} and Mirzakhani-Petri \cite{MP19}. 
\end{rem*}
\subsection{Shortest non-simple closed geodesic} A shortest non-simple closed geodesic on a closed hyperbolic surface is always a figure eight geodesic (\eg see \cite[Theorem 4.24]{Buser10}), and has length at least $4\arcsinh(1)$ (\eg see \cite[4.2.2]{Buser10}). The \emph{non-simple systole} $\lns(X)$ of a hyperbolic surface $X$ is defined as  
$$
\lns(X):=\min\big\{\ell_\gamma(X)\,;\, \text{$\gamma\subset X$ is a non-simple closed geodesic}\big\}.
$$
As another application of Theorem \ref{main}, we show that as $g$ goes to infinity, asymptotically on a generic point $X\in \sM_g$ the non-simple systole behaves roughly like $\log g$. More precisely,
\begin{theorem}\label{cor:ns systole}
Given any $\epsilon>0$, then we have
$$
\lim \limits_{g\to \infty} \Prob\left(X\in \M_g\,;\, (1-\epsilon)\log {\textit{g}}< \lns(X)< 2\log {\textit{g}} \right)=1.
$$
\end{theorem}

\subsection{Shortest separating closed multi-geodesics} The union of disjoint non-separating simple closed curves may also separate a closed surface. The following geometric quantity was used by Schoen-Wolpert-Yau \cite{SWY80} to study the eigenvalues of the Laplacian operator on hyperbolic surfaces. 
\begin{definition}
For any $X\in \sM_g$, we define
$$
\sL_1(X):=\min\left\{\ell_\gamma(X)\,;\ \parbox[l]{5.5cm}{
	$\gamma=\gamma_1+\cdots+\gamma_k$ is a simple closed multi-geodesics separating $X$
}\right\}.
$$ 
\end{definition}
\noindent As a byproduct of the proof of Theorem \ref{main}, we show a similar result on $\sL_1\bracedcdot$ as follows.

\begin{theorem}\label{cor L1}
Let $\omega(g)$ be a function satisfying \eqref{eq-omega}. Consider the following two conditions defined for all $X\in\M_g$:
\begin{itemize}
\item[(e).]  $|\sL_1(X)-(2\log g - 4\log \log g)| \leq \omega(g)$;

\item[(f).]  $\sL_1(X)$ is achieved by either a simple closed geodesic separating $X$ into $S_{1,1}\cup S_{g-1,1}$ or three simple closed geodesics separating $X$ into $S_{0,3}\cup S_{g-2,3}$.
\end{itemize}
Then we have
$$
\lim \limits_{g\to \infty} \Prob\left(X\in \M_g\,;\, \textit{$X$ satisfies $(e)$ and $(f)$} \right)=1.
$$ 
\end{theorem}

Now we consider the expected value of $\mathcal{L}_1\bracedcdot$ over $\sM_g$. Unlike the unboundness of $\ell_{\sys}^{\rm sep}\bracedcdot$ on $\M_g$ we first show that $\sup_{X\in \sM_g}\mathcal{L}_1(X)\leq C \log g$ for some universal constant $C>0$ independent of $g$ (see Proposition \ref{L1-upp}). And then we apply Theorem \ref{cor L1} to show that 
\begin{theorem}\label{cor E[L1]}
The expected value $\E[\sL_1]$ of $\sL_1\bracedcdot$ on $\M_g$ satisfies
\begin{equation*}
\limg\frac{\E[\sL_1]}{\log g} = 2.
\end{equation*}
\end{theorem}

As another byproduct of the proof of Theorem \ref{main} we show the following useful property. First we make the following definition generalizing $\mathcal{L}_1\bracedcdot$.
\begin{definition}
For any integer $m\in [1,g-1]$ and $X\in \sM_g$, we define
\[\mathcal{L}_{1,m}(X):=\min_{\Gamma} \ell_{\Gamma}(X)\]
where the minimum runs over all simple closed multi-geodesics $\Gamma$ separating $X$ into $S_{g_1,k}\cup S_{g_2,k}$ with $|\chi(S_{g_1,k})|\geq |\chi(S_{g_2,k})|\geq m.$
\end{definition}

\noindent By definition we know that
\[\mathcal{L}_{1,1}(X)=\mathcal{L}_{1}(X)\]
and
\[\mathcal{L}_{1,m-1}(X)\leq \mathcal{L}_{1,m}(X), \quad \forall m \in [2,g-1].\]

\begin{proposition}\label{lower bound for chi geq 2}
Let $\omega(g)$ be a function satisfying \eqref{eq-omega}. Then we have that for any fixed $m\geq 1$ independent of $g$,
\begin{equation*}
\limg\Prob\left(X\in\M_g\,;\, \mathcal{L}_{1,m}(X) \geq  2m\log g - (6m-2)\log\log g -\omega(g)\right) = 1.
\end{equation*}
\end{proposition}
\noindent If $m=1$, this is part of Theorem \ref{cor L1}.

\begin{rem*}
As in \cite{Mirz13}, for all $1\leq m\leq g-1$ the \emph{$m$-th geometric Cheeger constant} $H_m(X)$ of $X$ is defined as
\[H_m(X):= \inf \limits_{\gamma}\frac{\ell_{\gamma}(X)}{2\pi m}\]
where $\gamma$ is a simple closed multi-geodesics on $X$ with $X\setminus \gamma=X_1\cup X_2$, and $X_1$ and $X_2$ are connected subsurfaces of $X$ such that $|\chi(X_1)|=m\leq |\chi(X_2)|$.

The \emph{geometric Cheeger constant} $H(X)$ of $X$ is defined as
\[H(X):=\min \limits_{1\leq m\leq g-1}H_m(X).\]

\noindent Mirzakhani in \cite{Mirz13} showed that  
\[\lim \limits_{g\to \infty}\Prob \left(X\in \M_g\,;\, H(X)> \frac{\log 2}{2\pi} \right)=1.\]

\noindent As a direct consequence of Theorem \ref{cor L1}, we obtain the following result on the asymptotic behavior of the first geometric Cheeger constant $H_1\bracedcdot$ on $\M_g$.
\begin{corollary}
For any $\epsilon>0$, we have
\[\lim \limits_{g\to \infty}\Prob \left(X\in \M_g\,;\, (1-\epsilon)\cdot \frac{\log g}{\pi}< H_1(X)< \frac{\log g}{\pi}\right)=1.\]
\end{corollary} 
\noindent It would be also \emph{interesting} to study the asymptotic behavior of $H_m\bracedcdot$ on $\M_g$ when $2\leq m \leq (g-1)$ as $g$ goes to infinity. One may see the last section for more discussions.
\end{rem*}

\noindent \emph{\bf Strategy on the proof of Theorem \ref{main}.} We conclude this introduction by a brief outline of the proof of Theorem \ref{main}. We divide the statements into two parts. An easier part is that
$\sL_1(X)\geq 2\log g-4\log\log g-\omega(g)$ with high probability, or equivalently, 
$$
\limg\Prob\big(X\in\M_g\,;\,\sL_1(X) \leq 2\log g - 4\log \log g - \omega(g)  \big) = 0.
$$
The proof uses the techniques of Mirzakhani in \cite[Section 4.3]{Mirz13}, and is based on expectation estimates of functions on $\M_g$ of the form $N_{g_0,n_0}(X,L)$, which counts simple closed multi-geodesics on $X$ with length at most $L$ bounding subsurfaces of type $S_{g_0,n_0}$. 

The second part, where the main novelty of this paper lies, is the assertion that with high probability, a random $X\in\M_g$ contains a simple closed geodesic $\gamma$ with length at most $L(g):=2\log g-4\log\log g+\omega(g)$ bounding a one-holed torus, or equivalently,
$$
\lim_{g\to\infty}\Prob\big(X\in\M_g\,;\, N_{1,1}(X,L(g))=0\big)=0.
$$
We consider instead the function $N_{1,1}^*(X,L(g))$, which counts those $\gamma$'s whose intersection with any other $\gamma$ is ``simple'' in a certain sense, and show the stronger statement that the above limit holds if $N_{1,1}$ is replaced by $N^*_{1,1}$. To achieve this, first Chebyshev's Inequality (see Equation \eqref{prob(N*=0) leq 3 parts}) tells that one may bound it from above by the following terms where $Z^*(X,L)$ has certain intersections involved: 
\begin{eqnarray}
&&\Prob\big(N^*_{1,1}(X,L)=0 \big)\leq \frac{1}{\E[N^*_{1,1}(X,L)]} \nonumber\\
&& + \frac{\E[Y^*(X,L)] - \E[N^*_{1,1}(X,L)]^2}{\E[N^*_{1,1}(X,L)]^2}+ \frac{\E[Z^*(X,L)]}{\E[N^*_{1,1}(X,L)]^2}.\nonumber
\end{eqnarray}
As $g\to \infty$, we combine the ideas in the aforementioned works \cite{Mirz07, Mirz13, MZ15} to show that the first two terms of the \rm{RHS} above converge to $0$. The most delicate part of the proof, inspired by the work \cite{MP19}, is to show the third term of the \rm{RHS} above also converges to $0$ as $g\to \infty$: we estimate the expected number of intersecting pairs of $\gamma$'s by resolving the intersections, then apply a new usage of Mirzakhani's generalized McShane identity \cite{Mirz07} on $4$-holed spheres and $2$-holed tori to control the multiplicity occurring in the resolution procedure. 

\subsection*{Plan of the paper.} In Sections \ref{section preliminaries}, \ref{section union}, \ref{section wp volume} and \ref{section McShane identity}, we review the backgrounds, introduce some notations, and prove a few technical lemmas. We then prove the lower bound part and the upper bound part of Theorem \ref{main} and \ref{cor L1} in Sections \ref{section lower bound} and \ref{section upper bound}, respectively. In Section \ref{sec:half collar}, we prove Theorem \ref{thm:half collar} and \ref{cor:extremal}. Theorem \ref{cor:ns systole} and \ref{cor E[L1]} will be proved in Section \ref{section exp-L1}. We will pose several advanced questions in Section \ref{section questions}.

\subsection*{Acknowledgements.}
The authors would like to thank Jeffrey Brock, Curtis McMullen, and Michael Wolf for their interest and comments on this paper. The second  named author is partially supported by the NSFC grant No. $12171263$. The authors also would like to thank the anonymous referees for their careful reading and valuable comments which improve this article.

\setcounter{tocdepth}{1}
\tableofcontents


\section{Preliminaries}\label{section preliminaries}
In this section, we set our notations and review the relevant background material about moduli spaces of Riemann surfaces, \wep \ metric and Mirzakhani's Integration Formula.

\subsection{Weil-Petersson metric.} \label{sec:wp background}
We denote by $S_{g,n}$ an oriented surface of genus $g$ with $n$ punctures such that $2g+n\geq 3$, and let $\sT_{g,n}$ denote the Teichm\"uller space of $S_{g,n}$, formed by equivalence classes of complete hyperbolic surfaces of finite area marked by $S_{g,n}$. The tangent space $T_X\sT_{g,n}$ at a point $X\cong(S_{g,n},\sigma(z)|dz|^2)$ is identified with the space of finite area {\it harmonic Beltrami differentials} on $X$, i.e. forms on $X$ expressible as
$\mu=\overline{\psi}/\sigma$ where $\psi \in Q(X)$ is a holomorphic quadratic differential on $X$. Let $z=x+iy$ and $dA=\sigma(z)dxdy$ be the volume form. The \textit{Weil-Petersson metric} is the Hermitian
metric on $\sT_{g,n}$ arising from the \textit{Petersson scalar  product}
\begin{equation}
 \left<\varphi,\psi \right>= \int_X \frac{\varphi \cdot \overline{\psi}}{\sigma^2}dA\nonumber
\end{equation}
via duality. We will concern ourselves primarily with its Riemannian part $g_{\mathrm{WP}}$. Throughout this paper we denote by $\Teich(S_{g,n})$ the Teichm\"uller space endowed with the Weil-Petersson metric. By definition it is easy to see that the mapping class group $\Mod_{g,n}:=\Diff^+(S_{g,n})/\Diff^0(S_{g,n})$ acts on $\Teich(S_{g,n})$ as isometries. Thus, the \wep \ metric descends to a metric, also called the \wep \ metric, on the moduli space of Riemann surfaces $\sM_{g,n}$ which is defined as $\sT_{g,n}/\Mod_{g,n}$. Throughout this paper we also denote by $\sM_{g,n}$ the moduli space endowed with the Weil-Petersson metric and write $\sM_g = \sM_{g,0}$ for simplicity. Given $\ve{L}=(L_1,\cdots, L_n)\in\R_{\geq0}^n$, the weighted Teichm\"uller space $\T_{g,n}(\ve{L})$ parametrizes hyperbolic surfaces $X$ marked by $S_{g,n}$ such that for each $i=1,\cdots,n$,
\begin{itemize}
	\item if $L_i=0$, the $i^\text{th}$ puncture of  $X$ is a cusp;
	\item if $L_i>0$, one can attach a circle to the $i^\text{th}$ puncture of $X$ to form a geodesic boundary loop of length $L_i$. 
\end{itemize}
The weighted moduli space $\M_{g,n}(\ve{L}):=\T_{g,n}(\ve{L})/\Mod_{g,n}$ then parametrizes unmarked such surface.
The Weil-Petersson volume form is also well-defined on $\M_{g,n}(\ve{L})$ and its total volume, denoted by $V_{g,n}(\ve{L})$, is finite.

\subsection{The Fenchel-Nielsen coordinates} Recall that for any surface $S_{g,n}$, a \emph{pants decomposition} $\mathcal{P}$ of $S_{g,n}$ is a set of $(3g+n-3)$ disjoint simple closed curves $\{\alpha_i\}_{i=1}^{3g+n-3}$ such that the complement $S_{g,n}\setminus \cup_{i=1}^{3g+n-3}\alpha_i$ of $S_{g,n}$ consists of disjoint union of three-holed spheres. For each $\alpha_i \in \mathcal{P}$, there are two natural real positive functions on $\sT_{g,n}$: the geodesic length function $\ell_{\alpha_i}(X)$ and the twist function $\tau_{\alpha_i}(X)$ along $\alpha_i$. Associated to $\mathcal{P}$, the \emph{Fenchel-Nielsen coordinates}, given by $X \mapsto (\ell_{\alpha_i}(X),\tau_{\alpha_i}(X))_{i=1}^{3g+n-3}$, is a global coordinate for $\sT_{g,n}$. Wolpert in \cite{Wolpert82}
showed that the \wep \ sympletic structure has a simple form in Fenchel-Nielsen coordinates:
\bt[Wolpert]\label{wol-wp}
The \wep \ sympletic form $\omega_{\WP}$ on $\sT_{g,n}$ is given by
\[\omega_{\WP}=\sum_{i=1}^{3g+n-3}d\ell_{\alpha_i}\wedge d\tau_{\alpha_i}.\]
\et

In the sequel, we mainly work with the \emph{Weil-Petersson volume form}
$$
\dvol_{\WP}:=\tfrac{1}{(3g+n-3)!}\underbrace{\omega_{\WP}\wedge\cdots\wedge\omega_{\WP}}_{\text{$3g+n-3$ copies}}~.
$$
It is a $\Mod_{g,n}$-invariant measure on $\T_{g,n}$, hence is the lift of a measure on $\M_{g,n}$, which we still denote by $\dvol_{\WP}$. The total volume of $\M_{g,n}$ is known to be finite and is denoted by $V_{g,n}$.

Our main objects of study are geometric quantities on $\M_g$. Following \cite{Mirz13}, we view such a quantity $f:\M_g\to\R$ as a random variable on $\M_g$ with respect to the probability measure $\Prob$ defined by normalizing $\dvol_{\WP}$, and let $\E[f]$ denote the expectation. Namely,
$$
\Prob(\mathcal{A}):=\frac{1}{V_g}\int_{\M_g}\mathbf{1}_{\mathcal{A}}dX,\quad \E[f]:=\frac{1}{V_g}\int_{\M_g}f(X)dX,
$$
where $\mathcal{A}\subset\M_g$ is any Borel subset, $\mathbf{1}_\mathcal{A}:\M_g\to\{0,1\}$ is its characteristic function, and we always write $\dvol_{\WP}(X)$ as $dX$ for short. In this paper, we view certain geometric quantities as random variables on $\sM_g$, and study their asymptotic behaviors as $g\to \infty$. One may also see \cite{DGZZ22, GMST19, GPY11, MT20, MP19} for related interesting topics.

\subsection{Mirzakhani's Integration Formula}

In \cite{Mirz07}, Mirzakhani gave a formula to integrate geometric functions over moduli spaces, which is essential in the study of random surfaces with respect to Weil-Petersson metric. Then in the same paper she calculated the volume of moduli spaces together with her generalized McShane identity. In \cite{Mirz13}, applying this formula, she gave many estimations for some geometry variables in probability meaning. Here we give the version stated in \cite{Mirz13}, which is a little more general than the one in \cite{Mirz07}.

Given a homotopy class $\gamma$  of closed curves on a topological surface $S_{g,n}$ and $X\in\T_{g,n}$, we denote by $\ell_\gamma(X)$ the hyperbolic length of the unique closed geodesic in the homotopy class $\gamma$ on $X$. We also write $\ell(\gamma)$ for simplicity if the surface $X$ is clear from the context. Let $\Gamma=(\gamma_1,\cdots,\gamma_k)$ be an ordered k-tuple of disjoint homotopy classes of nontrivial, non-peripheral, simple closed curves on $S_{g,n}$. Denote the orbit of $\Gamma$ under the $\Mod_{g,n}$-action by
\begin{equation*}
\mathcal O_{\Gamma} = \{(h\cdot\gamma_1,\cdots,h\cdot\gamma_k) \,;\, h\in\Mod\nolimits_{g,n}\}.
\end{equation*}
Given a function $F:\R^k_{\geq0} \rightarrow \R$, we consider the function on $\M_{g,n}$ given by
\begin{eqnarray*}
F^\Gamma:\M_{g,n} &\rightarrow& \R \\
X &\mapsto& \sum_{(\alpha_1,\cdots,\alpha_k)\in \mathcal O_\Gamma} F(\ell_{\alpha_1}(X),\cdots,\ell_{\alpha_k}(X))
\end{eqnarray*}
provided that the sum converges.
\begin{rem*}
Although $\ell_\gamma\bracedcdot$ is only defined on $\T_{g,n}$, the function $F^\Gamma\bracedcdot$ is well-defined on $\M_{g,n}$.
\end{rem*}

Assume $S_{g,n}-\cup\gamma_j = \cup_{i=1}^s S_{g_i,n_i}$. For any given $\boldsymbol{x}=(x_1,\cdots,x_k)\in \R^k_{\geq0}$, we consider the moduli space $\M(S_{g,n}(\Gamma); \ell_{\Gamma}=\boldsymbol{x})$ of (possibly disconnected) hyperbolic Riemann surfaces homeomorphic to $S_{g,n}-\cup\gamma_j$ with $\ell_{\gamma_i^1} = \ell_{\gamma_i^2} =x_i$ for $i=1,\cdots,k$, where $\gamma_i^1$ and $\gamma_i^2$ are the two boundary components of $S_{g,n}-\cup\gamma_j$ given by $\gamma_i$. We consider the volume
\begin{equation*}
V_{g,n}(\Gamma,\boldsymbol{x}) = \Vol\nolimits_{\rm WP}\big(\M(S_{g,n}(\Gamma) \,;\, \ell_{\Gamma}=\boldsymbol{x})\big).
\end{equation*}
In general
\begin{equation*}
V_{g,n}(\Gamma,\boldsymbol{x}) = \prod_{i=1}^s V_{g_i,n_i}(\boldsymbol{x}^{(i)})
\end{equation*}
where $\boldsymbol{x}^{(i)}$ is the list of those coordinates $x_j$ of $\boldsymbol{x}$ such that $\gamma_j$ is a boundary component of $S_{g_i,n_i}$. And $V_{g_i,n_i}(\boldsymbol{x}^{(i)})$ is the Weil-Petersson volume of the moduli space $\M_{g_i,n_i}(\boldsymbol{x}^{(i)})$. Mirzakhani used Theorem \ref{wol-wp} of Wolpert to get the following integration formula. One may see \cite[Theorem 7.1]{Mirz07} or \cite[Theorem 2.2]{MP19} or \cite[Theorem 4.1]{Wright-tour}.

\begin{theorem}\label{Mirz int formula}
For any $\Gamma=(\gamma_1,\cdots,\gamma_k)$, the integral of $F^\Gamma$ over $\M_{g,n}$ with respect to Weil-Petersson metric is given by
\begin{equation*}
\int_{\M_{g,n}} F^\Gamma(X)dX =
C_\Gamma\int_{\R^k_{\geq0}} F(x_1,\cdots,x_k)V_{g,n}(\Gamma,\boldsymbol{x}) \boldsymbol{x}\cdot d\boldsymbol{x}
\end{equation*}
where $\boldsymbol{x}\cdot d\boldsymbol{x} = x_1\cdots x_k dx_1\wedge\cdots\wedge dx_k$ and the constant $C_\Gamma \in (0,1]$ only depends on $\Gamma$. Moreover, $C_\Gamma=2^{-k}$ if each $\gamma_i$ in $\Gamma$ separates off a one-holed torus. 
\end{theorem}

\begin{rem*}
An explicit expression for $C_\Gamma$ is provided by Wright in \cite[Section 4]{Wright-tour}. 
\end{rem*}

\begin{rem*}
Given an unordered multi-curve $\gamma=\sum_{i=1}^k c_i \gamma_i$ where $\gamma_i's$ are distinct disjoint homotopy classes of nontrivial, non-peripheral, simple closed curves on $S_{g,n}$, when $F$ is a symmetric function, we can define
\begin{eqnarray*}
F_\gamma:\M_{g,n} &\rightarrow& \R \\
X &\mapsto& \sum_{\sum_{i=1}^k c_i\alpha_i \in \Mod_{g,n}\cdot \gamma} F(c_1\ell_{\alpha_1}(X),\cdots,c_k\ell_{\alpha_k}(X)).
\end{eqnarray*}

\noindent It is easy to check that
\begin{equation*}
F^\Gamma(X) = |\Sym(\gamma)| \cdot F_\gamma(X)
\end{equation*}
where $\Gamma=(c_1\gamma_1,\cdots,c_k\gamma_k)$ and $\Sym(\gamma)$ is the symmetry group of $\gamma$ defined by
\begin{equation*}
\Sym(\gamma) = \mathop{\rm Stab}(\gamma) / \cap_{i=1}^k \mathop{\rm Stab}(\gamma_i).
\end{equation*}

\noindent Actually we consider the integration of $F_\gamma$ for most times in this paper.
\end{rem*}

\subsection{Counting functions}
In this subsection we introduce some notations that will be used in the whole paper here.

On a topological surface $S_{g,n}$ with $\chi(S_{g,n})=2-2g-n<0$, let $\gamma=\sum_{i=1}^k  \gamma_i$ be a simple closed multi-curves where $\gamma_i's$ are disjoint homotopy classes of nontrivial, non-peripheral, simple closed curves on $S_{g,n}$. For any $X\in \T_{g,n}$, we put
\begin{equation*}
\ell_\gamma(X) := \sum_{i=1}^k  \ell_{\gamma_i}(X).
\end{equation*}
We sometimes write $\ell_\gamma(X)$ as $\ell(\gamma)$ if the surface $X$ is clear from the context.

Consider the orbit of $\gamma$ under the mapping class group $\Mod_{g,n}$ action, which we denote by
\begin{equation*}
\mathcal O_{\gamma} = \{h\cdot\gamma \,;\,  h\in\Mod\nolimits_{g,n}\}
\end{equation*}
where $h\cdot \gamma = h\cdot \sum_{i=1}^k\gamma_i = \sum_{i=1}^k h\cdot\gamma_i$.

For any $X\in\T_{g,n}$ and $L>0$, we can define the counting function
\begin{equation*}
N_\gamma(X,L) := \# \{\alpha\in \mathcal O_{\gamma}\,;\,  \ell_{\alpha}(X)\leq L \}.
\end{equation*}

Moreover, although $\ell_\gamma\bracedcdot$ is only defined for $\T_{g,n}$, the counting function $N_\gamma(\cdot,L)$ is well-defined on $\M_{g,n}$.

Note that the orbit $\mathcal{O}_\gamma$ of a simple closed multi-curve $\gamma$ is determined by the topology of $S_{g,n}-\gamma$. We also use the following notations for some special types of multi-curves $\gamma$.

When $\alpha$ consists of $n_0$ simple closed curves separating $S_g$ into $S_{g_0,n_0} \cup S_{g-g_0-n_0+1,n_0}$ (\eg see Figure \ref{figure:counting} for the case that $n_0=1$ and $g_0=1$), we write
\begin{equation*}
N_{g_0,n_0}(X,L) := N_\alpha(X,L).
\end{equation*}

When $\gamma$ consists of $n_0$ simple closed curves separating $S_g$ into $q+1$ pieces $S_{g_0,n_0} \cup S_{g_1,n_1} \cup S_{g_2,n_2} \cup \cdots \cup S_{g_q,n_q}$ with $n_1+\cdots+n_q=n_0$ and $g_0+g_1+\cdots+g_q+n_0-q=g$ (\eg see Figure \ref{figure:counting}), we write
\begin{equation*}
N_{g_0,n_0}^{(g_1,n_1),\cdots,(g_q,n_q)}(X,L) := N_\gamma(X,L).
\end{equation*}
In particularly, $N_{g_0,n_0}(X,L)= N_{g_0,n_0}^{(g-g_0-n_0+1,n_0)}(X,L)$.

\begin{figure}[h]
	\centering	
	\includegraphics[width=8.2cm]{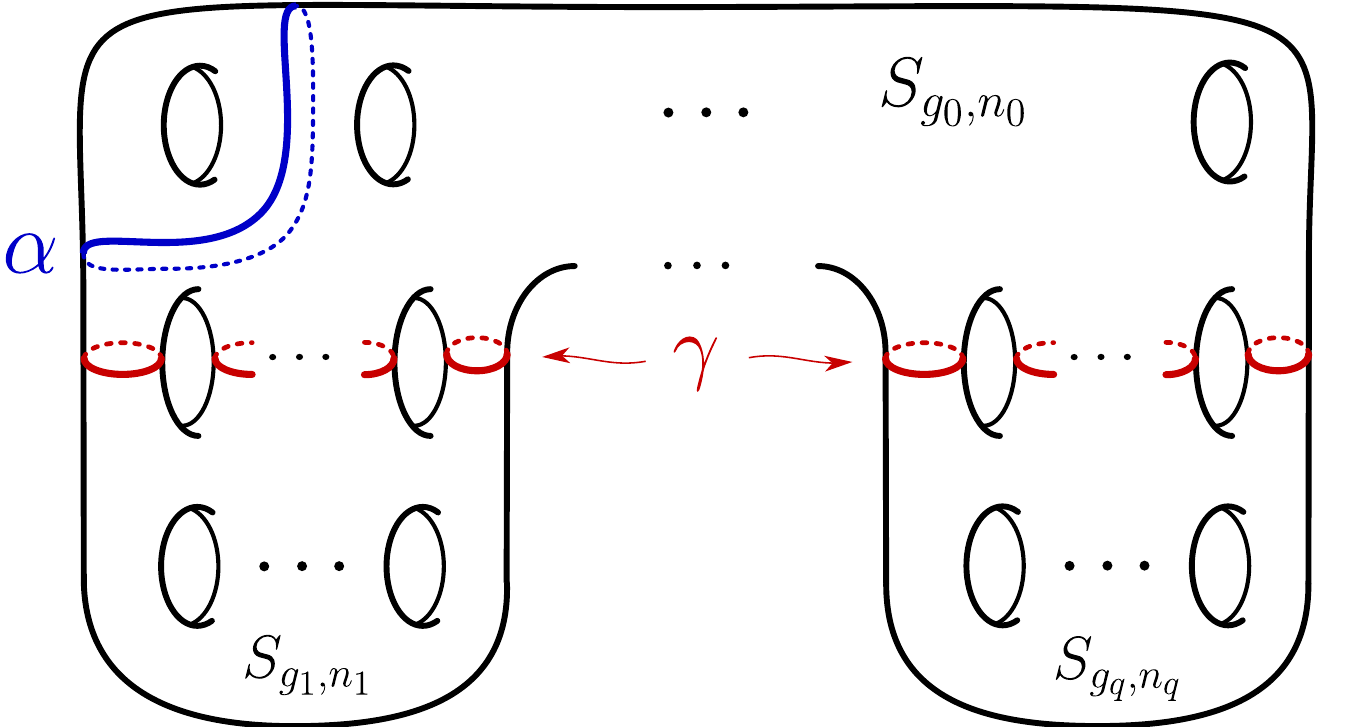}
	\caption{}
	\label{figure:counting}
\end{figure}

We will also use some other less common counting functions in this paper and will introduce them when required.

\section{Union of two subsurfaces with geodesic boundaries}\label{section union}

In this section, we present some hyperbolic-geometric constructions and lemmas used in Section \ref{section upper bound} below.

\begin{con*}
Fix a closed hyperbolic surface $X\in\M_g$ and let $X_1, X_2$ be two distinct connected, compact subsurfaces of $X$ with geodesic boundaries, such that $X_1\cap X_2\neq\emptyset$ and neither of them contains the other. Then the union $X_1\cup X_2$ is a subsurface whose boundary is only piecewise geodesic. We can construct from it a new subsurface, with geodesic boundary, by deforming each of its boundary components $\xi\subset\partial(X_1\cup X_2)$ as follows: 
\begin{itemize}
	\item if $\xi$ is homotopically both nontrivial and distinct from any other component of $\partial(X_1\cup X_2)$, we deform $X_1\cup X_2$ by shrinking $\xi$ to the unique simple closed geodesic homotopic to it;
	\item if $\xi$ is homotopically trivial, we fill into $X_1\cup X_2$ the disc bounded by $\xi$;
	\item if $\xi$ is homotopic to some other component $\xi'$ of $\partial(X_1\cup X_2)$, we fill into $X_1\cup X_2$ the annulus bounded by $\xi$ and $\xi'$.
\end{itemize}
Denoted the resulting compact subsurface with geodesic boundary by $X_3$.
\end{con*}

By construction, it is clear that
$$
\ell(\partial X_3) \leq \ell(\partial X_1) + \ell(\partial X_2).
$$

We will mainly apply this construction to the situation where $X_1$ and $X_2$ are both one-holed torus (that is, of type $S_{1,1}$). We introduce the following notation for this case:
\begin{definition}\label{def U}
	Suppose $X\in\M_g$. For a simple closed geodesic $\alpha\subset X$ bounding a one-holed torus, let $X_\alpha$ denote the one-holed torus bounded by $\alpha$. For two such geodesics $\alpha,\beta$ with $\alpha\neq\beta$, $\alpha\cap\beta\neq\emptyset$, let $X_{\alpha\beta}$ denote the subsurface $X_3$ of $X$ constructed above for $X_1=X_\alpha$ and $X_2=X_\beta$. See Figure \ref{figure_examples}.
\end{definition}
\begin{figure}[ht]
	\centering
	\includegraphics[width=11cm]{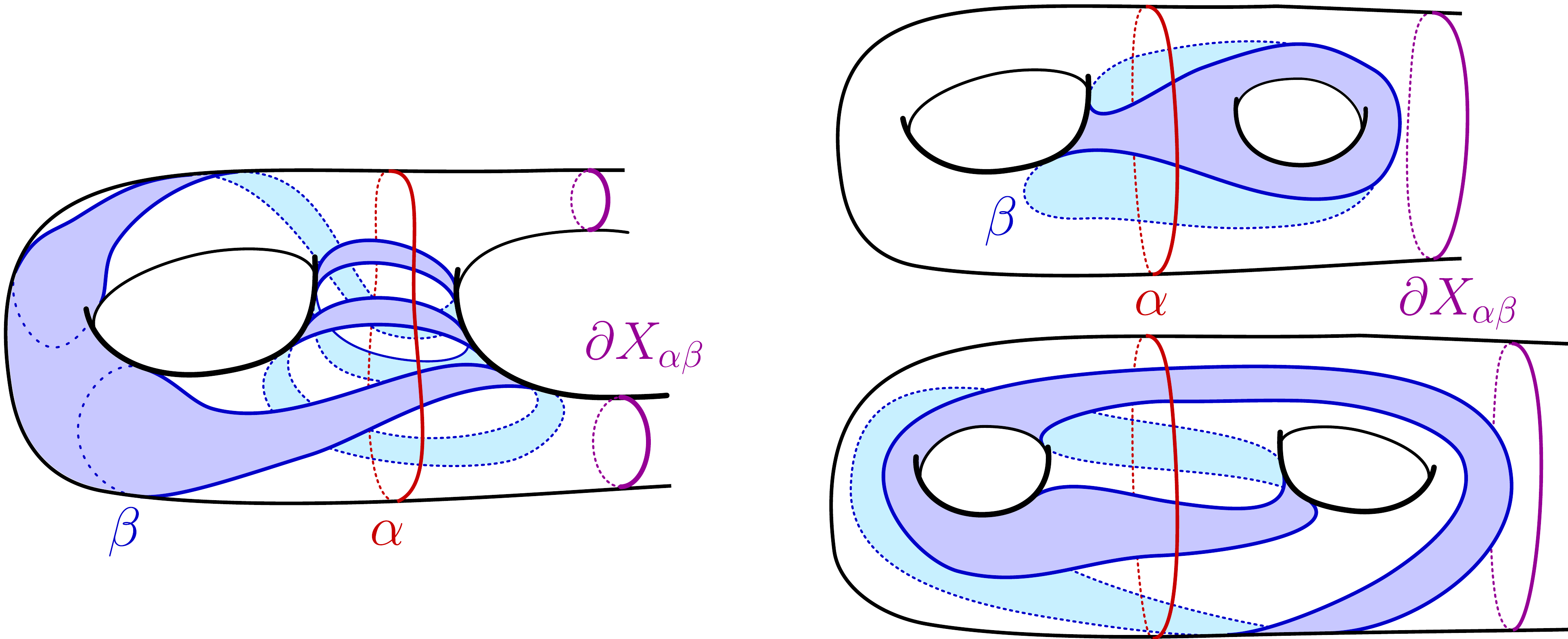}
	\caption{Examples of $(\alpha,\beta)$ and $X_{\alpha\beta}$. The one-holed torus $X_\beta$ is colored. $X_{\alpha\beta}$ is of type $S_{1,2}$ in the first example and of type $S_{2,1}$ in both examples on the right.}
	\label{figure_examples}
\end{figure}
\begin{rem*}
The first example in Figure \ref{figure_examples} illustrates the case where $\beta$ is obtained from $\alpha$ by $n$-times Dehn twist along another simple closed curve. In this case, $X_{\alpha\beta}$ is always of type $S_{1,2}$. Note that $X_\beta\setminus X_\alpha$ is a disjoint union of strips homotopic to each other in this case. So one can construct a pair $(\alpha,\beta)$ with $|\chi(X_{\alpha\beta})|$ arbitrarily large by modifying these strips, making them not homotopic.
\end{rem*}

We now return to the general case and establish a basic property for $X_3$:
\begin{lemma}\label{alpha sbs U}
	Let $X_1$, $X_2$ and $X_3$ be as above. Then we have
	\begin{equation*}
	X_1 \cup X_2 \sbs X_3,
	\end{equation*}
	and the complement $X_3 \setminus (X_1 \cup X_2)$ is a disjoint union of topological discs and cylinders.
\end{lemma}
\begin{proof}
	We begin with the observation that $X_0:=X_1\cup X_2$ is a subsurface of $X$ with \emph{concave} piecewise geodesic boundary, where the concavity means that for each junction point $p\in\partial X_0$ of two geodesic pieces of $\partial X_0$, the inner angle $\angle_pX_0$ of $X_0$ at $p$ is greater than $\pi$ (see Figure \ref{figure:crown1}). This is because $\angle_pX_0$ is formed by overlapping the two $\pi$-angles given by $X_1$ and $X_2$ at $p$.
	
	\begin{figure}[ht]
		\centering
		\includegraphics[width=4cm]{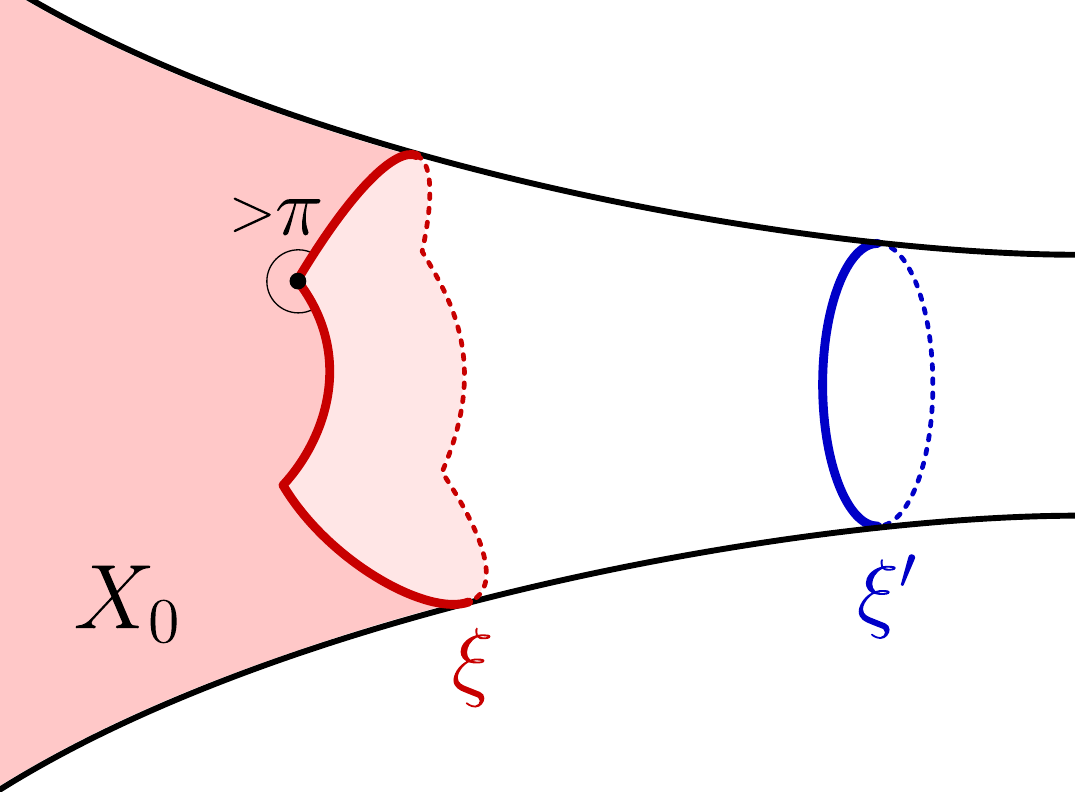}
		\caption{}
		\label{figure:crown1}
	\end{figure}
	
	By the construction of $X_3$, in order to prove the required statements, we only need to show that if $\xi$ is a component of $\partial X_0$ which is homotopically nontrivial and consists of at least two geodesic pieces, then $\xi$ and the simple closed geodesic $\xi'$ homotopic to $\xi$ together bound an annulus outside of $X_0$, as Figure \ref{figure:crown1} shows. 
	
	Suppose by contradiction that $\xi$ violates this property. Then we are in one of the following cases: 
	
	\textbf{Case 1. $\xi'$ is contained in $X_0\setminus\xi$} (see Figure \ref{figure:crown2}).
	\begin{figure}[ht]
		\centering
		\includegraphics[width=4cm]{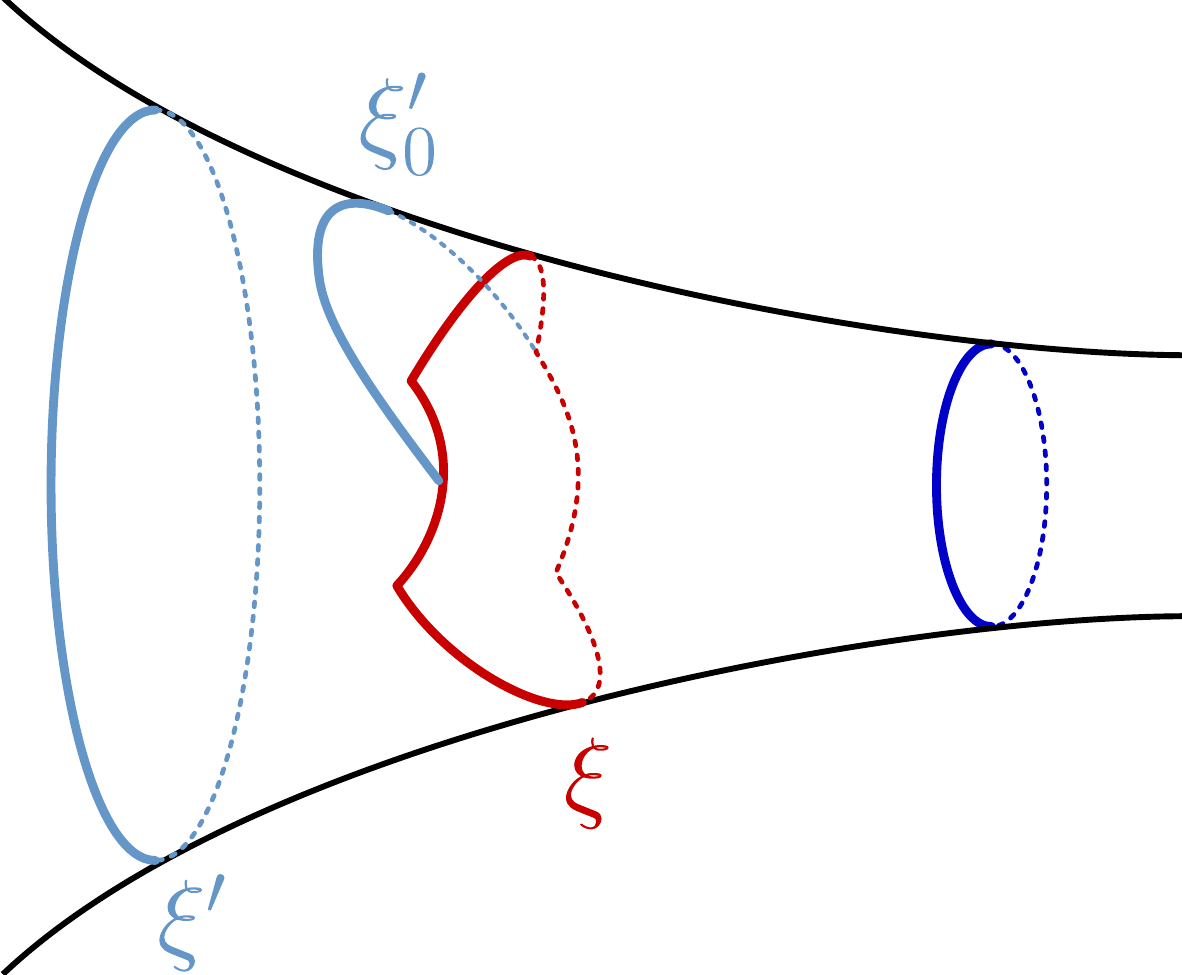}
		\caption{}
		\label{figure:crown2}
	\end{figure}
	Applying the Gauss-Bonnet formula to the annulus $A$ bounded by $\xi$ and $\xi'$ in this case, we get
	$$
	-\area(A)+\sum_{p\in J(\xi)}(\pi-\angle_pX_0)=2\pi\chi(A)=0,
	$$
	where $J(\xi)$ denote the set of junction points of the geodesic pieces of $\xi$. This is a contradiction because the LHS is negative.
	
	\textbf{Case 2.} Otherwise, we have $\xi'\cap\xi\neq\emptyset$. In this case, $\xi'$ contains an arc $\xi'_0$ in $X_0\setminus\xi$ joining two points $q_1,q_2$ of $\xi$ (see Figure \ref{figure:crown2}). These two points separate $\xi$ into two arcs and one of them, denoted by $\xi_0$, bounds a disc $D$ together with $\xi_0'$ because $\xi'$ is homotopic to $\xi$. Applying Gauss-Bonnet to $D$, we get
	$$
	-\area(D)+(\pi-\angle_{q_1}D)+(\pi-\angle_{q_2}D)+\sum_{p\in J(\xi_0)}(\pi-\angle_pX_0)=2\pi\chi(D)=2\pi,
	$$
	which also leads to a contradiction. 
\end{proof}

We proceed to give bounds on the Euler characteristic of $X_3$: 
\begin{lemma}\label{lemma chiU12}
	Let $X_1$, $X_2$ and $X_3$ be as above. Then we have
	\begin{equation*}
	|\chi(X_3)| \geq 1+\max\{|\chi(X_1)|,|\chi(X_2)|\}
	\end{equation*}
and
	\begin{equation*}
	|\chi(X_3)| \leq |\chi(X_1)| + |\chi(X_2)|+\frac{\ell(\partial X_1)+\ell(\partial X_2)}{2\pi}.
	\end{equation*}
\end{lemma}

\begin{proof}
	By Gauss-Bonnet formula and the assumption that neither $X_1$ nor $X_2$ contains the other, we have
	\begin{align*}
	|\chi(X_3)|&=\frac{1}{2\pi}\area(X_3) \\
	&>\frac{1}{2\pi}\max\{\area(X_1), \area(X_2) \} \\
	&=\max\{|\chi(X_1)|,|\chi(X_2)|\},
	\end{align*}
	which is equivalent to the required lower bound of $|\chi(X_3)|$ because Euler characteristics are integers.
	
	To prove the upper bound, let $\xi_1,\cdots,\xi_r$ be the boundary components of $X_1\cup X_2$ which are piecewise geodesics with at least two pieces. Let $I$ denote the set of indices $i\in\{1,\cdots,r\}$ such that $\xi_i$ is homotopically trivial, $J$ denote the set of indices $j\in\{1,\cdots,r\}$ such that $\xi_j$ is homotopic to a component of $\partial X_3$, and $K$ denote the set of indices $k\in\{1,\cdots,r\}$ such that $\xi_k$ is homotopic to a geodesic in the interior of $X_3$.
	
	By Lemma \ref{alpha sbs U},  $X_3 \setminus (X_1 \cup X_2)$ is a disjoint union of topological discs $\{D_i\}_{i\in I}$, cylinders $\{C_j\}_{j\in J}$ and cylinders $\{C'_{p}\}_{p\in P}$, where $\partial D_i$ is exactly $\xi_i$, $\partial C_j$ is the union of $\xi_j$ and some boundary component of $X_3$, and $\partial C'_{p}$ is the union of two elements $\xi_{k_p^1}$ and $\xi_{k_p^2}$ of $\{\xi_k\}_{k\in K}$. Each element of $\{\xi_1,\cdots,\xi_r\}$ appears in $\{\partial D_i\}_{i\in I}$ or $\{\partial C_j\}_{j\in J}$ or $\{\partial C'_{p}\}_{p\in P}$ exactly once.
	
	By Isoperimetric Inequality for topological discs and cylinders on hyperbolic surfaces (\eg see \cite{Buser10} or \cite{WX18}), we have
	\begin{equation*}
	\area(D_i) \leq \ell(\partial D_i)=\ell(\xi_i),
	\end{equation*}
	\begin{equation*}
	2\area(C_j) = \area(2C_j)\leq \ell(\partial (2C_j)) = 2\ell(\xi_j),
	\end{equation*}
	\begin{equation*}
	\area(C'_p) \leq \ell(\partial C'_p)=\ell(\xi_{k_p^1})+\ell(\xi_{k_p^2}),
	\end{equation*}
	where $2C_i$ denote the double of $C_i$ along its geodesic boundary component in $\partial X_3$.
	Therefore,
	\begin{align*}
	\area(X_3)
	&= \area(X_1 \cup X_2) +\sum_{i\in I}\area(D_i)+\sum_{j\in J}\area(C_j)+\sum_{p\in P}\area(C'_p) \\
	&\leq  \area(X_1 \cup X_2) +\ell(\xi_1)+\cdots+\ell(\xi_r) \\
	&\leq \area(X_1) + \area(X_2) + \ell(\partial (X_1 \cup X_2))\\
	&\leq \area(X_1) + \area(X_2) + \ell(\partial X_1)+\ell(\partial X_2).
	\end{align*}
	This gives the required upper bound of $|\chi(X_3)|$ again by Gauss-Bonnet.
\end{proof}

In the case where $X_1$ and $X_2$ are one-holed torus, Lemma \ref{lemma chiU12} implies:
\begin{lemma}\label{area U small}
	On $X\in\M_g$, let $\alpha,\beta$ be two simple closed geodesics bounding one-holed torus with $\ell(\alpha)\leq L, \ell(\beta)\leq L$ and $\alpha\neq\beta$, $\alpha\cap\beta\neq\emptyset$. Then we have
	\begin{enumerate}[label=(\arabic*)]
		\item\label{item:U1} The genus of $X_{\alpha\beta}$ is at least $1$, and the Euler characteristic $\chi(X_{\alpha\beta})$ satisfies 
		\[2 \leq |\chi(X_{\alpha\beta})| \leq \frac{1}{\pi}L+2.\]
		\item\label{item:U2}
		If $|\chi(X_{\alpha\beta})|=2$ and $g\geq3$, then $X_{\alpha\beta}$ is of type $S_{1,2}$.
	\end{enumerate}
\end{lemma}
\begin{proof}
	Statement \ref{item:U1} follows from Lemma \ref{lemma chiU12}. Statement \ref{item:U2} is because the only surfaces $S_{g_0,n_0}$ such that $|\chi(S_{g_0,n_0})|=2$ and $g_0\geq1$ are $S_{1,2}$ and $S_{2,0}$, whereas $X$ cannot have a subsurface of type $S_{2,0}$ if $g\geq3$.
\end{proof}

Finally, we show that in the case where $X_{\alpha\beta}$ is of type $S_{1,2}$, under some additional assumptions, one can reduce $X_{\alpha\beta}$ to a $4$-holed sphere:
\begin{lemma}\label{lemma:4holded}
	On $X\in\M_g$, let $\alpha,\beta$ be two simple closed geodesics bounding one-holed torus with the following properties for some $L>0$:
	\begin{itemize}
		\item $\alpha\neq\beta$, $\alpha\cap\beta\neq\emptyset$;
		\item $X_{\alpha\beta}$ is of type $S_{1,2}$;
		\item $\ell(\alpha)\leq L$, $\ell(\beta)\leq L$;
		\item $\ell(\partial X_{\alpha\beta})\geq\frac{5}{3}L$.
	\end{itemize}
	Then $\alpha$ and $\beta$ have exactly $4$ intersection points, and the intersection $\mathring{X_\alpha}\cap \mathring{X_\beta}$ (where `$\circ$' denotes the interior) contains a  unique simple closed geodesic $\delta$ (see Figure \ref{a new geodesic in S12}).
\end{lemma}
Note that since a one-holed torus with geodesic boundary is cut by any simple closed geodesic in its interior into a pair-of-pants, the geodesic $\delta$ given by the lemma cuts $X_{\alpha\beta}$ into a $4$-holed sphere containing both $\alpha$ and $\beta$.
\begin{figure}[h]
	\centering
	\includegraphics[width=5cm]{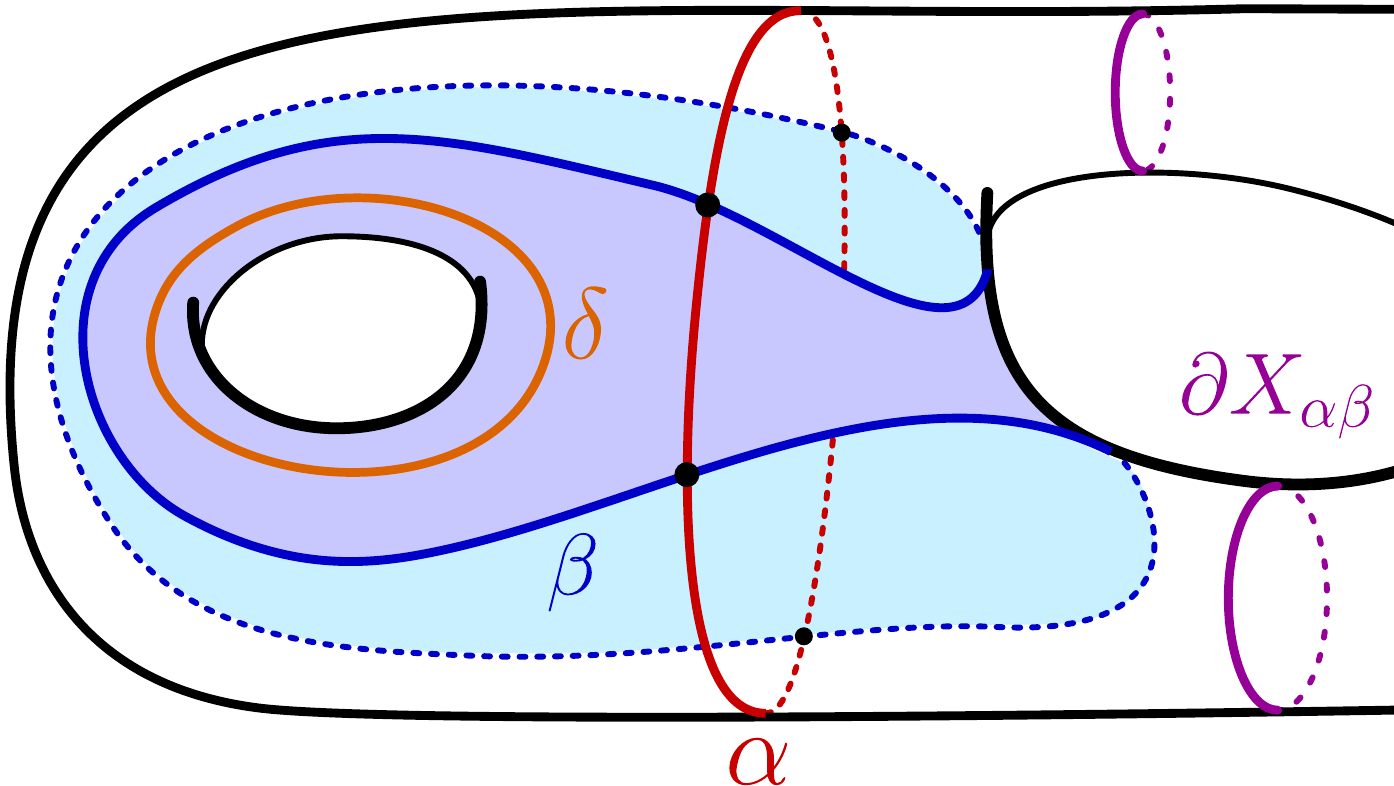}
	\caption{Simple closed geodesic $\delta\subset \mathring{X_\alpha}\cap \mathring{X_\beta}$.}
	\label{a new geodesic in S12}
\end{figure} 
\begin{proof}
    We first show $$\#(\alpha\cap\beta)=4.$$ 
    
	Since $\alpha$ and $\beta$ are both separating and hence represent the zero homology class, $\#(\alpha\cap\beta)$ is a positive even number. Moreover, if $\#(\alpha\cap\beta)=2$, then $\beta\cap X_\alpha$ is a single simple geodesic arc splitting $X_\alpha$ into at least two pieces (namely, the connected components of $X_\alpha\cap \mathring{X_\beta}$ and $X_\alpha\setminus X_\beta$), which is a contradiction because a simple geodesic arc in a one-holed torus joining boundary points cannot separate the one-holed torus. 
	
	Thus, if $\#(\alpha\cap\beta)$ is not $4$, then it is at least $6$. In this case, $\beta\setminus \mathring{X_\alpha}$ consists of at least $3$ segment. We denote the shortest two among these segments by $\beta_1$ and $\beta_2$, whose total length satisfy
	$$
	\ell(\beta_1)+\ell(\beta_2)\leq\tfrac{2}{3}\ell(\beta)\leq \tfrac{2}{3}L.
	$$	
	Since $\beta_1$ and $\beta_2$ are disjoint geodesic arcs in the pair of pants $X_{\alpha\beta}\setminus \mathring{X_\alpha}$ with endpoints in the same boundary component $\alpha$, they are homotopic to each other relative to $\alpha$. Therefore, $\partial X_{\alpha\beta}$ is homotopic to the two closed piecewise geodesics formed by $\beta_1$, $\beta_2$ along with two disjoint segments of $\alpha$, and hence
	$$
	\ell(\partial X_{\alpha\beta})<\ell(\alpha)+\ell(\beta_1)+\ell(\beta_2)\leq L+\tfrac{2}{3}L=\tfrac{5}{3}L,
	$$
	contradicting the assumption $\ell(X_{\alpha\beta})\geq\frac{5}{3}L$. This proves $\#(\alpha\cap\beta)=4$.
	
	As a consequence, $\beta$ is split by $\alpha$ into $4$ segments. Since the two segments $\beta_1,\beta_2$ outside of $\mathring{X_\alpha}$ are homotopic relative to $\alpha$ as above, $X_\beta\setminus \mathring{X_\alpha}$ is homeomorphic to a disk. We now consider the other two segments, which are in $X_\alpha$, and denote them by $\beta_1',\beta_2'$.
	
	It is a basic fact that given a one-holed torus $Y$ with geodesic boundary, any two disjoint simple geodesic arcs $a_1,a_2\subset Y$ with endpoints in $\partial Y$ belong to one of the following cases (see Figure \ref{figure:arcs in one handle}):
	\begin{itemize}
		\item[(1)] If $a_1$ and $a_2$ are homotopic relative to $\partial Y$, then they split $Y$ into two pieces, namely a topological cylinder and a topological disk; 
		\item[(2)] Otherwise, $a_1$ and $a_2$ split $Y$ into a single topological disk.
	\end{itemize} 
\begin{figure}[h]
	\centering
	\includegraphics[width=6.5cm]{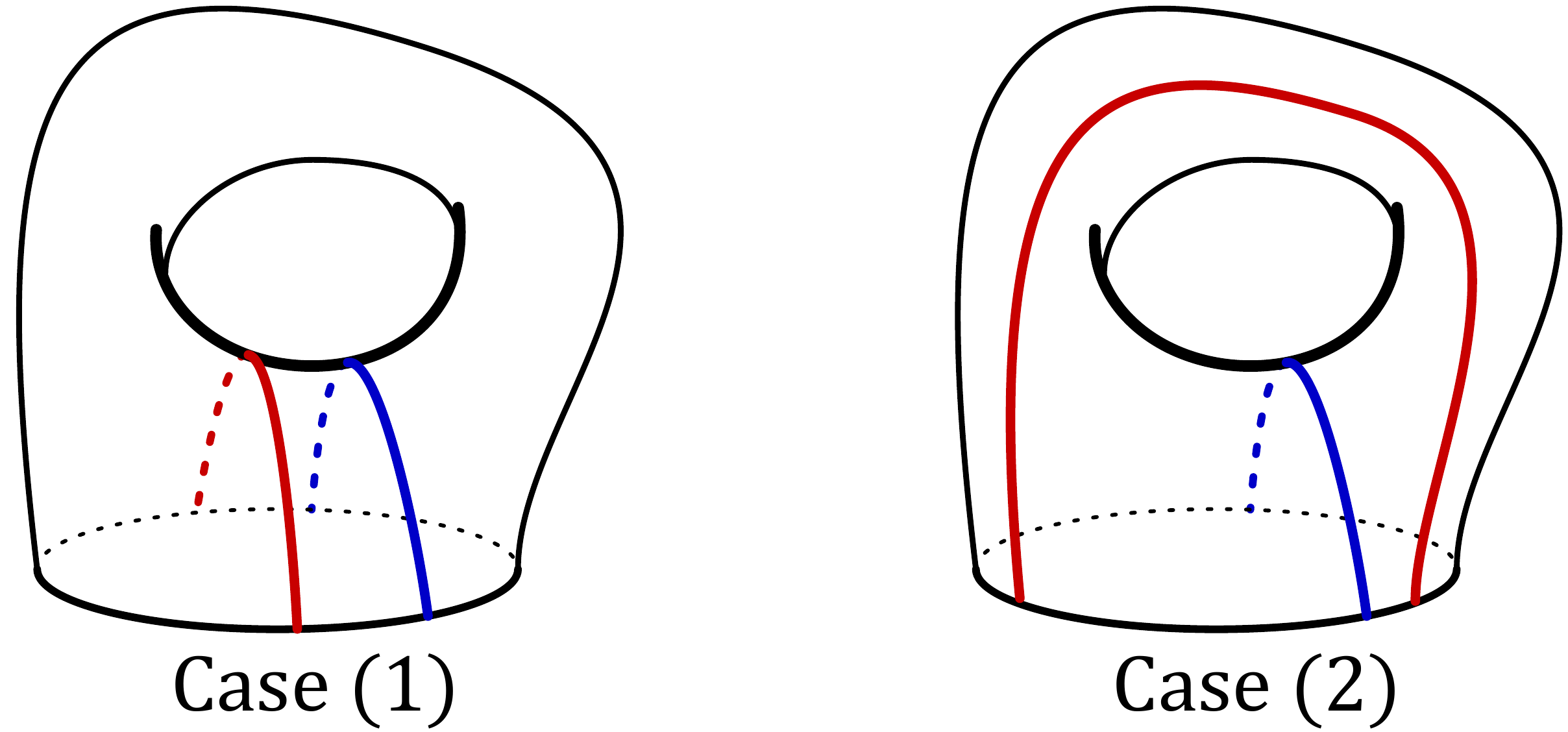}
	\caption{Two arcs in a one-holed torus.}\label{figure:arcs in one handle}
\end{figure} 
	
Since $\beta_1'$ and $\beta_2'$ separate $X_\alpha$, they are in Case (1). Thus, among the two pieces of $X_\alpha$ split out by $\beta$, namely $X_\alpha\cap X_\beta$ and $X_\alpha\setminus \mathring{X_\beta}$, one is a cylinder and the other is a disk. But we have shown above that $X_\beta\setminus \mathring{X_\alpha}$ is a disk, and the argument implies that $X_\alpha\setminus \mathring{X_\beta}$ is a disk as well if we switch the roles of $\alpha$ and $\beta$. Therefore, we conclude that $X_\alpha\cap X_\beta$ is a cylinder as shown in Figure \ref{a new geodesic in S12}. This cylinder contains a unique simple closed geodesic $\delta$, namely the one homotopic to its boundary loops. And $\delta$ is in the interior of the cylinder since it is contained in both $X_\alpha$ and $X_\beta$, as required.
\end{proof}

\begin{rem*}
By construction, $\alpha\cup\beta$ is homotopic to $\partial X_{\alpha \beta}\cup2\delta$, where $2\delta$ means two copies of $\delta$'s. We will use this observation later in Subsection \ref{proof of third tend to 0}.
\end{rem*}

\begin{rem*}\label{remark:topology}
	The second statement of Lemma \ref{lemma:4holded} actually holds true for any intersecting pair $(\alpha,\beta)$ of simple closed geodesics bounding one-holed torus such that $X_{\alpha\beta}$ is of type $S_{1,2}$ (\cf the first example in Figure \ref{figure_examples}). The proof is more complicated and not necessary for our purpose.
\end{rem*}


\section{Weil-Petersson volume}\label{section wp volume}

In this section we give some results on the \wep \ volumes of moduli spaces. All of these are already known or generalizations of known results. We denote $V_{g,n}(x_1,\cdots,x_n)$ to be the \wep \ volume of $\M_{g,n}(x_1,\cdots,x_n)$ and $V_{g,n}= V_{g,n}(0,\cdots,0)$. One may also see \cite{Agg21, Grus01, LX14, Mirz07, Mirz07-int, Mirz13, MZ15, Penner92, ST01, Zograf08, AM22} for the asymptotic behavior of $V_{g,n}$ and its deep connection to the intersection theory of $\M_{g,n}$.

First we recall several results of Mirzakhani and her coauthors.

\begin{theorem}\label{Mirz vol lemma 0}
\begin{enumerate}
  \item \cite[Theorem 1.1]{Mirz07}
The volume $V_{g,n}(x_1,\cdots,x_n)$ is a polynomial in $x_1^2,\cdots,x_n^2$ with degree $3g-3+n$. Namely we have
\begin{equation*}
V_{g,n}(x_1,\cdots,x_n) = \sum_{\alpha;\,|\alpha|\leq 3g-3+n} C_\alpha \cdot x^{2\alpha}
\end{equation*}
where $C_\alpha>0$ lies in $\pi^{6g-6+2n-|2\alpha|} \cdot \mathbb Q$. Here $\alpha=(\alpha_1,\cdots,\alpha_n)$ is a multi-index and $|\alpha|=\alpha_1+\cdots+\alpha_n$, $x^{2\alpha}= x_1^{2\alpha_1}\cdots x_n^{2\alpha_n}$.

  \item \cite[Table 1]{Mirz07}
\begin{equation*}
V_{0,3}(x,y,z) = 1,
\end{equation*}
\begin{equation*}
V_{1,1}(x) = \frac{1}{48}(x^2+4\pi^2).
\end{equation*}
\end{enumerate}
\begin{rem*}
We remark here that $V_{1,1}(x)$ is only one-half of the quantity originally given in \cite{Mirz07} or \cite{Mirz08}. The reason is that there exists an involution for $S_{1,1}$. One may also see \cite[Remark 3.3]{DGZZ21} for more details.
\end{rem*}
\end{theorem}

\begin{lemma}\label{Mirz vol lemma 1}
\begin{enumerate}

  \item \cite[Lemma 3.2]{Mirz13}
\begin{equation*}
V_{g,n} \leq V_{g,n}(x_1,\cdots,x_n) \leq e^{\frac{x_1+\cdots+x_n}{2}} V_{g,n}.
\end{equation*}

  \item \cite[Lemma 3.2]{Mirz13}
For any $g,n\geq 0$
\begin{equation*}
V_{g-1,n+4} \leq V_{g,n+2}
\end{equation*}
and
\begin{equation*}
b_0\leq \frac{V_{g,n+1}}{(2g-2+n)V_{g,n}} \leq b_1
\end{equation*}
for some universal constants $b_0,b_1>0$ independent of $g,n$.

  \item \cite[Theorem 3.5]{Mirz13}
For fixed $n\geq 0$, as $g\rightarrow \infty$ we have
\begin{equation*}
\frac{V_{g,n+1}}{2g V_{g,n}} = 4\pi^2 + O\left(\frac{1}{g}\right),
\end{equation*}
\begin{equation*}
\frac{V_{g,n}}{V_{g-1,n+2}} = 1 + O\left(\frac{1}{g}\right).
\end{equation*}
Where the implied constants depend on $n$ but not on $g$.
\end{enumerate}

\end{lemma}
\begin{rem*}
For Part $(3)$, one may also see the following Theorem \ref{MZ vol thm} of Mirzakhani-Zograf.
\end{rem*}

\begin{lemma}\cite[Corollary 3.7]{Mirz13} \label{Mirz vol lemma 2}
For fixed $b,k,r\geq 0$ and $C<C_0= 2\log 2$,
\begin{equation*}
\sum_{\begin{array}{c}
        g_1+g_2=g+1-k  \\
        r+1\leq g_1\leq g_2
      \end{array}}
e^{Cg_1} \cdot g_1^b \cdot V_{g_1,k} \cdot V_{g_2,k} \asymp \frac{V_g}{g^{2r+k}}
\end{equation*}
as $g\rightarrow\infty$. The implied constants depend on $b,k,r,C$ but not on $g$. Here $A\asymp B$ means $c_1 A\leq B\leq c_2 A$ for two constants $c_1,c_2>0$ independent of $g$.

\end{lemma}

\begin{theorem}\cite[Theorem 1.2]{MZ15} \label{MZ vol thm}
There exists a universal constant $\alpha>0$ such that for any given $n\geq0$,
\begin{equation*}
V_{g,n} = \alpha \frac{1}{\sqrt{g}} (2g-3+n)! (4\pi^2)^{2g-3+n} \left(1+O\left(\frac{1}{g}\right) \right)
\end{equation*}
as $g\rightarrow\infty$. The implied constant depend on $n$ but not on $g$.
\end{theorem}

\begin{rem*}
It is conjectured by Zograf in \cite{Zograf08} that $\alpha=\frac{1}{\sqrt\pi}$, which is still open.
\end{rem*}

The following result is motivated by \cite[Proposition 3.1]{MP19} whose error term in the form of multiplication is replaced by the following summation. One may also see the very recent work of Anantharaman and Monk \cite{AM22} for sharper results.
\begin{lemma} \label{MP vol lemma}
Let $g,n\geq 1$ and $x_1,\cdots,x_n\geq 0$, then there exists a constant $c=c(n)>0$ independent of $g,x_1,\cdots,x_n$ such that
\begin{equation*}
\prod_{i=1}^n \frac{\sinh(x_i/2)}{x_i/2} \big(1- c(n)\frac{\sum_{i=1}^n x_i^2}{g}\big)
\leq \frac{V_{g,n} (x_1,\cdots,x_n)}{ V_{g,n}}
\leq \prod_{i=1}^n \frac{\sinh(x_i/2)}{x_i/2}.
\end{equation*}
\end{lemma}

\begin{proof}
By Theorem \ref{Mirz vol lemma 0} we know that $V_{g,n} (2x_1,\cdots,2x_n)$ is a polynomial of $x_1^2,\cdots,x_n^2$ with degree $3g-3+n$. As in \cite[(3.1)]{Mirz13} we write
\begin{equation*}
V_{g,n} (2x_1,\cdots,2x_n) = \sum_{|\boldsymbol{d}| \leq 3g-3+n} [\tau_{d_1},\cdots,\tau_{d_n}]_{g,n} \frac{x_1^{2d_1}}{(2d_1+1)!}\cdots\frac{x_n^{2d_n}}{(2d_n+1)!}
\end{equation*}
where $\boldsymbol{d} = (d_1,\cdots,d_n)$ with $d_i\geq 0$ and $|\boldsymbol{d}| = d_1+\cdots+d_n$. In \cite[page 286]{Mirz13},
Mirzakhani gave the following bound for $[\tau_{d_1},\cdots,\tau_{d_n}]_{g,n}$. Given $n\geq1$, we have
\begin{equation*}
0\leq 1- \frac{[\tau_{d_1},\cdots,\tau_{d_n}]_{g,n} }{V_{g,n}} \leq c_0 \frac{|\boldsymbol{d}|^2}{g}
\end{equation*}
where $c_0$ is independent of $g$ and $\boldsymbol{d}$ (but may depend on $n$). So we have
\begin{equation*}
\frac{V_{g,n}(2x_1,\cdots,2x_n)}{V_{g,n}} \leq \sum_{|\boldsymbol{d}| \leq 3g-3+n} \frac{x_1^{2d_1}}{(2d_1+1)!}\cdots\frac{x_n^{2d_n}}{(2d_n+1)!}
\end{equation*}
and
\begin{eqnarray*}
\frac{V_{g,n}(2x_1,\cdots,2x_n)}{V_{g,n}}
&&\geq \sum_{|\boldsymbol{d}| \leq 3g-3+n} \frac{x_1^{2d_1}}{(2d_1+1)!}\cdots\frac{x_n^{2d_n}}{(2d_n+1)!} \\
& & - \frac{c_0}{g} \sum_{|\boldsymbol{d}| \leq 3g-3+n} |\boldsymbol{d}|^2\frac{x_1^{2d_1}}{(2d_1+1)!}\cdots\frac{x_n^{2d_n}}{(2d_n+1)!}.
\end{eqnarray*}
Recall that
\begin{equation*}
\prod_{i=1}^n \frac{\sinh(x_i)}{x_i} =
\sum_{d_1,\cdots,d_n =0}^\infty \frac{x_1^{2d_1}}{(2d_1+1)!}\cdots\frac{x_n^{2d_n}}{(2d_n+1)!}.
\end{equation*}
So we get the upper bound
\begin{equation*}
\frac{V_{g,n}(2x_1,\cdots,2x_n)}{V_{g,n}} \leq \prod_{i=1}^n \frac{\sinh(x_i)}{x_i}.
\end{equation*}

For the lower bound, first we have
\begin{eqnarray*}
x_1^2\prod_{i=1}^n \frac{\sinh(x_i)}{x_i} &=& 
\left(\sum_{d_1=1}^\infty  \frac{x_1^{2d_1}}{(2d_1-1)!}\right) \left(\sum_{d_2,\cdots,d_n =0}^\infty \frac{x_2^{2d_2}}{(2d_2+1)!} \cdots \frac{x_n^{2d_n}}{(2d_n+1)!}\right) \\
&=& \sum_{d_1,\cdots,d_n =0}^\infty \frac{x_1^{2d_1}}{(2d_1+1)!} \cdots \frac{x_n^{2d_n}}{(2d_n+1)!} 2d_1(2d_1+1).
\end{eqnarray*}
So
\begin{equation*}
\left(\sum_{i=1}^n x_i^2\right) \prod_{i=1}^n \frac{\sinh(x_i)}{x_i} =
\sum_{d_1,\cdots,d_n =0}^\infty \left(\frac{x_1^{2d_1}}{(2d_1+1)!}\cdots\frac{x_n^{2d_n}}{(2d_n+1)!} \sum_{i=1}^n 2d_i(2d_i+1)\right).
\end{equation*}
Then by Cauchy-Schwarz inequality we have
\begin{equation*}
\sum_{|\boldsymbol{d}| \leq 3g-3+n} |\boldsymbol{d}|^2\frac{x_1^{2d_1}}{(2d_1+1)!}\cdots\frac{x_n^{2d_n}}{(2d_n+1)!}
\leq \frac{n}{4}(x_1^2+\cdots+x_n^2)\prod_{i=1}^n \frac{\sinh(x_i)}{x_i}.
\end{equation*}
Recall that the Stirling formula says that
$$k!\sim \sqrt{2\pi k} (\frac{k}{e})^k$$
which implies that for large $k>0$, $$k!\geq (\frac{k}{e})^k.$$ Hence, we have
\begin{eqnarray*}
&& \sum_{|\boldsymbol{d}| > 3g-3+n} \frac{x_1^{2d_1}}{(2d_1+1)!}\cdots\frac{x_n^{2d_n}}{(2d_n+1)!} \\
&\leq& \sum_{k>3g-3+n} \frac{1}{k!} \sum_{|\boldsymbol{d}| =k}
\frac{k!}{d_1!\cdots d_n!}(x_1^2)^{d_1}\cdots(x_n^2)^{d_n} \\
&=& \sum_{k>3g-3+n} \frac{1}{k!} (x_1^2+\cdots+x_n^2)^k \\
&\leq& \sum_{k>3g-3+n} \big( \frac{e\cdot (x_1^2+\cdots+x_n^2)}{k} \big)^k.
\end{eqnarray*}
If $\frac{e\cdot (x_1^2+\cdots+x_n^2)}{3g-2+n} \leq 0.5$, we have
\begin{eqnarray*}
 \sum_{|\boldsymbol{d}| > 3g-3+n} \frac{x_1^{2d_1}}{(2d_1+1)!}\cdots\frac{x_n^{2d_n}}{(2d_n+1)!}
&\leq& 2 \big( \frac{e\cdot (x_1^2+\cdots+x_n^2)}{3g-2+n} \big)^{3g-2+n} \\
&\leq& 4 \frac{x_1^2+\cdots+x_n^2}{g} \prod_{i=1}^n \frac{\sinh(x_i)}{x_i}.
\end{eqnarray*}

Then we get when $\frac{e\cdot (x_1^2+\cdots+x_n^2)}{3g-2+n} \leq 0.5$,
\begin{eqnarray*}
\frac{V_{g,n}(2x_1,\cdots,2x_n)}{V_{g,n}}
&\geq& \prod_{i=1}^n \frac{\sinh(x_i)}{x_i} - \sum_{|\boldsymbol{d}| > 3g-3+n} \frac{x_1^{2d_1}}{(2d_1+1)!}\cdots\frac{x_n^{2d_n}}{(2d_n+1)!} \\
& & - \frac{c_0}{g} \sum_{|\boldsymbol{d}| \leq 3g-3+n} |\boldsymbol{d}|^2\frac{x_1^{2d_1}}{(2d_1+1)!}\cdots\frac{x_n^{2d_n}}{(2d_n+1)!} \\
&\geq& \prod_{i=1}^n \frac{\sinh(x_i)}{x_i} \big( 1-(\frac{n}{4}c_0 + 4)\frac{x_1^2+\cdots+x_n^2}{g} \big).
\end{eqnarray*}

\noindent If $\frac{e(x_1^2+\cdots+x_n^2)}{3g-2+n} > 0.5$, then $e \cdot \frac{x_1^2+\cdots+x_n^2}{g}>1$ and the lower bound is trivial in this case.
\end{proof}

\begin{rem*} In the proof above,
\begin{enumerate}
\item for the lower bound, the $x_i's$ may depend on $g$ but $n$ is fixed;
\item for the upper bound, both the $x_i's$ and $n$ may depend on $g$ as $g\rightarrow\infty$.
\end{enumerate}
\end{rem*}

One may observe from Lemma \ref{Mirz vol lemma 1}, \ref{Mirz vol lemma 2} and Theorem \ref{MZ vol thm} that the asymptotic behavior of $V_{g,n}$ is related to the Euler characteristic $\chi(S_{g,n})=-2g+2-n$. We use the following quantity $W_r$ to approximate any $V_{g,n}$ with $2g-2+n=r$: 
$$
W_{r}:=
\begin{cases}
V_{\frac{r}{2}+1,0}&\text{if $r$ is even},\\[5pt]
V_{\frac{r+1}{2},1}&\text{if $r$ is odd}.
\end{cases}
$$

Now we provide the following properties for $W_r$ which will be applied later.
\begin{lemma}\label{Wr-prop}
	\begin{enumerate}
		\item For any $g,n\geq 0$, we have
		$$V_{g,n}  \leq c \cdot W_{2g-2+n}$$
		for some universal constant $c>0$.
		\item For any $r\geq1$ and $m_0\leq \frac{1}{2}r$, we have
		$$\sum_{m=m_0}^{[\frac{r}{2}]} W_m W_{r-m} \leq c(m_0) \frac{1}{r^{m_0}}W_r$$
		for some constant $c(m_0)>0$ only depending on $m_0$.
	\end{enumerate}
\end{lemma}

\begin{proof}
	For $(1)$, first by Part $(2)$ of Lemma \ref{Mirz vol lemma 1} we know that there exists a pair $(g',n')$  with $0\leq n'\leq 3$ and $2g'-2+n'=2g-2+n$ such that
	$$V_{g,n} \leq V_{g',n'}.$$
	Again by Part $(3)$ of Lemma \ref{Mirz vol lemma 1} or Theorem \ref{MZ vol thm} we know that there is a universal constant $c>0$ such that 
	$$V_{g',2}\leq c V_{g'+1} \quad \text{and} \quad V_{g',3}\leq cV_{g'+1,1}.$$
	So for odd $n>0$ we have 
	$$V_{g,n}\leq V_{g',n'}\leq c V_{g+\frac{n-1}{2},1} = c W_{2g-2+n},$$
	and for even $n\geq 0$ we also have
	$$V_{g,n}\leq V_{g',n'}\leq c V_{g+\frac{n}{2}} = c W_{2g-2+n}$$
which completes the proof of $(1)$.\\

	For $(2)$, we only show it for the case that both $m_0$ and $r$ are odd. The proofs of other cases are similar. We leave them as an exercise to the readers. First by Part $(3)$ of Lemma \ref{Mirz vol lemma 1}, there is a universal constant $c>0$ such that for odd $m$,
	$$W_{m}\leq c\frac{1}{m} V_{\frac{m+3}{2}}.$$
Recall that Part $(3)$ of Lemma \ref{Mirz vol lemma 1} implies that for some universal constant $c'>0$,
\[\frac{V_{g+1}}{V_{g,1}}\leq c'\cdot g.\]	
Then it follows by Lemma \ref{Mirz vol lemma 2} that there exist two constants $c'(m_0),c(m_0)>0$ only depending on $m_0$ such that
	\begin{eqnarray*}
	\sum_{m=m_0}^{[\frac{r}{2}]} W_{m} W_{r-m}
	&\leq& \sum_{\tiny\begin{array}{c}m=m_0+1 \\ m\ \text{even}\end{array}}^{[\frac{r}{2}]} \frac{c}{r-m} V_{\frac{m}{2}+1} V_{\frac{r-m+3}{2}} + \sum_{\tiny\begin{array}{c}m=m_0 \\ m\ \text{odd}\end{array}}^{[\frac{r}{2}]}  \frac{c}{m} V_{\frac{m+3}{2}} V_{\frac{r-m}{2}+1} \\
	&\leq& \frac{c}{r} \sum_{k=\frac{m_0+3}{2}}^{[\frac{r}{4}]+1} V_{k} V_{\frac{r+5}{2}-k} + \frac{c}{m_0} \sum_{k=\frac{m_0+3}{2}}^{[\frac{r}{4}]+1} V_{k} V_{\frac{r+5}{2}-k} \\
	&\leq& c'(m_0)\frac{1}{r^{m_0+1}} V_{\frac{r+3}{2}} \\
	&\leq& c(m_0)\frac{1}{r^{m_0}} V_{\frac{r+1}{2},1} \\
	&=& c(m_0)\frac{1}{r^{m_0}} W_r,
	\end{eqnarray*} 	
as required.
\end{proof}

The following lemma is a generalization of \cite[lemma 3.2]{MP19} and \cite[lemma 6.3]{GMST19}. Here we allow the $n_i$'s and $q$ to depend on $g$ as $g\to \infty$.

\begin{lemma}\label{sum vol lemma}
Assume $q\geq 1$, $n_1,\cdots,n_q\geq 0$, $r\geq2$. Then there exists two universal constants $c,D>0$ such that
\begin{equation*}
\sum_{\{g_i\}} V_{g_1,n_1}\cdots V_{g_q,n_q} \leq c \big(\frac{D}{r}\big)^{q-1} W_r
\end{equation*}
where the sum is taken over all $\{g_i\}_{i=1}^q \sbs \N$ such that $2g_i-2+n_i \geq 1$ for all $i=1,\cdots,q$, and $\sum_{i=1}^q (2g_i-2+n_i) = r$. 

\end{lemma}

\begin{proof}
Given a $\{g_i\}$ in the summation, let $g_i'\geq0$ and $0\leq n_i'\leq 3$ be such that $2g_i'-2+n_i'=2g_i-2+n_i$ for each $i$. By Lemma \ref{Mirz vol lemma 1} we know that 
$$V_{g_i,n_i} \leq V_{g'_i,n'_i}.$$
And by Theorem \ref{MZ vol thm}, we have
\begin{align*}
V_{g'_i,n'_i}&\leq \alpha_0 \frac{\sqrt2}{\sqrt{2g'_i-3+n'_i}} (2g'_i-3+n'_i)! (4\pi^2)^{2g'_i-3+n'_i}\\
&=\alpha_0 \frac{\sqrt2}{\sqrt{2g_i-3+n_i}} (2g_i-3+n_i)! (4\pi^2)^{2g_i-3+n_i}
\end{align*}
and 
$$
W_r\geq \alpha_1 \frac{\sqrt2}{\sqrt{r-1}} (r-1)! (4\pi^2)^{r-1}
$$
for universal constants $\alpha_0 > \alpha_1 >0$. 

Recall that the Stirling's formula says that as $k\rightarrow\infty$, $$k!\sim \sqrt{2\pi k} (\frac{k}{e})^k.$$  So there exist two universal constants $a_0>a_1>0$ such that
\begin{equation*}
a_1 \sqrt{2\pi } (\frac{k}{e})^k\leq \frac{k!}{\sqrt{k}} \leq a_0 \sqrt{2\pi } (\frac{k}{e})^k.
\end{equation*}

\noindent Now we have
\begin{eqnarray}\label{mp-eq-1}
& & \frac{\sum_{\{g_i\}} V_{g_1,n_1}\cdots V_{g_q,n_q}}{W_r} \\
&\leq& \frac{\sum_{\{g_i\}} \prod_{i=1}^q 2\sqrt{\pi}a_0\alpha_0 (\frac{2g_i-3+n_i}{e})^{2g_i-3+n_i} (4\pi^2)^{2g_i-3+n_i}}
{2\sqrt{\pi}a_1\alpha_1 (\frac{r-1}{e})^{r-1} (4\pi^2)^{r-1}} \nonumber\\
&=& \frac{1}{2\sqrt{\pi}a_1\alpha_1\frac{e}{4\pi^2}} (2\sqrt{\pi}a_0\alpha_0\frac{e}{4\pi^2})^q
\frac{\sum_{\{g_i\}} \prod_{i=1}^q (2g_i-3+n_i)^{2g_i-3+n_i}} {(r-1)^{r-1}}.\nonumber
\end{eqnarray}
For each $i=1,\cdots,q$, we have $2g_i-3+n_i\geq 0$. Now assume exactly $j$ of numbers $(2g_i-3+n_i)$ are non-zero. The number of such $\{g_i\}$ (such that $\sum_{i=1}^q (2g_i-3+n_i) = r-q$) is bounded from above by
\begin{equation*}
\binom{q}{j}
\binom{r-q-1}{j-1}
\end{equation*}
where $\binom{q}{j} = \frac{q!}{j!(q-j)!}$ is the binomial coefficient.

Recall the following elementary fact: if $\sum_{i=1}^j x_i = S$ and $x_i\geq 1$ for all $i$, then $\prod_{i=1}^j x_i^{x_i}$ reaches the maximum value when $j-1$ of the $x_i's$ are $1$. As a result, we have
\begin{equation*}
\prod_{i=1}^j x_i^{x_i} \leq (S-j+1)^{S-j+1}.
\end{equation*}

\noindent Thus for each such $\{g_i\}$ we have
\begin{equation*}
\prod_{i=1}^q (2g_i-3+n_i)^{2g_i-3+n_i} \leq 1^1\cdots 1^1\cdot (r-q-j+1)^{r-q-j+1}.
\end{equation*}
So we have
\begin{eqnarray*}
& & \frac{\sum_{\{g_i\}} \prod_{i=1}^q (2g_i-3+n_i)^{2g_i-3+n_i}} {(r-1)^{r-1}} \\
& \leq & \frac{1} {(r-1)^{r-1}} \sum_{j=0}^q
\binom{q}{j}
\binom{r-q -1}{j-1}
(r-q-j+1)^{r-q-j+1} \\
& \leq & \sum_{j=0}^q \binom{q}{j}
\frac{(r-q-1)^{j-1} (r-q-j+1)^{r-q-j+1}} {(r-1)^{r-1}} \\
& \leq & \sum_{j=0}^q\binom{q}{j}
\frac{1} {(r-1)^{q-1}} \\
& = & \frac{2^q} {(r-1)^{q-1}}.
\end{eqnarray*}

Then combining \eqref{mp-eq-1} we get
\begin{equation*}
\frac{\sum_{\{g_i\}} V_{g_1,n_1}\cdots V_{g_q,n_q}}{W_r}
\leq 2\frac{a_0\alpha_0}{a_1\alpha_1} \big(\frac{a_0\alpha_0\frac{e}{\pi^{3/2}}}{r-1}\big)^{q-1},
\end{equation*}
as required.
\end{proof}

We finish this section by the following useful property.
\begin{proposition}\label{1 over gm} 
	Given $m\geq 1$, there exists a constant $c(m)>0$ only depending on $m$ such that for any $g\geq m+1$, $q\geq 1$ and $n_1,...,n_q\geq 1$, we have
	\begin{equation*}
	\sum_{\{g_i\}} V_{g_1,n_1}\cdots V_{g_q,n_q} \leq c(m)\frac{1}{g^m} V_g,
	\end{equation*}
    where the sum is taken over all $\{g_i\}_{i=1}^q \sbs \N$ such that $2g_i-2+n_i \geq 1$ for all $i=1,\cdots,q$, and $\sum_{i=1}^q (2g_i-2+n_i) = 2g-2-m$.
\end{proposition}

\begin{proof}
	If $g$ is bounded from above, then the nonnegative integers $m,q,n_1,\cdots n_q$ are all bounded from above, and hence the inequality is trivial. It suffices to show it for large enough $g$. First by Lemma \ref{sum vol lemma} we know that 
	\begin{equation*}
		\sum_{\{g_i\}} V_{g_1,n_1}\cdots V_{g_q,n_q} \leq c \big(\frac{D}{2g-2-m}\big)^{q-1} W_{2g-2-m}.
	\end{equation*}
	By Part $(3)$ of Lemma \ref{Mirz vol lemma 1} or Theorem \ref{MZ vol thm} we know that 
 \[\frac{V_g}{V_{g-1}}\asymp g^2 \quad \textit{and} \quad \frac{V_{g,1}}{V_g}\asymp g \]
where we say $f\asymp h$ if there exists a uniform constant $C\geq 1$ such that $\frac{1}{C}\leq \frac{f}{h}\leq C$. This in particular implies that there exists a constant $c'(m)>0$ only depending on $m$ such that
\begin{equation*}
		W_{2g-2-m} \leq c'(m) \frac{1}{g^m} V_g.
	\end{equation*}
Therefore, we have that for large enough $g>0$,
    \begin{eqnarray*}
    	\sum_{\{g_i\}} V_{g_1,n_1}\cdots V_{g_q,n_q} &\leq& c'(m) (\frac{D}{g})^{q-1} \frac{1}{g^m} V_g\\
& \leq& c(m) \frac{1}{g^m} V_g
    \end{eqnarray*}
for some constant $c(m)>0$ only depending on $m$, as required.
\end{proof}

\section{Mirzakhani's generalized McShane identity}\label{section McShane identity}

In \cite{Mirz07} Mirzakhani generalized McShane's identity \cite{McS98} as follows, and then calculated the Weil-Petersson volume of moduli spaces by applying her integration formula (see Theorem \ref{Mirz int formula}).

\begin{theorem}\cite[Theorem 1.3]{Mirz07} \label{McShane id}
For $X\in\M_{g,n}(L_1,\cdots,L_n)$ with $n$ geodesic boundaries $\beta_1,\cdots,\beta_n$ of length $L_1,\cdots,L_n$, we have
\begin{equation*}
\sum_{\{\gamma_1,\gamma_2\}} \mathcal D(L_1, \ell(\gamma_1), \ell(\gamma_2)) +
\sum_{i=2}^n \sum_\gamma \mathcal R(L_1,L_i,\ell(\gamma)) = L_1
\end{equation*}
where the first sum is over all unordered pairs of simple closed geodesics $\{\gamma_1, \gamma_2\}$ bounding a pair of pants with $\beta_1$, and the second sum is over all simple closed geodesics $\gamma$ bounding a pair of pants with $\beta_1$ and $\beta_i$. Here $\mathcal D$ and $\mathcal R$ are given by
\begin{equation*}
\mathcal D(x,y,z) = 2\log\big( \frac{e^{\frac{x}{2}}+e^{\frac{y+z}{2}}}{e^{\frac{-x}{2}}+e^{\frac{y+z}{2}}} \big),
\end{equation*}
\begin{equation*}
\mathcal R(x,y,z) = x - \log\big( \frac{\cosh(\frac{y}{2})+\cosh(\frac{x+z}{2})}{\cosh(\frac{y}{2})+\cosh(\frac{x-z}{2})} \big).
\end{equation*}

\end{theorem}

We will use this identity in subsection \ref{proof of third tend to 0} to control the number of certain types of closed geodesics in a surface. Here we provide the following elementary properties for $\mathcal D(x,y,z)$ and $\mathcal R(x,y,z)$.

\begin{lemma} \label{estimation R,D}
Assume that $x,y,z> 0$, then the following properties hold.

\begin{enumerate}
  \item $\mathcal R(x,y,z)\geq 0$ and $\mathcal D(x,y,z)\geq0$.
  \item $\mathcal R(x,y,z)$ is decreasing with respect to $z$ and increasing with respect to $y$. $\mathcal D(x,y,z)$ is decreasing with respect to $y$ and $z$ and increasing with respect to $x$.
  \item We have
\begin{equation*}
\frac{x}{\mathcal R(x,y,z)} \leq 100(1+x)(1+e^{\frac{z}{2}}e^{-\frac{x+y}{2}}),
\end{equation*}
and
\begin{equation*}
\frac{x}{\mathcal D(x,y,z)} \leq 100(1+x)(1+e^{\frac{y+z}{2}}e^{-\frac{x}{2}}).
\end{equation*}
Moreover, if $x+y>z$, we have
\begin{equation*}
\frac{x}{\mathcal R(x,y,z)} \leq 500+ 500\frac{x}{x+y-z}.
\end{equation*}

\end{enumerate}
\end{lemma}

\begin{proof}
Part $(1)$ is easy to check. Actually $\mathcal D$ and $\mathcal R$ given in \cite{Mirz07} are lengths of certain segments for $x,y,z>0$.\\

For Part $(2)$, a direct computation shows that
\begin{eqnarray*}
&& \frac{d}{dz}\big( \frac{\cosh(\frac{y}{2})+\cosh(\frac{x+z}{2})}{\cosh(\frac{y}{2})+\cosh(\frac{x-z}{2})} \big) \\
&=& \frac{\frac12 \sinh\frac{x+z}{2}(\cosh\frac{y}{2}+\cosh\frac{x-z}{2}) + \frac12 \sinh\frac{x-z}{2}(\cosh\frac{y}{2}+\cosh\frac{x+z}{2}) } {(\cosh(\frac{y}{2})+\cosh(\frac{x-z}{2}))^2} \\
&=& \frac{\sinh\frac{x}{2}\cosh\frac{z}{2}\cosh\frac{y}{2} + \frac12 \sinh x} {(\cosh(\frac{y}{2})+\cosh(\frac{x-z}{2}))^2} \\
&>& 0
\end{eqnarray*}
where we have used the elementary equations
\begin{eqnarray*}
\sinh(a+b)=\sinh a\cosh b + \cosh a\sinh b, \\
\sinh a+\sinh b = 2\sinh\frac{a+b}{2}\cosh\frac{a-b}{2}.
\end{eqnarray*}
So $\mathcal R(x,y,z)$ is decreasing with respect to $z$. The other parts of $(2)$ are obvious.\\

For Part $(3)$, first as for $\mathcal R(x,y,z)$ we have
\begin{eqnarray}\label{estimation R,D--R geq}
\mathcal R(x,y,z)
&=& \log\big( e^x \frac{\cosh(\frac{y}{2})+\cosh(\frac{x-z}{2})}{\cosh(\frac{y}{2})+\cosh(\frac{x+z}{2})} \big) \\
&=& \log\big( e^x \frac{e^{\frac{y}{2}}+e^{\frac{-y}{2}} + e^{\frac{x-z}{2}}+e^{\frac{z-x}{2}}}
{e^{\frac{y}{2}}+e^{\frac{-y}{2}} + e^{\frac{x+z}{2}}+e^{\frac{-x-z}{2}}} \big) \nonumber\\
&=& \log\big( e^x \frac{e^y e^{\frac{x+z}{2}}+e^{\frac{x+z}{2}} + e^x e^{\frac{y}{2}}+e^z e^{\frac{y}{2}}}
{e^y e^{\frac{x+z}{2}}+e^{\frac{x+z}{2}} + e^{x+z} e^{\frac{y}{2}}+e^{\frac{y}{2}}} \big) \nonumber\\
&=& \log\big( 1+ \frac{(e^x -1)(e^y +1)e^{\frac{x+z}{2}} + (e^{2x}-1)e^{\frac{y}{2}}}
{e^y e^{\frac{x+z}{2}}+e^{\frac{x+z}{2}} + e^{x+z} e^{\frac{y}{2}}+e^{\frac{y}{2}}} \big) \nonumber\\
&\geq& \log\big( 1+ \frac{(e^x -1)(e^y +1)e^{\frac{x+z}{2}} }
{e^y e^{\frac{x+z}{2}}+e^{\frac{x+z}{2}} + 2e^{x+z} e^{\frac{y}{2}}} \big) \nonumber\\
&=& \log\big( 1+ \frac{e^x -1}{1 + e^{\frac{x+z}{2}} \frac{2e^{\frac{y}{2}}}{e^y +1} } \big) \nonumber\\
&=& \log\big( 1+ \frac{e^x -1}{1 + e^{\frac{x+z}{2}} \frac{1}{\cosh \frac{y}{2}} } \big). \nonumber
\end{eqnarray}
Then we treat the following cases separately:

\textbf{Case 1: $\frac{e^x -1}{1 + e^{\frac{x+z}{2}} \frac{1}{\cosh \frac{y}{2}} } \geq 1$}. 
Then we have $e^x\geq 2$ and by \eqref{estimation R,D--R geq}
\begin{equation}\label{estimation R,D--x/R case 1}
\frac{x}{\mathcal R(x,y,z)} \leq \frac{x}{\log 2} \leq 2x.
\end{equation}

\textbf{Case 2: $\frac{e^x -1}{1 + e^{\frac{x+z}{2}} \frac{1}{\cosh \frac{y}{2}} }< 1$.} Recall that $\log(1+t)\geq \frac{t}{2}$ for $0<t\leq 1$. Then by \eqref{estimation R,D--R geq} we have
\begin{eqnarray}\label{estimation R,D--x/R case 2}
\frac{x}{\mathcal R(x,y,z)}
&\leq& \frac{x}{\frac{1}{2} \frac{e^x -1}{1 + e^{\frac{x+z}{2}} \frac{1}{\cosh \frac{y}{2}}} } \\
&=& \frac{2x}{e^x-1}+ e^{\frac{z}{2}}\frac{x}{\sinh\frac{x}{2}}\frac{1}{\cosh \frac{y}{2}} \nonumber \\
&\leq& 2 + e^{\frac{z}{2}}\frac{x}{\sinh\frac{x}{2}}\frac{1}{\cosh \frac{y}{2}} \nonumber\\
&\leq& 2 + 100(1+x) e^{\frac{z}{2}}e^{-\frac{x+y}{2}}.\nonumber
\end{eqnarray}
So combining \eqref{estimation R,D--x/R case 1} and \eqref{estimation R,D--x/R case 2} we have
\begin{equation*}
\frac{x}{\mathcal R(x,y,z)} \leq 100(1+x)(1+e^{\frac{z}{2}}e^{-\frac{x+y}{2}}).
\end{equation*}

Now assume $x+y>z$ and consider the following subcases of Case 1. 

\textbf{Case 1a:} $\frac{e^x -1}{1 + e^{\frac{x+z}{2}} \frac{1}{\cosh \frac{y}{2}} } \geq 1$ (which implies $e^x\geq 2$) and $1\geq e^{\frac{x+z}{2}} \frac{1}{\cosh \frac{y}{2}}$. Then by \eqref{estimation R,D--R geq} we have
\begin{eqnarray}\label{estimation R,D--x/R moreover case 1.1}
\frac{x}{\mathcal R(x,y,z)}
&\leq& \frac{x}{\log\big( \frac{e^x+1}{2} \big) } \\
&\leq& 100.\nonumber
\end{eqnarray}

\textbf{Case 1b:} $\frac{e^x -1}{1 + e^{\frac{x+z}{2}} \frac{1}{\cosh \frac{y}{2}} } \geq 1$ (which implies $e^x\geq 2$) and $e^{\frac{x+z}{2}} \frac{1}{\cosh \frac{y}{2}}\geq1$. Then by \eqref{estimation R,D--R geq} we have
\begin{eqnarray}\label{estimation R,D--x/R moreover case 1.2}
\frac{x}{\mathcal R(x,y,z)}
&\leq& \frac{x}{\log\big(1+ \frac{e^x-1}{2e^{\frac{x+z}{2}} \frac{1}{\cosh \frac{y}{2}}} \big) }\\
&\leq& \frac{x}{\log\big(1+ \frac{e^x-1}{4e^x} e^{\frac{x+y-z}{2}} \big) } \nonumber\\
&\leq& 100 \frac{x}{x+y-z} \nonumber
\end{eqnarray}
where in the last inequality we apply the elementary inequality $1+ae^{x}\geq e^{ax}$ for any $a\in (0,1)$ and $x>0$.

So combining \eqref{estimation R,D--x/R case 2}, \eqref{estimation R,D--x/R moreover case 1.1} and \eqref{estimation R,D--x/R moreover case 1.2} we have
\begin{eqnarray*}
\frac{x}{\mathcal R(x,y,z)}
&\leq& 100+100 \frac{x}{x+y-z} + 2 + 100(1+x) e^{\frac{z}{2}}e^{-\frac{x+y}{2}} \\
&\leq& 202+ 200\frac{x}{x+y-z}.
\end{eqnarray*}

As for $\mathcal D(x,y,z)$, we have
\begin{equation*}
\mathcal D(x,y,z) = 2\log\big( 1+\frac{2\sinh\frac{x}{2}} {e^{-\frac{x}{2}}+e^{\frac{y+z}{2}}} \big).
\end{equation*}

\textbf{Case 1c:} $\frac{2\sinh\frac{x}{2}} {e^{-\frac{x}{2}}+e^{\frac{y+z}{2}}}\leq 1$. Then by the fact that $\log(1+t)\geq \frac{t}{2}$ for $0<t\leq 1$, we have
\begin{eqnarray}\label{estimation R,D--x/D case 1}
\frac{x}{\mathcal D(x,y,z)}
&\leq& \frac{x}{2\cdot \frac{1}{2}\frac{2\sinh\frac{x}{2}} {e^{-\frac{x}{2}}+e^{\frac{y+z}{2}}}}\\
&=& xe^{-\frac{x}{2}}\frac{1}{2\sinh\frac{x}{2}} + xe^{\frac{y+z}{2}}\frac{1}{2\sinh\frac{x}{2}}. \nonumber
\end{eqnarray}

\textbf{Case 1d:} $\frac{2\sinh\frac{x}{2}} {e^{-\frac{x}{2}}+e^{\frac{y+z}{2}}}> 1$. Then we have
\begin{equation}\label{estimation R,D--x/D case 2}
\frac{x}{\mathcal D(x,y,z)} \leq \frac{1}{2\log2}x.
\end{equation}

So combining \eqref{estimation R,D--x/D case 1} and \eqref{estimation R,D--x/D case 2} we have
\begin{equation*}
\frac{x}{\mathcal D(x,y,z)} \leq 100(1+x)(1+e^{\frac{y+z}{2}}e^{-\frac{x}{2}})
\end{equation*}
 (one may prove this inequality for $0<x\leq 1$ and $x>1$ respectively), as required.
\end{proof}


\section{Lower bound} \label{section lower bound}

In this section, we will show the easier part of Theorems \ref{main} and \ref{cor L1}, namely the lower bound. More precisely, we show that

\begin{proposition}\label{prop lower bound}
Let $\omega(g)$ be a function satisfying \eqref{eq-omega}. Then we have
\begin{equation*}
\limg\Prob\big(X\in\M_g\,;\, \ell_{\sys}^{\rm sep}(X) \geq 2\log g - 4\log \log g - \omega(g)  \big) = 1
\end{equation*}
and
\begin{equation*}
\limg\Prob\big(X\in\M_g\,;\, \sL_1(X) \geq 2\log g - 4\log \log g - \omega(g)  \big) = 1.
\end{equation*}
\end{proposition}

Since $\ell_{\sys}^{\rm sep}(X) \geq \sL_1(X)$, it suffices to prove the second limit. We follow the method in \cite[Section 4.3]{Mirz13} for this part.

Let $L>0$ and assume $\sL_1(X) \leq L$. Then there exists a simple closed multi-geodesic of length $\leq L$ separating $X$ into $S_{g_0,k}\cup S_{g-g_0-k+1,k}$ for some $(g_0,k)$ with $|\chi(S_{g_0,k})|\leq \frac{1}{2}|\chi(S_g)|=g-1$. That is,
\begin{equation*}
\sum_{(g_0,k)\,;\,1\leq 2g_0-2+k\leq g-1} N_{g_0,k}(X,L)\geq 1.
\end{equation*}
So we have
\begin{eqnarray}\label{prob(L1 leq L) leq}
&&\Prob\big(X\in\M_g\,;\, \sL_1(X) \leq L  \big) \\
&&\leq \Prob\left( \sum_{(g_0,k);\,  1\leq 2g_0-2+k\leq g-1} N_{g_0,k}(X,L)\geq 1  \right) \nonumber\\
&&\leq \sum_{(g_0,k);\, 1\leq 2g_0-2+k\leq g-1} \E[N_{g_0,k}(X,L)], \nonumber
\end{eqnarray}
where the last equality uses the fact that $\mathbb{P}(N\geq1)\leq \mathbb{E}(N)$ for any $\Z_{\geq0}$-valued random variable $N$.

By Mirzakhani's Integration Formula (see Theorem \ref{Mirz int formula}), we have
\begin{eqnarray}\label{E[N_g0,k]}
&& \E[N_{g_0,k}(X,L)]
=\frac{1}{V_g} \frac{2^{-M}}{|\Sym|} \int_{\R^k_{\geq 0}} \mathbf 1_{[0,L]}(x_1+\cdots+x_k)\\
&& \times V_{g_0,k}(x_1,\cdots,x_k) V_{g-g_0-k+1,k}(x_1,\cdots,x_k) x_1\cdots x_k dx_1\cdots dx_k  \nonumber
\end{eqnarray}
where $M=1$ if $(g_0,k)=(1,1)$ and $M=0$ otherwise, and $|\Sym|\geq k!$ in general and $|\Sym|=k!$ if $g>2$ and $(g_0,k)=(1,1)$ or $(0,3)$.

Then we split the proof of Proposition \ref{prop lower bound} by calculating the quantity $\E[N_{g_0,k}(X,L)]$ for three different cases.

\begin{lemma}\label{E[N]}
For $(g_0,k)=(1,1)$ or $(0,3)$, we have
\begin{equation*}
\E[N_{1,1}(X,L)] = \frac{1}{384\pi^2} L^2 e^{\frac{L}{2}} \tfrac{1}{g} \left(1+O\left(\tfrac{1}{g}\right)\right)  \left(1+O\left(\tfrac{1}{L}\right)\right) \left(1+O\left(\tfrac{L^2}{g}\right)\right)
\end{equation*}
and
\begin{equation*}
\E[N_{0,3}(X,L)] = \frac{1}{48\pi^2} L^2 e^{\frac{L}{2}} \tfrac{1}{g} \left(1+O\left(\tfrac{1}{g}\right)\right) \left(1+O(\tfrac{1}{L})\right) \left(1+O(\tfrac{L^2}{g})\right),
\end{equation*}
where the implied constants are independent of $L$ and $g$.
\end{lemma}

\begin{proof}
First we consider the case $(g_0,k)=(1,1)$. By using Theorem \ref{Mirz vol lemma 0} and Lemma \ref{Mirz vol lemma 1}, \ref{MP vol lemma} and Equation \eqref{E[N_g0,k]}, we have
\begin{eqnarray*}
\E[N_{1,1}(X,L)]&=&\frac{1}{V_g} \frac{1}{2} \int_0^L V_{1,1}(x)V_{g-1,1}(x)xdx  \\
&=& \frac{1}{2V_g}\int_0^L\frac{1}{48}(x^2+4\pi^2)x\frac{\sinh(x/2)}{x/2} dx \times V_{g-1,1} \big(1+O(\tfrac{L^2}{g})\big) \\
&=& \frac{1}{48} L^2 e^{\frac{L}{2}} \frac{V_{g-1,1}}{V_g} \left(1+O(\tfrac{1}{L})\right) \left(1+O(\tfrac{L^2}{g})\right) \\
&=& \frac{1}{384\pi^2} L^2 e^{\frac{L}{2}} \frac{1}{g} \left(1+O(\tfrac{1}{g})\right) \left(1+O(\tfrac{1}{L})\right) \left(1+O(\tfrac{L^2}{g})\right).
\end{eqnarray*}

Similarly for $(g_0,k)=(0,3)$, we have
\begin{eqnarray*}
\E[N_{0,3}(X,L)]
&=& \frac{1}{V_g} \frac{1}{3!} \int_{0\leq x+y+z \leq L}  \frac{\sinh(x/2)}{x/2}\frac{\sinh(y/2)}{y/2}\frac{\sinh(z/2)}{z/2} \\
& & xyz dxdydz \times V_{g-2,3} \big(1+O(\tfrac{L^2}{g})\big) \\
&=& \frac{1}{6} L^2 e^{\frac{L}{2}} \frac{V_{g-2,3}}{V_g} \left(1+O(\tfrac{1}{L})\right) \left(1+O(\tfrac{L^2}{g})\right) \\
&=& \frac{1}{48\pi^2} L^2 e^{\frac{L}{2}} \frac{1}{g} \left(1+O(\tfrac{1}{g})\right) \left(1+O(\tfrac{1}{L})\right) \left(1+O(\tfrac{L^2}{g})\right),
\end{eqnarray*}
as required.
\end{proof}

\begin{rem*}\label{remark:comes from}
	The dominating term $L^2e^\frac{L}{2}\frac{1}{g}$ in both expressions in Lemma \ref{E[N]} is where the upper and lower bounds $2\log g-4\log \log g\pm\omega(g)$ in Theorems \ref{main} and \ref{cor L1} come from. In fact, a function $L(g)$ in the variable $g\in\{2,3,\cdots\}$ has the form $2\log g-4\log\log g+\omega(g)$, with $\omega(g)$ satisfying the assumption of Theorems \ref{main}, if and only if 
	$$
	\lim_{g\to\infty}L(g)^2e^\frac{L(g)}{2}\frac{1}{g}\to+\infty,\quad L(g)^2e^\frac{L(g)}{2}\frac{1}{g}=O\big((\log g)^\epsilon\big)
	$$ 
	for any $\epsilon>0$. Similarly, $L(g)$ has the form $2\log g-4\log\log g-\omega(g)$ if and only if $\lim_{g\to\infty}L(g)^2e^\frac{L(g)}{2}\frac{1}{g}\to0$ and $L(g)^2e^\frac{L(g)}{2}\frac{1}{g}\geq C(\log g)^{-\epsilon}$ when $g$ is large enough for any $C,\epsilon>0$.
\end{rem*}

\begin{lemma}\label{sum chi=m E[N]}
For any given positive integer $m$, there exists a constant $c(m)>0$ independent of $L$ and $g$ such that
\begin{equation*}
\sum_{|\chi(S_{g_0,k})| = m} \E[N_{g_0,k}(X,L)] \leq c(m) (1+L^{3m-1})e^{\frac{L}{2}}\frac{1}{g^m}
\end{equation*}
where the summation is over all pairs $(g_0,k)$ satisfying $g_0\geq 0$, $k\geq 1$, and $$|\chi(S_{g_0,k})| = m.$$
\end{lemma}

\begin{proof}
Assume $2g_0-2+k=|\chi(S_{g_0,k})| = m$. Then both $g_0$ and $k$ are bounded from above by $m+2$. By Theorem \ref{Mirz vol lemma 0} of Mirzakhani we know that $V_{g_0,k}(x_1,\cdots,x_k)$ is a polynomial of degree $6g_0-6+2k$ with coefficients bounded by some constant only depending on $m$. So when $0\leq x_1+\cdots+x_k\leq L$ we have
\begin{equation*}
V_{g_0,k}(x_1,\cdots,x_k) \leq c'(m) (1+L^{6g_0-6+2k})
\end{equation*}
for some constant $c'(m)>0$ only depending on $m$. Then by Lemma \ref{Mirz vol lemma 1}, \ref{MP vol lemma} and Equation \eqref{E[N_g0,k]} we have
\begin{eqnarray*}
\E[N_{g_0,k}(X,L)]
&\leq& \frac{1}{V_g} \int_{0\leq x_1+\cdots+x_k \leq L} c'(m) (1+L^{6g_0-6+2k}) \\
& & \frac{\sinh(x_1/2)}{x_1/2}\cdots\frac{\sinh(x_k/2)}{x_k/2} x_1\cdots x_k dx_1\cdots dx_k V_{g-g_0-k+1,k}  \\
&\leq& c'(m) (1+L^{6g_0-7+3k}) e^{\frac{L}{2}} \frac{V_{g-g_0-k+1,k}}{V_g}   \\
&\leq& c'(m) (1+L^{3m-1}) e^{\frac{L}{2}} \frac{1}{g^m} .
\end{eqnarray*}
Since $g_0,k$ are bounded above by $m+2$, there exists a constant $c(m)>0$ only depending on $m$ such that
\begin{equation*}
\sum_{|\chi(S_{g_0,k})| = m} \E[N_{g_0,k}(X,L)] \leq c(m) (1+L^{3m-1})e^{\frac{L}{2}}\frac{1}{g^m}
\end{equation*}
as desired.
\end{proof}

\begin{lemma}\label{sum chi geq m E[N]}
For any given positive integer $m$, there exists a constant $c(m)>0$ independent of $L$ and $g$ such that
\begin{equation*}
\sum_{m\leq|\chi(S_{g_0,k})| \leq g-1} \E[N_{g_0,k}(X,L)] \leq c(m) e^{2L}\frac{1}{g^m}
\end{equation*}
where the summation is over all pairs $(g_0,k)$ satisfying $g_0\geq 0$, $k\geq 1$, and $$m\leq |\chi(S_{g_0,k})| \leq g-1.$$
\end{lemma}

\begin{proof}
First by Part $(1)$ of Lemma \ref{Mirz vol lemma 1} we know that
\begin{equation*}
V_{g_0,k}(x_1,\cdots,x_k)\leq e^{\frac{x_1+\cdots+x_k}{2}}V_{g_0,k}
\end{equation*}
and
\begin{equation*}
V_{g-g_0-k+1,k}(x_1,\cdots,x_k)\leq e^{\frac{x_1+\cdots+x_k}{2}}V_{g-g_0-k+1,k}.
\end{equation*}

\noindent Then by \eqref{E[N_g0,k]} we have
\begin{eqnarray*}
\E[N_{g_0,k}(X,L)]
&\leq& \frac{1}{V_g} \frac{1}{k!} \int_{0\leq \sum x_i \leq L} e^{x_1+\cdots+x_k} x_1\cdots x_k dx_1\cdots dx_k V_{g_0,k} V_{g-g_0-k+1,k} \\
&\leq& \frac{1}{k!} \frac{V_{g_0,k} V_{g-g_0-k+1,k}}{V_g} e^L \int_{0\leq \sum x_i \leq L} x_1\cdots x_k dx_1\cdots dx_k \\
&=& \frac{1}{k!} \frac{L^{2k}}{(2k)!} e^L \frac{V_{g_0,k} V_{g-g_0-k+1,k}}{V_g}.
\end{eqnarray*}

\noindent Recall that Part (2) of Lemma \ref{Mirz vol lemma 1} says that for any $g,n\geq 0$
\[V_{g-1,n+4}\leq V_{g,n+2}.\]
So we have
$$V_{g_0,k} \leq V_{g_0+\frac{k-k'}{2},k'} \quad \emph{and} \quad V_{g-g_0-k+1,k} \leq V_{g-g_0-k+1+\frac{k-k'}{2},k'}$$ where $k' \in \{1,2,3\}$ with even $k-k'\geq0$. For any fixed integer $k>0$, we consider the summation over $g_0$ with $m\leq|\chi(S_{g_0,k})| \leq g-1$. By Lemma \ref{Mirz vol lemma 2} we have
\begin{equation*}
\sum_{g_0;\,m\leq |\chi(S_{g_0,k})| \leq g-1} \E[N_{g_0,k}(X,L)] \leq c(m) \frac{1}{k!} \frac{L^{2k}}{(2k)!} e^L \frac{1}{g^{m}}
\end{equation*}
for some constant $c(m)>0$ only depending on $m$. Then the total summation satisfies that
\begin{eqnarray*}
\sum_{m\leq |\chi(S_{g_0,k})| \leq g-1} \E[N_{g_0,k}(X,L)]
&\leq& \sum_{k\geq 1} c(m) \frac{1}{k!} \frac{L^{2k}}{(2k)!} e^L \frac{1}{g^{m}} \\
&\leq& c(m) e^{2L} \frac{1}{g^{m}}
\end{eqnarray*}
because $\sum_{k\geq 1} \frac{1}{k!} \frac{L^{2k}}{(2k)!} \leq e^L$. This finishes the proof. 
\end{proof}

Now we are ready to prove Proposition \ref{prop lower bound}.
\bp [Proof of Proposition \ref{prop lower bound}]
First by Equation \eqref{prob(L1 leq L) leq}, Lemma \ref{E[N]}, \ref{sum chi=m E[N]} and \ref{sum chi geq m E[N]},  we have that for large $g>0$,
\begin{eqnarray*}
&&\Prob\big(X\in\M_g\,;\, \sL_1(X) \leq L  \big)
\leq \sum_{(g_0,k);\,1\leq |\chi(S_{g_0,k})|\leq g-1} \E[N_{g_0,k}(X,L)] \\
&&=\left(\E[N_{1,1}(X,L)]+\E[N_{0,3}(X,L)]\right)\\
&&+ \sum_{m=2}^{10} \sum_{|\chi(S_{g_0,k})| = m}  \E[N_{g_0,k}(X,L)]+\sum_{11\leq |\chi(S_{g_0,k})|\leq g-1}  \E[N_{g_0,k}(X,L)]\\
&&\leq c L^2 e^{\frac{L}{2}}\frac{1}{g} + \sum_{m=2}^{10} cL^{3m-1}e^{\frac{L}{2}}\frac{1}{g^m} + ce^{2L}\frac{1}{g^{11}}
\end{eqnarray*}
for some uniform constant $c>0$. Now for $$L= 2\log g - 4\log \log g - \omega(g),$$ we have
\[L^2 e^{\frac{L}{2}}\frac{1}{g}=O(e^{-\frac{\omega(g)}{2}}),\]

\[\sum_{m=2}^{10} L^{3m-1}e^{\frac{L}{2}}\frac{1}{g^m}=O(\frac{(\log g)^{29}}{g}),\]
and
\[ \frac{e^{2L}}{g^{11}}=O(\frac{1}{g^7}).\]
Recall that $\omega(g)\to \infty$ as $g\to \infty$. Hence we get
\begin{eqnarray*}
\lim \limits_{g\to \infty}\Prob\big(X\in\M_g\,;\, \sL_1(X) \leq 2\log g - 4\log \log g - \omega(g)  \big)=0,
\end{eqnarray*}
which implies that
\[\lim \limits_{g\to \infty}\Prob\big(X\in\M_g\,;\, \sL_1(X) \geq 2\log g - 4\log \log g - \omega(g)  \big)=1,\]
as desired.
\ep

Actually the argument above also leads to Proposition \ref{lower bound for chi geq 2}, which will be applied later. First we recall the following definition generalizing $\mathcal{L}_1$ in the Introduction. For any integer $m\in [1,g-1]$ and $X\in \sM_g$, 
\[\mathcal{L}_{1,m}(X):=\min_{\Gamma} \ell_{\Gamma}(X)\]
where the minimum runs over all simple closed multi-geodesics $\Gamma$ separating $X$ into $S_{g_1,k}\cup S_{g_2,k}$ with
\[|\chi(S_{g_1,k})|\geq |\chi(S_{g_2,k})|\geq m.\]

Now we are ready to prove Proposition \ref{lower bound for chi geq 2}.

\begin{proposition}[=Proposition \ref{lower bound for chi geq 2}] \label{lower bound for chi geq 2-1}
Let $\omega(g)$ be a function satisfying \eqref{eq-omega}. Then we have that for any fixed $m\geq 1$ independent of $g$,
\begin{equation*}
\limg\Prob\left(X\in\M_g\,;\, \mathcal{L}_{1,m}(X) \geq  2m\log g - (6m-2)\log\log g -\omega(g)\right) = 1.
\end{equation*}
\end{proposition}

\begin{proof}
The proof is almost the same as the proof of Proposition \ref{prop lower bound}. First we have that for large $g>0$,
\begin{eqnarray*}
&&\Prob\big(X\in\M_g\,;\, \sL_{1,m}(X) \leq L  \big)
\leq \sum_{(g_0,k);\,m\leq |\chi(S_{g_0,k})|\leq g-1} \E[N_{g_0,k}(X,L)] \\
&&=\sum_{|\chi(S_{g_0,k})| = m}  \E[N_{g_0,k}(X,L)]+  \sum_{m+1\leq |\chi(S_{g_0,k})|\leq 10m}  \E[N_{g_0,k}(X,L)]\\
&&+\sum_{10m+1\leq |\chi(S_{g_0,k})|\leq g-1}  \E[N_{g_0,k}(X,L)].
\end{eqnarray*}

\noindent Now for $$L= 2m\log g - (6m-2)\log \log g - \omega(g),$$  by Lemma \ref{E[N]}, \ref{sum chi=m E[N]} and \ref{sum chi geq m E[N]} we have
\[\sum_{|\chi(S_{g_0,k})| = m}  \E[N_{g_0,k}(X,L)]=O(L^{3m-1} e^{\frac{L}{2}}\frac{1}{g^m})=O(e^{-\frac{\omega(g)}{2}}),\]

\[\sum_{m+1\leq |\chi(S_{g_0,k})|\leq 10m}  \E[N_{g_0,k}(X,L)]=O(\sum_{j=m+1}^{10m} L^{3j-1}e^{\frac{L}{2}}\frac{1}{g^j})=O(\frac{(\log g)^{30m-1}}{g}),\]
and
\[\sum_{10m+1\leq |\chi(S_{g_0,k})|\leq g-1}  \E[N_{g_0,k}(X,L)]= O(\frac{e^{2L}}{g^{10m+1}})=O(\frac{1}{g^{6m+1}}).\]
Recall that $\omega(g)\to \infty$ as $g\to \infty$. Hence we get
\begin{eqnarray*}
\lim \limits_{g\to \infty}\Prob\big(X\in\M_g\,;\, \sL_{1,m}(X) \leq 2m\log g - (6m-2)\log \log g - \omega(g) \big)=0
\end{eqnarray*}
which implies that
\[\lim \limits_{g\to \infty}\Prob\big(X\in\M_g\,;\, \sL_{1,m}(X) \geq 2m\log g - (6m-2)\log \log g - \omega(g)  \big)=1\]
as desired.
\end{proof}

\begin{rem*}
The $m=1$ case of Proposition \ref{lower bound for chi geq 2} is exactly Proposition \ref{prop lower bound}.
\end{rem*}

\section{Upper bound}\label{section upper bound}

In this section, we will show the upper bound in Theorem \ref{main} and \ref{cor L1}. We begin with the following definition.

\begin{definition}
Assume $\omega(g)$ is a function satisfying \eqref{eq-omega}. For any $X\in \sM_g$, we say $X\in \mathcal A(\omega(g))$ if there exists a simple closed geodesic $\gamma$ on $X$ such that
\begin{enumerate}
\item $\gamma$ separates $X$ into $S_{1,1}\cup S_{g-1,1}$;

\item the length $\ell_{\gamma}(X) \leq 2\log g- 4\log \log g +\omega(g)$.
\end{enumerate}
\end{definition}

Now we are ready to state the upper bound of Theorem \ref{main} which is also the essential part of this paper.
\begin{theorem}\label{prop upper bound} Let $\omega(g)$ be a function satisfying \eqref{eq-omega}. Then we have
\begin{equation*}
\limg\Prob\big(X\in \sM_g\,;\, X\in\mathcal A(\omega(g)) \big)= 1.
\end{equation*}
\end{theorem}

\subsection{Proofs of Theorem \ref{main} and \ref{cor L1}}
We postpone the proof of Theorem \ref{prop upper bound} to the next subsections and give here the proof of Theorem \ref{main} and \ref{cor L1} assuming Theorem \ref{prop upper bound}. 
\begin{theorem}[=Theorem \ref{main}]\label{main-1}
Let $\omega(g)$ be a function satisfying \eqref{eq-omega}. Consider the following two conditions defined for all $X\in\M_g$:
\begin{itemize}
\item[(a).] $|\ell_{\sys}^{\rm sep}(X)-(2\log g - 4\log \log g)| \leq \omega(g)$;

\item[(b).] $\ell_{\sys}^{\rm sep}(X)$ is achieved by a simple closed geodesic separating $X$ into $S_{1,1}\cup S_{g-1,1}$.
\end{itemize}
Then we have
$$
\lim \limits_{g\to \infty} \Prob\left(X\in \M_g \,;\, \textit{$X$ satisfies $(a)$ and $(b)$} \right)=1.
$$ 
\end{theorem}

\bp 
Let $m=2$ in Proposition \ref{lower bound for chi geq 2} we get
\begin{equation*}
\limg\Prob\big(X\in\M_g \,;\, \mathcal{L}_{1,2}(X)>  3.9\log g)=1.
\end{equation*}
 Set
\[\mathcal A'(\omega(g)):=\{X\in \sM_g \,;\, \ell_{\sys}^{\rm sep}(X) \geq 2\log g - 4\log \log g - \omega(g) \}\]
and
\[\mathcal A''(g):=\{X\in \sM_g \,;\, \mathcal{L}_{1,2}(X)>  3.9\log g \}.\]
Then for any $X\in \mathcal A(\omega(g))\cap \mathcal A''(g)$ and large enough $g>0$, the quantity $\ell_{\sys}^{\rm sep}(X)$ is realized by a simple closed geodesic separating $X$ into $S_{1,1}\cup S_{g-1,1}$. For any $X\in \mathcal A(\omega(g))\cap \mathcal A'(\omega(g))$, we have
\[|\ell_{\sys}^{\rm sep}(X)-(2\log g -4\log \log g)|\leq \omega(g).\]
Thus, it follows by Proposition \ref{prop lower bound}, Proposition \ref{lower bound for chi geq 2} for $m=2$ and Theorem \ref{prop upper bound} that as $g\to \infty$, $\Prob\left(A(\omega(g))\right),\Prob\left(A'(\omega(g))\right)$ and $\Prob\left(A''(\omega(g))\right)$ all tend to $1$. Therefore, we have
\[\limg\Prob\big(X\in\M_g\,;\, X\in \mathcal A(\omega(g))\cap \mathcal A'(\omega(g))\cap \mathcal A''(g)  \big)=1,\]
as required.
\ep

\begin{theorem}[=Theorem \ref{cor L1}]\label{cor L1-1}
Let $\omega(g)$ be a function satisfying \eqref{eq-omega}. Consider the following two conditions defined for all $X\in\M_g$:
\begin{itemize}
\item[(e).]  $|\sL_1(X)-(2\log g - 4\log \log g)| \leq \omega(g)$;

\item[(f).]  $\sL_1(X)$ is achieved by either a simple closed geodesic separating $X$ into $S_{1,1}\cup S_{g-1,1}$ or three simple closed geodesics separating $X$ into $S_{0,3}\cup S_{g-2,3}$.
\end{itemize}
Then we have
$$
\lim \limits_{g\to \infty} \Prob\left(X\in \M_g\,;\, \textit{$X$ satisfies $(e)$ and $(f)$} \right)=1.
$$ 
\end{theorem}

\bp 
The proof is similar as the proof of Theorem \ref{main}. Let $m=2$ in Proposition \ref{lower bound for chi geq 2} we get
\begin{equation*}
\limg\Prob\big(X\in\M_g\,;\, \mathcal{L}_{1,2}(X)>  3.9\log g)=1.
\end{equation*}
 Set
\[\mathcal A'(\omega(g)):=\{X\in \sM_g \,;\, \mathcal{L}_1(X) \geq 2\log g - 4\log \log g - \omega(g) \}\]
and
\[\mathcal A''(g):=\{X\in \sM_g \,;\, \mathcal{L}_{1,2}(X)>  3.9\log g \}.\]
Then for any $X\in \mathcal A(\omega(g))\cap \mathcal A''(g)$ and large enough $g>0$, the quantity $\mathcal L_1(X)$ is realized by either a simple closed geodesic separating $X$ into $S_{1,1}\cup S_{g-1,1}$ or three simple closed geodesic separating $X$ into $S_{0,3}\cup S_{g-2,3}$. For any $X\in \mathcal A(\omega(g))\cap \mathcal A'(\omega(g))$, we have
\[|\mathcal L_1(X)-(2\log g -4\log \log g)|\leq \omega(g).\]
Thus, it follows by Proposition \ref{prop lower bound}, Proposition \ref{lower bound for chi geq 2} for $m=2$ and Theorem \ref{prop upper bound} that
\[\limg\Prob\big(X\in\M_g\,;\,  X\in \mathcal A(\omega(g))\cap \mathcal A'(\omega(g))\cap \mathcal A''(g)  \big)=1,\]
as required.
\ep

\begin{rem*}
It is interesting to study whether $\mathcal{L}_1(X)$ is realized just by a simple closed geodesic separating $X$ into $S_{1,1}\cup S_{g-1,1}$ on a generic point $X\in \sM_g$. Or does the following limit hold:
\[\limg\Prob\big(X\in\M_g\,;\,   \mathcal{L}_1(X)=\ell_{\sys}^{\rm sep}(X)  \big)=1?\]
\end{rem*}

Set
\begin{equation}
L=L(g)=2\log g -4\log \log g +\omega(g)
\end{equation}
where $\omega(g)$ is given as above in \eqref{eq-omega}. In the following arguments we always assume that $g$ is large enough. So $L$ is also large enough.\\

In order to prove Theorem \ref{prop upper bound}, it suffices to show that
\begin{equation} \label{N(1,1)=0}
\limg\Prob\big(X\in\M_g\,;\, N_{1,1}(X,L)=0 \big)= 0.
\end{equation}

For each $X\in\M_g$, we denote $\sN_{1,1}(X,L)$ to be the set of simple closed geodesics on $X$ which separate $X$ into $S_{1,1}\cup S_{g-1,1}$ and has length $\leq L$. Then 
$$N_{1,1}(X,L) = \# \sN_{1,1}(X,L).$$
Instead of $\sN_{1,1}(X,L)$, we consider the subset $\sN^*_{1,1}(X,L)$ which is defined as follows.

\begin{definition}\label{N*1,1}
We denote
\begin{equation*}
\sN^*_{1,1}(X,L):= \left\{ \alpha\in\sN_{1,1}(X,L)\,;\
\parbox[l]{3.6cm}{$\forall\alpha\neq\gamma\in\sN_{1,1}(X,L)$,\\either $\alpha \cap \gamma =\emptyset$ or\\  $X_{\alpha\gamma}$ is of type $S_{1,2}$}\right\}
\end{equation*}
and
\begin{equation*}
N^*_{1,1}(X,L) := \#\sN^*_{1,1}(X,L),
\end{equation*}
where $X_{\alpha\gamma}$ is defined in section \ref{section union}.
\end{definition}

\noindent Since $N^*_{1,1}(X,L)\leq N_{1,1}(X,L)$, we clearly have that
\begin{equation*}
\Prob\big(X\in \M_g\,;\, N_{1,1}(X,L)=0 \big) \leq \Prob\big(X\in \M_g;\ N^*_{1,1}(X,L)=0 \big).
\end{equation*}
We will show the following limit which implies \eqref{N(1,1)=0}.
\begin{equation}\label{N*1,1=0}
\limg\Prob\big(X\in \M_g\,;\, N^*_{1,1}(X,L)=0 \big)= 0.
\end{equation}

\begin{rem*}
The purpose to study $N^*_{1,1}(X,L)$ instead of $N_{1,1}(X,L)$ is to simplify certain estimations. Actually the following method also works for $N_{1,1}(X,L)$ by adding more detailed discussions.
\end{rem*}

\subsection{Bounding probability by expectation}
For any nonnegative integer-valued random variable $N$, by Cauchy-Schwarz inequality we have
$$
\mathbb E[N]^2=\mathbb E\big[N\cdot\mathbf{1}_{\{N>0\}}\big]^2\leq \mathbb E[N^2]\cdot\mathbb E \big[\mathbf{1}_{\{N>0\}}^2\big]=\mathbb E[N^2] \cdot\mathbb{P}(N>0).
$$
So we have
$$\mathbb P(N>0)\geq \frac{\mathbb E[N]^2}{\mathbb E[N^2]}.$$ Then since the variance $\mathop{\rm Var} [N] = \mathbb E[N^2] - \mathbb E[N]^2$ is nonnegative, we have
$$
\mathbb{P}(N=0)\leq \frac{\mathbb E[N^2]-\mathbb E[N]^2}{\mathbb E[N^2]} \leq \frac{\mathbb E[N^2] -\mathbb E[N]^2}{\mathbb E[N]^2}.
$$
Applying this to $N^*_{1,1}(X,L)$, we get 
\begin{equation}\label{prob(N*=0) leq}
\Prob\big(N^*_{1,1}(X,L)=0 \big)
\leq \frac{\E[(N^*_{1,1}(X,L))^2] - \E[N^*_{1,1}(X,L)]^2}{\E[N^*_{1,1}(X,L)]^2}.
\end{equation}

In order to control the RHS above, the most essential part is to study $(N^*_{1,1}(X,L))^2$. We decompose it into three different parts as follows. We define

\begin{definition}
\begin{equation*}
\mathcal Y^*(X,L):= \left\{ (\alpha,\beta)\in\sN^*_{1,1}(X,L)\times\sN^*_{1,1}(X,L) \,;\,  \alpha\neq \beta, \alpha\cap\beta=\emptyset \right\},
\end{equation*}
\begin{equation*}
\mathcal Z^*(X,L):= \left\{ (\alpha,\beta)\in\sN^*_{1,1}(X,L)\times\sN^*_{1,1}(X,L) \,;\,  \alpha\neq \beta, \alpha\cap\beta\neq\emptyset \right\}.
\end{equation*}
Denote
\begin{equation*}
Y^*(X,L) := \#\mathcal Y^*(X,L),
\end{equation*}
\begin{equation*}
Z^*(X,L) := \#\mathcal Z^*(X,L).
\end{equation*}
Then we have $$N^*_{1,1}(X,L)^2=N^*_{1,1}(X,L)+ Y^*(X,L) +Z^*(X,L).$$
\end{definition}

\noindent Inserting this decomposition into the \rm{RHS} of \eqref{prob(N*=0) leq} we get 
\begin{eqnarray}\label{prob(N*=0) leq 3 parts}
&&\Prob\big(N^*_{1,1}(X,L)=0 \big)\leq \frac{1}{\E[N^*_{1,1}(X,L)]}\\
&& + \frac{\E[Y^*(X,L)] - \E[N^*_{1,1}(X,L)]^2}{\E[N^*_{1,1}(X,L)]^2}+ \frac{\E[Z^*(X,L)]}{\E[N^*_{1,1}(X,L)]^2}.\nonumber
\end{eqnarray}

In the following subsections we will show that each of the three terms on the RHS of \eqref{prob(N*=0) leq 3 parts} goes to $0$ as $g\to \infty$ for $L=L(g)=2\log g-4\log\log g+\omega(g)$, which in particular implies Theorem \ref{prop upper bound}. More precisely,
\begin{equation}\label{tend to 0 (1)}\tag{A}
\limg\frac{1}{\E[N^*_{1,1}(X,L)]} = 0,
\end{equation}
\begin{equation}\label{tend to 0 (2)}\tag{B}
\limg\frac{\E[Y^*(X,L)] - \E[N^*_{1,1}(X,L)]^2}{\E[N^*_{1,1}(X,L)]^2} = 0
\end{equation}
and
\begin{equation}\label{tend to 0 (3)}\tag{C}
\limg\frac{\E[Z^*(X,L)]}{\E[N^*_{1,1}(X,L)]^2} = 0.
\end{equation}

\begin{rem*}
The proofs of \eqref{tend to 0 (1)} and \eqref{tend to 0 (2)} are similar: we use $\E[N_{1,1}]$ and $\E[Y]$ to approximate $\E[N^*_{1,1}]$ and $\E[Y^*]$ respectively (we will define $Y(X,L)$ in Subsection \ref{subsec:proof of (B)}). For the proof of \eqref{tend to 0 (3)}, we will control the number of certain types of simple closed geodesics by using Mirzakhani's generalized McShane identity.
\end{rem*}

We will prove \eqref{tend to 0 (1)}, \eqref{tend to 0 (2)} and \eqref{tend to 0 (3)} in the following subsections.


\subsection{Proof of \eqref{tend to 0 (1)}}

Recall that $L=L(g)=2\log g -4\log \log g +\omega(g)$ goes to $\infty$ as $g\to \infty$. By Lemma \ref{E[N]} we have
$$\lim \limits_{g\to \infty}\E[N_{1,1}(X,L)]=\infty.$$ We will show that $\E[N^*_{1,1}(X,L)]$ is close to $\E[N_{1,1}(X,L)]$ for large $g>0$. More precisely,

\begin{proposition}\label{N-N*}
With the notations as above, we have
\begin{equation*}
\limg \big(\E[N_{1,1}(X,L)] - \E[N^*_{1,1}(X,L)] \big)=0.
\end{equation*}
In particular, Equation \eqref{tend to 0 (1)} holds.
\end{proposition}

We split the proof into several parts. We always assume that $g>0$ is large enough.

By definition, $N_{1,1}(X,L) - N^*_{1,1}(X,L) \geq0$. Assume that $$\gamma\in\sN_{1,1}(X,L) \setminus \sN^*_{1,1}(X,L).$$
By definition of $N^*_{1,1}(X,L)$ and Lemma \ref{area U small} we know that there exists a simple closed geodesic $\alpha\in \sN_{1,1}(X,L)$ with $\alpha\neq\gamma$ such that $$\gamma\cap\alpha \neq \emptyset \quad \text{and} \quad |\chi(X_{\gamma\alpha})| \geq 3.$$ 
Assume that $X_{\gamma\alpha}$ is of type $S_{g_0,k}$. Then $\partial X_{\gamma\alpha}$ is a simple closed multi-geodesic that split off an $S_{g_0,k}$ from $X$. By Lemma \ref{area U small} we know that
$$g_0\geq 1 \quad \text{and} \quad 3\leq 2g_0-2+k\leq g-1.$$ And we have
\begin{equation*}
\ell(\partial X_{\gamma\alpha}) \leq \ell(\alpha) +\ell(\gamma) \leq 2L.
\end{equation*}

Note that by Lemma \ref{alpha sbs U} we have $$X_\gamma\sbs X_{\gamma\alpha}.$$ Now we define a counting function as follows: 

\begin{definition}
	Define the counting function  $\hat{N}_{g_0,n_0}^{(g_1,n_1),\cdots,(g_q,n_q)}(X,L_1,L_2)$ to be the number of pairs $(\gamma_1,\gamma_2)$ satisfying 
	\begin{itemize}
		\item $\gamma_2$ is a simple closed multi-geodesics in $X$ consisting of $n_0$ geodesics that split off an $S_{g_0,n_0}$ from $X$, and its complement $X\setminus S_{g_0,n_0}$ consists of $q$ components $S_{g_1,n_1},\cdots,S_{g_q,n_q}$ for some $q\geq 1$;
		\item $\gamma_1$ is a simple closed geodesic in that $S_{g_0,n_0}$ and splits off a one-holed torus from that $S_{g_0,n_0}$;
		\item $\ell(\gamma_1)\leq L_1$ and $\ell(\gamma_2)\leq L_2$.
	\end{itemize}
    (see Figure \ref{figure:def hat N}.) 
\end{definition}

\begin{figure}[h]
	\centering	
	\includegraphics[width=8.2cm]{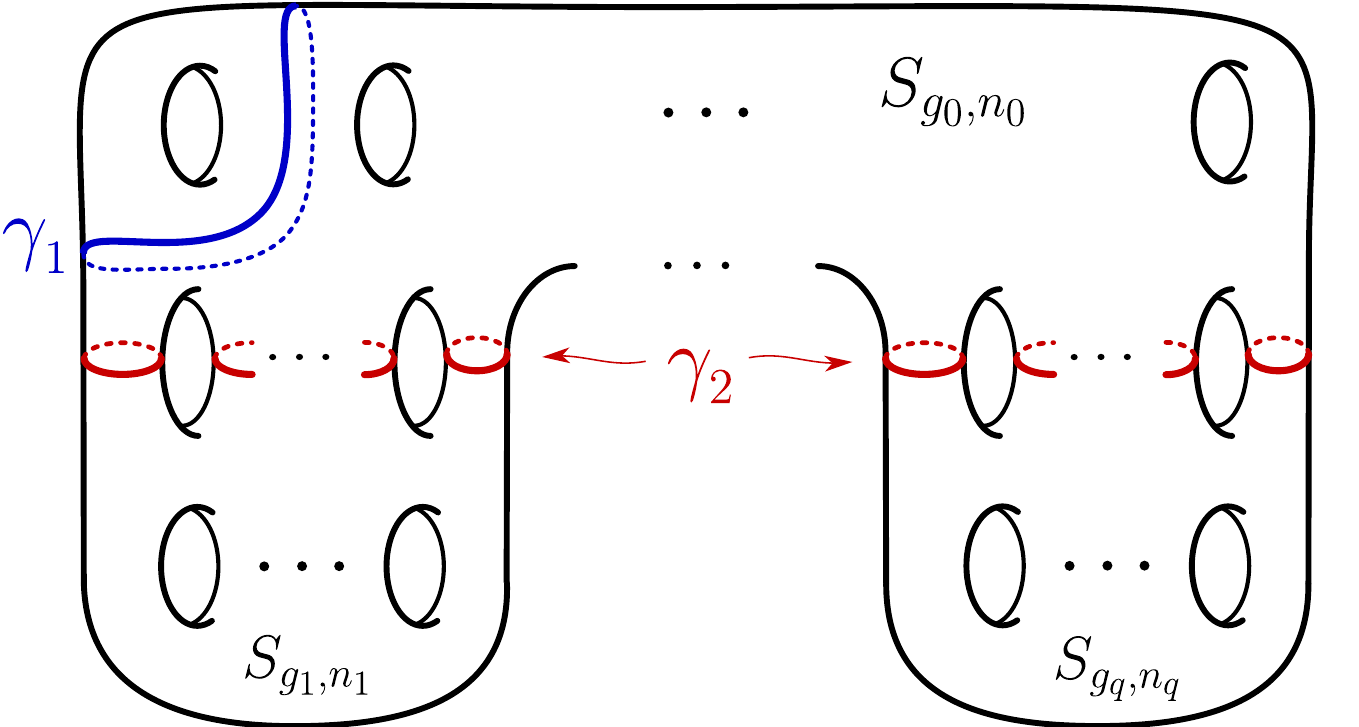}
	\caption{}
	\label{figure:def hat N}
\end{figure}

Note that the map
\begin{equation*}
\gamma \mapsto (\gamma,\partial X_{\gamma\alpha})
\end{equation*}
is injective and $\gamma\cap \partial X_{\gamma\alpha} = \emptyset$, then we have
\begin{equation}\label{N-N* leq sum N Gamma}
N_{1,1}(X,L) - N^*_{1,1}(X,L) \leq \sum \hat{N}_{g_0,k}^{(g_1,n_1),\cdots,(g_q,n_q)}(X,L,2L)
\end{equation}
where the summation takes over all possible $(g_0,k)$, $q\geq 1$, and $(g_1,n_1),\cdots,(g_q,n_q)$ such that 
\begin{itemize}
	\item $g_0\geq 1$, $3\leq 2g_0-2+k \leq g-1$;
	\item $n_i\geq 1$, $2g_i-2+n_i \geq 1$, $\forall 1\leq i\leq q$;
	\item $n_1+\cdots+n_q = k$, $g_0+g_1+\cdots+g_q + k-q =g$.
\end{itemize}

For such a counting function, by Mirzakhani's Integration Formula (see Theorem \ref{Mirz int formula}),
we have
\begin{eqnarray*}
& & \int_{\M_g} \hat{N}_{g_0,k}^{(g_1,n_1),\cdots,(g_q,n_q)}(X,L,2L) dX \\
= & & \frac{C_\Gamma}{|\Sym|} \int_{\R_{\geq0}^{k+1}} \mathbf 1_{[0,L]}(y)\mathbf 1_{[0,2L]}\big(\sum_{i=1}^q(x_{i,1} + \cdots +x_{i,n_i})\big) \\
& & V_{1,1}(y) V_{g_0-1,k+1}(y,x_{1,1},\cdots,x_{q,n_q}) \\
& & V_{g_1,n_1}(x_{1,1},\cdots,x_{1,n_1})\cdots V_{g_q,n_q}(x_{q,1},\cdots,x_{q,n_q}) \\
& & y x_{1,1}\cdots x_{q,n_q} dy dx_{1,1}\cdots dx_{q,n_q}.
\end{eqnarray*}
From Theorem \ref{Mirz vol lemma 0} of Mirzakhani we know that
\begin{equation*}
V_{1,1}(y) = \frac{1}{48}(y^2 + 4\pi^2).
\end{equation*}
Recall that $C_\Gamma \leq 1$ and it is clear that the symmetry satisfies
\begin{equation*}
|\Sym| \geq n_1!\cdots n_q!.
\end{equation*}
By Lemma \ref{MP vol lemma}, we have
\begin{equation*}
V_{g,n}(x_1,\cdots,x_n) \leq \prod_{i=1}^n \frac{\sinh (x_i/2)}{x_i/2} V_{g,n},
\end{equation*}
and we also have that for $x>0$,
\begin{equation*}
\frac{\sinh (x/2)}{x/2} \leq \frac{e^{x/2}}{x}.
\end{equation*}
Set the condition
\begin{equation*}
\mathrm{Cond}:=\left\{ 0\leq y\leq L, \ 0\leq x_{i,j}, \  \sum_{i=1}^q\sum_{j=1}^{n_i} x_{i,j} \leq 2L  \right\}.
\end{equation*}

\noindent Put all these equations together we get
\begin{eqnarray}\label{E[N Gamma] leq}
&&\int_{\M_g} \hat{N}_{g_0,k}^{(g_1,n_1),\cdots,(g_q,n_q)}(X,L,2L) dX
 \leq  \frac{1}{n_1!\cdots n_q!} V_{g_1,n_1}\cdots V_{g_q,n_q} \\
& & \times\int_\mathrm{Cond} \big( \frac{1}{48}(y^2 + 4\pi^2)y e^{(x_{1,1}+\cdots+x_{q,n_q})/2} V_{g_0-1,k+1}(y,x_{1,1},\cdots,x_{q,n_q})\big) \nonumber\\
& & dy dx_{1,1}\cdots dx_{q,n_q} . \nonumber
\end{eqnarray}

Next we control the summation $\sum \E[\hat{N}_{g_0,k}^{(g_1,n_1),\cdots,(g_q,n_q)}(X,L,2L)]$ for two different cases, and then combine them to obtain Proposition \ref{N-N*}.

\begin{lemma}\label{sum E[N Gamma] for chi=m}
Given an integer $m\geq 2$ independent of $g$, then there exists a constant $c(m)>0$ only depending on $m$ such that
\begin{equation*}
\sum_{|\chi(S_{g_0,k})|=m} \E[\hat{N}_{g_0,k}^{(g_1,n_1),\cdots,(g_q,n_q)}(X,L,2L)] \leq c(m) (1+L^{3m-1}) e^L  \frac{1}{g^m}
\end{equation*}
where the summation takes over all possible $(g_0,k)$, $q\geq 1$ and $(g_1,n_1),\cdots,(g_q,n_q)$ such that $g_0\geq 1$ and $2g_0-2+k=m$.

\end{lemma}

\begin{proof}
Since $|\chi(S_{g_0,k})|=m$, we have that all the nonnegative integers $g_0,k,q,n_1,...,n_q$ are all bounded from above by a constant only depending on $m$. By Theorem \ref{Mirz vol lemma 0} of Mirzakhani, we know that $V_{g_0-1,k+1}(y,x_{1,1},\cdots,x_{q,n_q})$ is a polynomial of degree $6g_0-10+2k$. Thus there exists a constant $c_1(m)>0$ only depending on $m$ such that
\be \label{V-upper-1-1}
V_{g_0-1,k+1}(y,x_{1,1},\cdots,x_{q,n_q})\leq c_1(m)(1+L^{6g_0-10+2k}).
\ene
For the integral in the RHS of \eqref{E[N Gamma] leq}, there exists a uniform constant $c>0$, and two constants $c'(m), c''(m)>0$ only depending on $m$ such that
\begin{eqnarray*}
& & \int_\mathrm{Cond}  \frac{1}{48}(y^2 + 4\pi^2)y e^{(x_{1,1}+\cdots+x_{q,n_q})/2} dydx_{1,1}\cdots dx_{q,n_q} \\
&\leq& c \cdot (1+L^4) \int_\mathrm{Cond} e^{(x_{1,1}+\cdots +x_{q,n_q})/2} dx_{1,1}\cdots dx_{q,n_q} \\
&\leq& c'(m) (1+L^4) (1+L^{\sum_{i=1}^q n_i-1}) e^L  \\
&\leq& c''(m) (1+L^{k+3}) e^L.
\end{eqnarray*}
Which together with \eqref{E[N Gamma] leq} and \eqref{V-upper-1-1} imply that there exists a constant $c'''(m)>0$ only depending on $m$ such that
\begin{eqnarray*}
\int_{\M_g} \hat{N}_{g_0,k}^{(g_1,n_1),\cdots,(g_q,n_q)}(X,L,2L)
&\leq& c'''(m) (1+L^{6g_0-7+3k}) e^L  V_{g_1,n_1}\cdots V_{g_q,n_q} \\
&=& c'''(m) (1+L^{3m-1}) e^L  V_{g_1,n_1}\cdots V_{g_q,n_q}.
\end{eqnarray*}

\noindent By Proposition \ref{1 over gm} we know that there exists a constant $c_2(m)>0$ only depending on $m$ such that
\begin{eqnarray*}
\sum_{g_1,\cdots,g_q} V_{g_1,n_1}\cdots V_{g_q,n_q}\leq c_2(m) \frac{1}{g^{m}} V_g.
\end{eqnarray*} 

\noindent So we have that there exists a constant $c_3(m)>0$ only depending on $m$ such that
\begin{equation*}
\sum_{g_1,\cdots,g_q} \int_{\M_g} \hat{N}_{g_0,k}^{(g_1,n_1),\cdots,(g_q,n_q)}(X,L,2L)
\leq c_3(m) (1+L^{3m-1}) e^L \frac{1}{g^{m}} V_g.
\end{equation*}
Recall that the nonnegative integers $g_0,k,q,n_1,\cdots,n_q$ are all bounded from above by a constant only depending on $m$. Therefore there exists a constant $c(m)>0$ only depending on $m$ such that
\begin{equation*}
\sum_{|\chi(S_{g_0,k})| =m}  \E[\hat{N}_{g_0,k}^{(g_1,n_1),\cdots,(g_q,n_q)}(X,L,2L)])
\leq c(m) (1+L^{3m-1}) e^L \frac{1}{g^m},
\end{equation*}
as required.
\end{proof}

\begin{lemma}\label{sum E[N Gamma] for chi geq m}
Given an integer $m\geq 2$ independent of $g$, then there exists a constant $c(m)>0$ only depending on $m$ such that
\begin{equation*}
\sum_{m+1\leq|\chi(S_{g_0,k})|\leq g-1} \E[\hat{N}_{g_0,k}^{(g_1,n_1),\cdots,(g_q,n_q)}(X,L,2L)] \leq c(m) (1+L^3) e^{\frac{9}{2}L}  \frac{1}{g^m}
\end{equation*}
where the summation takes over all possible $(g_0,k)$, $q\geq 1$ and $(g_1,n_1),\cdots,(g_q,n_q)$ such that $g_0\geq 1$ and $m+1\leq 2g_0-2+k\leq g-1$.

\end{lemma}

\begin{proof}
First by Lemma \ref{Mirz vol lemma 1} we know that
\begin{eqnarray*}
V_{g_0-1,k+1}(y,x_{1,1},\cdots,x_{q,n_q}) &\leq& \left(e^{y/2}\cdot \prod_{i=1}^{q} \prod_{j=1}^{n_i} e^{x_{i,j}/2}\right)\cdot V_{g_0-1,k+1}\\
&=&\left(e^{y/2}\cdot  e^{\sum x_{i,j}/2}\right)\cdot V_{g_0-1,k+1}.
\end{eqnarray*}
Then by \eqref{E[N Gamma] leq} we have
\begin{eqnarray*}
\int_{\M_g} \hat{N}_{g_0,k}^{(g_1,n_1),\cdots,(g_q,n_q)}(X,L,2L)
& \leq & \frac{1}{n_1!\cdots n_q!} V_{g_0-1,k+1} V_{g_1,n_1}\cdots V_{g_q,n_q} \\
& & \int_\mathrm{Cond}  \frac{y}{48}(y^2 + 4\pi^2)e^{y/2} e^{\sum x_{i,j}} \\
& & dy dx_{1,1}\cdots dx_{q,n_q} .
\end{eqnarray*}
For the integral in the {\rm RHS} above, there exists a universal constant $c>0$ such that for large enough $g$ and $L$,
\begin{eqnarray*}
& &\int_\mathrm{Cond}  \frac{y}{48}(y^2 + 4\pi^2)e^{y/2} e^{\sum x_{i,j}} dy dx_{1,1}\cdots dx_{q,n_q} \\
& = & \int_0^L \frac{y}{48}(y^2 + 4\pi^2)e^{y/2}dy \\
& & \int_{\sum x_{i,j}\leq 2L, x_{i,j}\geq0}
e^{\sum x_{i,j}} dx_{1,1}\cdots dx_{q,n_q} \\
& \leq & c (1+L^3) e^{L/2} e^{2L}
\int_{\sum x_{i,j}\leq 2L, x_{i,j}\geq0} dx_{1,1}\cdots dx_{q,n_q} \\
& = & c (1+L^3) e^{\frac{5}{2}L} \frac{(2L)^{k}}{k!} .
\end{eqnarray*}
So we have
\begin{equation*}
\int_{\M_g} \hat{N}_{g_0,k}^{(g_1,n_1),\cdots,(g_q,n_q)}(X,L,2L) \leq c (1+L^3) e^{\frac{5}{2}L} \frac{(2L)^{k}}{k!n_1!\cdots n_q!} V_{g_0-1,k+1} V_{g_1,n_1}\cdots V_{g_q,n_q}.
\end{equation*}

\noindent Similar as in proof of Lemma \ref{sum E[N Gamma] for chi=m}, it follows by Lemma \ref{sum vol lemma} that
\begin{equation*}
\sum_{g_1,\cdots,g_q} V_{g_1,n_1}\cdots V_{g_q,n_q}
\leq c \big(\frac{D}{2g-2g_0-k}\big)^{q-1} W_{2g-2g_0-k}
\end{equation*}

Recall that for fixed $k$, we always have
\begin{equation*}
\sum_{n_1+..+n_q=k,\ n_i\geq0} \frac{k!}{n_1!...n_q!} = q^{k}.
\end{equation*}
So we have that for large enough $g>0$,
\begin{eqnarray*}
& & \sum_{(g_0,k)} \sum_q\sum_{n_1,\cdots,n_q}\sum_{g_1,\cdots,g_q} \int_{\M_g} \hat{N}_{g_0,k}^{(g_1,n_1),\cdots,(g_q,n_q)}(X,L,2L) \\
&\leq & \sum_{(g_0,k)} \sum_q c(1+L^3) e^{\frac{5}{2}L} (\frac{D}{g})^{q-1} \frac{(2L)^{k}}{k!}\frac{q^{k}}{k!} V_{g_0-1,k+1} W_{2g-2g_0-k} \\
& \leq & \sum_{(g_0,k)} \sum_q c(1+L^3) e^{\frac{5}{2}L} (\frac{D}{g})^{q-1} e^{2L} e^q V_{g_0-1,k+1} W_{2g-2g_0-k} \\
& \leq & \sum_{(g_0,k)}  c(1+L^3) e^{\frac{9}{2}L} V_{g_0-1,k+1} W_{2g-2g_0-k} \\
\end{eqnarray*}

\noindent Recall that Part $(1)$ of Lemma \ref{Wr-prop} tells that $V_{g,n}\leq c W_{2g-2+n}$ for a universal constant $c>0$. Then it follows by Part $(2)$ of Lemma \ref{Wr-prop} that there exist two constants $c'(m),c(m)>0$ only depending on $m$ such that
\begin{eqnarray*}
&&\sum_{m+1\leq |\chi(S_{g_0,k})|\leq g-1} \E[\hat{N}_{g_0,k}^{(g_1,n_1),\cdots,(g_q,n_q)}(X,L,2L)]\cdot V_g\\
&&\leq \sum_{k} \sum_{g_0:\ m+1\leq 2g_0-2+k\leq g-1} c (1+L^3) e^{\frac{9}{2}L} V_{g_0-1,k+1} W_{2g-2g_0-k}\\
&&\leq \sum_{k} \sum_{g_0:\ m+1\leq 2g_0-2+k\leq g-1} c (1+L^3) e^{\frac{9}{2}L} W_{2g_0-3+k} W_{2g-2g_0-k}\\
&&= \sum_{k} \sum_{g_0:\ m\leq 2g_0-3+k\leq g-2} c (1+L^3) e^{\frac{9}{2}L} W_{2g_0-3+k} W_{2g-2g_0-k}\\
&&\leq \sum_{k} c'(m) (1+L^3) e^{\frac{9}{2}L} \frac{1}{g^{m}} W_{2g-3}  \\
&&= \sum_{k} c'(m) (1+L^3) e^{\frac{9}{2}L} \frac{V_{g-1,1}}{g^{m}}   \\
&&\leq c(m) (1+L^3) e^{\frac{9}{2}L} \frac{1}{g^{m}} V_g
\end{eqnarray*}
where in the last inequality we apply the facts that $k\leq g-1$ and $V_{g}\asymp gV_{g-1,1}$ (see Part $(2)$ and $(3)$ of Lemma \ref{Mirz vol lemma 1}). That is,
\[\sum_{m+1\leq |\chi(S_{g_0,k})|\leq g-1} \E[\hat{N}_{g_0,k}^{(g_1,n_1),\cdots,(g_q,n_q)}(X,L,2L)] \leq c(m) (1+L^3) e^{\frac{9}{2}L} \frac{1}{g^{m}},\]
as required.
\end{proof}

Now we are ready to prove Proposition \ref{N-N*}.

\bp[Proof of Proposition \ref{N-N*}]
First since $\sN^*_{1,1}(X,L) \sbs \sN_{1,1}(X,L)$,
\begin{equation}\label{1 upper bound}
\E[N_{1,1}(X,L)] - \E[N^*_{1,1}(X,L)]\geq 0.
\end{equation} It suffices to show the other side. By Equation \eqref{N-N* leq sum N Gamma}, Lemma
 \ref{sum E[N Gamma] for chi=m} and \ref{sum E[N Gamma] for chi geq m} we have
\begin{eqnarray*}
&& \E[N_{1,1}(X,L)]-\E[N^*_{1,1}(X,L)] \\
&\leq& \sum \E[\hat{N}_{g_0,k}^{(g_1,n_1),\cdots,(g_q,n_q)}(X,L,2L)] \\
&=& \sum_{3\leq |\chi(S_{g_0,k})|\leq 100} \E[\hat{N}_{g_0,k}^{(g_1,n_1),\cdots,(g_q,n_q)}(X,L,2L)]\\
&+&\sum_{|\chi(S_{g_0,k})|>100} \E[\hat{N}_{g_0,k}^{(g_1,n_1),\cdots,(g_q,n_q)}(X,L,2L)]\\
&\leq& \sum_{m=3}^{100} c(m) (1+L^{3m-1}) e^L  \frac{1}{g^m} + c(100) (1+L^3) e^{\frac{9}{2}L}  \frac{1}{g^{100}}.
\end{eqnarray*}

\noindent Recall that $L=L(g)=2\log g -4\log \log g +\omega(g)$. As $g\to \infty$ we have that
$$\E[N_{1,1}(X,L)]-\E[N^*_{1,1}(X,L)] = O\left(\frac{(\log g)^4}{g} e^{\omega(g)}\right) \to 0 \ \text{as}\ g\to\infty.$$
Which together with \eqref{1 upper bound} imply that
$$\limg \left(\E[N_{1,1}(X,L)] - \E[N^*_{1,1}(X,L)]\right) =0.$$
By Lemma \ref{E[N]}, we know that $\limg \E[N_{1,1}(X,L)] = \infty$. So
$$\limg \frac{\E[N^*_{1,1}(X,L)]}{\E[N_{1,1}(X,L)]}=1.$$

For \eqref{tend to 0 (1)}, as shown above and by Lemma \ref{E[N]} we have
\begin{equation*}
\frac{1}{\E[N^*_{1,1}(X,L)]} \sim \frac{1}{\E[N_{1,1}(X,L)]} \sim \frac{1}{\frac{1}{384\pi^2} L^2 e^{\frac{L}{2}} \frac{1}{g} } =O(e^{-\frac{\omega(g)}{2}}) \rightarrow 0
\end{equation*}
as $g\rightarrow\infty$, which proves \eqref{tend to 0 (1)}.
\ep

\subsection{Proof of \eqref{tend to 0 (2)}}\label{subsec:proof of (B)}
In this subsection we show \eqref{tend to 0 (2)}, whose proof is similar to the one of \eqref{tend to 0 (1)}. First we define

\begin{definition}
\begin{equation*}
\mathcal Y(X,L):= \left\{ (\alpha,\beta)\in\sN_{1,1}(X,L)\times\sN_{1,1}(X,L) \,;\, \alpha\neq \beta, \alpha\cap\beta=\emptyset \right\}
\end{equation*}
and
\begin{equation*}
Y(X,L) := \#\mathcal Y(X,L) = \sum_{\alpha\neq \beta,\alpha\cap\beta=\emptyset} \mathbf 1_{\sN_{1,1}(X,L)}(\alpha) \mathbf 1_{\sN_{1,1}(X,L)}(\beta).
\end{equation*}
\end{definition}

\begin{lemma}\label{E[Y]}
As $g\to \infty$, we have
\begin{eqnarray*}
\E[Y(X,L)]
&=& \frac{1}{(384\pi^2)^2} L^4 e^L \tfrac{1}{g^2} \left(1+O(\tfrac{1}{L})\right) \left(1+O(\tfrac{L^2}{g})\right) \left(1+O\left(\tfrac{1}{g}\right)\right) \\
&=&\E[N_{1,1}(X,L)]^2 \left(1+O(\tfrac{1}{L})\right) \left(1+O(\tfrac{L^2}{g})\right) \left(1+O\left(\tfrac{1}{g}\right)\right),
\end{eqnarray*}
where the implied constants are independent of $L$ and $g$. As a consequence, for $L(g):=2\log g-4\log\log g+\omega(g)$, we have
$$
\lim_{g\to0}\big(\E[Y(X,L(g))]-\E[N_{1,1}(X,L(g))]^2\big)=0.
$$
\end{lemma}

\begin{proof}
By Mirzakhani's Integration Formula (see Theorem \ref{Mirz int formula}),
\begin{eqnarray*}
&&\int_{\M_g}Y(X,L) dX \\
&=& 2^{-M}\int_{\R_{\geq 0}^2} \mathbf 1_{[0,L]}(x) \mathbf 1_{[0,L]}(y) V_{1,1}(x)V_{1,1}(y)V_{g-2,2}(x,y)xydxdy \\
&=& \frac{1}{4} \int_{[0,L]^2} \frac{1}{48}x(x^2 + 4\pi^2) \frac{1}{48}y(y^2 + 4\pi^2) V_{g-2,2}(x,y)dxdy.
\end{eqnarray*}

\noindent By Lemma \ref{Mirz vol lemma 1} we know that
\[\frac{V_{g-2,2}}{V_g}=\frac{1}{(8\pi^2 g)^2}\left(1+O\left(\frac{1}{g}\right)\right).\]
Thus, it follows by Lemma \ref{MP vol lemma} that
\begin{eqnarray*}
V_{g-2,2}(x,y)
&=& \frac{\sinh(x/2)}{x/2} \frac{\sinh(y/2)}{y/2} V_{g-2,2} \big(1+O(\frac{L^2}{g})\big) \\
&=& \frac{\sinh(x/2)}{x/2} \frac{\sinh(y/2)}{y/2} \frac{1}{64\pi^4 g^2} V_{g} \big(1+O(\frac{L^2}{g})\big) \big(1+O\left(\frac{1}{g}\right)\big).
\end{eqnarray*}
So we have
\begin{equation*}
\E[Y(X,L)]
= \frac{1}{(384\pi^2)^2} L^4 e^L \frac{1}{g^2} \big(1+O(\frac{1}{L})\big) \big(1+O(\frac{L^2}{g})\big) \big(1+O\left(\frac{1}{g}\right)\big).
\end{equation*}
By Lemma \ref{E[N]} we have
\begin{equation*}
\E[Y(X,L)]= \E[N_{1,1}(X,L)]^2 \big(1+O(\frac{1}{L})\big) \big(1+O(\frac{L^2}{g})\big) \big(1+O\left(\frac{1}{g}\right)\big),
\end{equation*}
as required.
\end{proof}

Recall that $L=L(g)=2\log g -4\log \log g +\omega(g)$. Lemma \ref{E[Y]} implies that as $g\rightarrow\infty$,
\begin{equation}\label{Y-X/X}
\frac{\E[Y(X,L)] - \E[N_{1,1}(X,L)]^2}{\E[N_{1,1}(X,L)]^2} = O\big(\frac{1}{L}+\frac{L^2}{g}+\frac{1}{g} \big) \rightarrow 0.
\end{equation}

We will show that $\E[Y^*]$ is an approximation of $\E[Y]$. More precisely,
\begin{proposition}\label{E[Y]-E[Y*]}
With the notations as above, we have
\begin{equation*}
\limg (\E[Y(X,L)] - \E[Y^*(X,L)])=0
\end{equation*}
Moreover, Equation \eqref{tend to 0 (2)} holds.
\end{proposition}

\begin{proof}
First by definition of $Y$ and $Y^*$ we know that
\begin{equation*}
Y^*(X,L) \leq Y(X,L).
\end{equation*}
So we have
\begin{equation}\label{1-Y upper bound}
\E[Y(X,L)] - \E[Y^*(X,L)] \geq 0.
\end{equation}
It suffices to show the other side. The proof is similar to the proof of Proposition \ref{N-N*}.

For any ordered pair $(\alpha,\beta) \in \mathcal Y(X,L)\setminus \mathcal Y^*(X,L)$, we have $\alpha\in\sN_{1,1}(X,L) \setminus \sN^*_{1,1}(X,L)$ or $\beta\in\sN_{1,1}(X,L) \setminus \sN^*_{1,1}(X,L)$. Without loss of generality we assume $\alpha\in\sN_{1,1}(X,L) \setminus \sN^*_{1,1}(X,L)$. Then by definition of $\sN_{1,1}(X,L)$ and $\sN^*_{1,1}(X,L)$, it follows by Lemma \ref{area U small} that there exists a simple closed geodesic $\alpha'\in \sN_{1,1}(X,L)$ with $\alpha'\neq\alpha$ such that 
$$\alpha\cap\alpha' \neq \emptyset \quad \text{and} \quad |\chi(X_{\alpha\alpha'})| \geq 3.$$ 
The relation between $X_\beta$ and $X_{\alpha\alpha'}$ can be divided into the following three cases. (see Figure \ref{figure:cases}.)
\begin{figure}[h]
	\includegraphics[width=12cm]{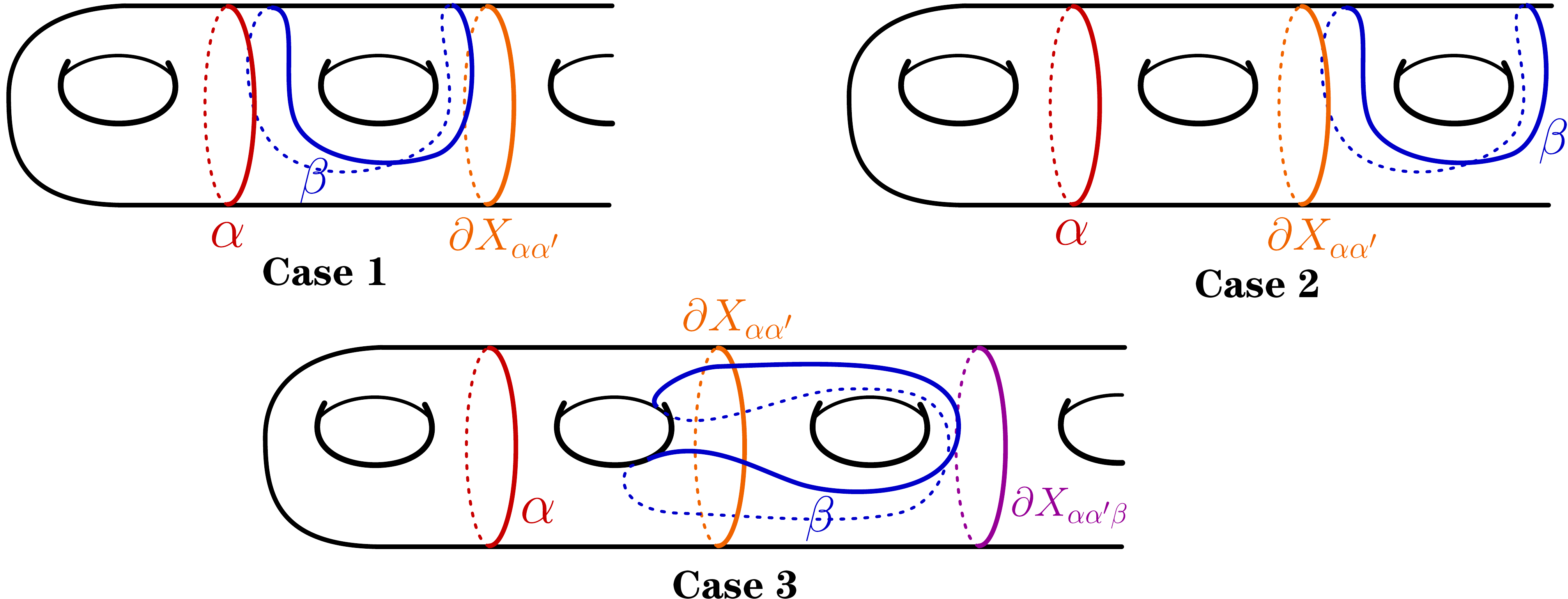}
	\caption{Relation between $X_\beta$ and $X_{\alpha\alpha'}$ in the three cases.}
	\label{figure:cases}
\end{figure}

\textbf{Case 1. $X_\beta \sbs X_{\alpha\alpha'}$.} \\
For this case we have $$\beta \cap \partial X_{\alpha\alpha'} = \emptyset.$$ ($\beta$ won't be part of $\partial X_{\alpha\alpha'}$ since $X_\beta$ is a one-holed torus and $X_{\alpha\alpha'}$ is not.) So $\alpha,\beta$ and $\partial X_{\alpha\alpha'}$ are pairwisely disjoint. Assume $X_{\alpha\alpha'}$ is of type $S_{g_0,k}$. Note that $X_\alpha,X_\beta$ are two disjoint one-holed torus in $X_{\alpha\alpha'}$, so $g_0\geq 2$. By Lemma \ref{area U small}, we have
\begin{equation*}
3 \leq |\chi(X_{\alpha\alpha'})| \leq g-1,
\end{equation*}
and 
\begin{equation*}
\ell(\alpha)\leq L, \ \ell(\beta)\leq L,\ \ell(\partial X_{\alpha\alpha'}) \leq 2L.
\end{equation*}
Similar to what we have done in the proof of Proposition \ref{N-N*}, we define a counting function as follows:
\begin{definition}
	Define the counting function  $\dot{N}_{g_0,n_0}^{(g_1,n_1),\cdots,(g_q,n_q)}(X,L_1,L_2,L_3)$ to be the number of pairs $(\gamma_1,\gamma_2,\gamma_3)$ satisfying 
	\begin{itemize}
		\item $\gamma_3$ is a simple closed multi-geodesics in $X$ consisting of $n_0$ geodesics that split off an $S_{g_0,n_0}$ from $X$ and the complement $X\setminus S_{g_0,n_0}$ consist of $q$ components $S_{g_1,n_1},\cdots,S_{g_q,n_q}$ for some $q\geq 1$;
		\item $\gamma_1$ and $\gamma_2$ are two disjoint simple closed geodesics in that $S_{g_0,n_0}$, and split off two disjoint one-holed torus in that $S_{g_0,n_0}$;
		\item $\ell(\gamma_1)\leq L_1$, $\ell(\gamma_2)\leq L_2$, $\ell(\gamma_3)\leq L_3$.
	\end{itemize}
\end{definition}

Since the map
\begin{equation*}
(\alpha,\beta) \mapsto (\alpha,\beta,\partial X_{\alpha\alpha'})
\end{equation*}
is injective, the number of pairs $(\alpha,\beta) \in \mathcal Y(X,L)\setminus \mathcal Y^*(X,L)$ satisfying Case $1$ is bounded from above by
\begin{equation}\label{Y-Y* case 1 leq}
Q_1 := \sum \dot{N}_{g_0,k}^{(g_1,n_1),\cdots,(g_q,n_q)}(X,L,L,2L)
\end{equation}
where the summation takes over all possible $(g_0,k)$, $q\geq 1$ and $(g_1,n_1),\cdots,(g_q,n_q)$ such that 
\begin{itemize}
	\item $g_0\geq 2$, $3\leq 2g_0-2+k \leq g-1$;
	\item $n_i\geq 1$, $2g_i-2+n_i \geq 1$, $\forall 1\leq i\leq q$;
	\item $n_1+\cdots+n_q = k$, $g_0+g_1+\cdots+g_q + k-q =g$.
\end{itemize}

\textbf{Case 2: $X_\beta \cap X_{\alpha\alpha'} = \emptyset$.}\\
For this case we have that $\alpha,\beta$ and $\partial X_{\alpha\alpha'}$ are pairwisely disjoint. Assume $X_{\alpha\alpha'}$ is of type $S_{g_0,k}$. By Lemma \ref{area U small}, we have
\begin{equation*}
g_0\geq 1,\ \ 3 \leq |\chi(X_{\alpha\alpha'})| \leq g-1,
\end{equation*}
and 
\begin{equation*}
\ell(\alpha)\leq L, \ \ell(\beta)\leq L,\ \ell(\partial X_{\alpha\alpha'}) \leq 2L.
\end{equation*}
Similar to what we have done in the proof of Proposition \ref{N-N*}, we define a counting function as follows:
\begin{definition}
	Define the counting function  $\ddot{N}_{g_0,n_0}^{(g_1,n_1),\cdots,(g_q,n_q)}(X,L_1,L_2,L_3)$ to be the number of pairs $(\gamma_1,\gamma_2,\gamma_3)$ satisfying 
	\begin{itemize}
		\item $\gamma_3$ is a simple closed multi-geodesics in $X$ consisting of $n_0$ geodesics that split off an $S_{g_0,n_0}$ from $X$ and the complement $X\setminus S_{g_0,n_0}$ consist of $q$ components $S_{g_1,n_1},\cdots,S_{g_q,n_q}$ for some $q\geq 1$;
		\item $\gamma_1$ is a simple closed geodesic in that $S_{g_0,n_0}$, and splits off a one-holed torus in that $S_{g_0,n_0}$;
		\item $\gamma_2$ is a simple closed geodesic in that $S_{g_1,n_1}$, and splits off a one-holed torus in that $S_{g_1,n_1}$;
		\item $\ell(\gamma_1)\leq L_1$, $\ell(\gamma_2)\leq L_2$, $\ell(\gamma_3)\leq L_3$.
	\end{itemize}
\end{definition}

Since the map
\begin{equation*}
(\alpha,\beta) \mapsto (\alpha,\beta,\partial X_{\alpha\alpha'})
\end{equation*}
is injective, the number of pairs $(\alpha,\beta) \in \mathcal Y(X,L)\setminus \mathcal Y^*(X,L)$ satisfying Case $2$ is bounded from above by
\begin{equation}\label{Y-Y* case 2 leq}
Q_2 := \sum \ddot{N}_{g_0,k}^{(g_1,n_1),\cdots,(g_q,n_q)}(X,L,L,2L)
\end{equation}
where the summation takes over all possible $(g_0,k)$, $q\geq 1$ and $(g_1,n_1),\cdots,(g_q,n_q)$ such that 
\begin{itemize}
	\item $g_0\geq 1$, $3\leq 2g_0-2+k \leq g-1$;
	\item $n_i\geq 1$, $2g_i-2+n_i \geq 1$, $\forall 1\leq i\leq q$;
	\item $n_1+\cdots+n_q = k$, $g_0+g_1+\cdots+g_q + k-q =g$;
	\item $g_1\geq 1$.
\end{itemize}

\textbf{Case 3: $X_\beta \cap X_{\alpha\alpha'} \neq \emptyset$ and $X_\beta \nsubseteq X_{\alpha\alpha'}$.}\\
For this case we have that $\beta$ and $\partial X_{\alpha\alpha'}$ are not disjoint. We consider the subsurface with geodesic boundary $X_{\alpha\alpha'\beta}\subset X$ constructed from $X_{\alpha\alpha'}$ and $X_\beta$ in the way described in Section \ref{section union} (\ie $X_1=X_{\alpha\alpha'}$, $X_2=X_\beta$ and $X_3=X_{\alpha\alpha'\beta}$ in the notation of Section \ref{section union}). Then $\alpha,\beta$ and $\partial X_{\alpha\alpha'\beta}$ are pairwisely disjoint. Assume $X_{\alpha\alpha'\beta}$ is of type $S_{g_0,k}$. Note that $X_\alpha,X_\beta$ are two disjoint one-holed torus in $X_{\alpha\alpha'\beta}$, so $g_0\geq 2$. Recall that $L=L(g)=2\log g -4\log \log g +\omega(g)$. Thus, by Lemma \ref{area U small} we have that for large enough $g>0$,
\begin{equation*}
3 \leq |\chi(X_{\alpha\alpha'})| \leq \tfrac{1}{2}g,
\end{equation*}
and 
\begin{equation*}
\ell(\partial X_{\alpha\alpha'}) \leq 2L.
\end{equation*}
Then again by Lemma \ref{lemma chiU12}, we have for large enough $g>0$,
\begin{equation*}
4 \leq |\chi(X_{\alpha\alpha'\beta})| \leq g-1.
\end{equation*}
and 
\begin{equation*}
\ell(\alpha)\leq L, \ \ell(\beta)\leq L,\ \ell(\partial X_{\alpha\alpha'\beta}) \leq 3L.
\end{equation*}

Since the map
\begin{equation*}
(\alpha,\beta) \mapsto (\alpha,\beta,\partial X_{\alpha\alpha'})
\end{equation*}
is injective, the number of $(\alpha,\beta) \in \mathcal Y(X,L)\setminus \mathcal Y^*(X,L)$ satisfying Case $3$ is bounded from above by
\begin{equation}\label{Y-Y* case 3 leq}
Q_3 := \sum \dot{N}_{g_0,k}^{(g_1,n_1),\cdots,(g_q,n_q)}(X,L,L,3L)
\end{equation}
where the summation takes over all possible $(g_0,k)$, $q\geq 1$ and $(g_1,n_1),\cdots,(g_q,n_q)$ such that 
\begin{itemize}
	\item $g_0\geq 2$, $4\leq 2g_0-2+k \leq g-1$;
	\item $n_i\geq 1$, $2g_i-2+n_i \geq 1$, $\forall 1\leq i\leq q$;
	\item $n_1+\cdots+n_q = k$, $g_0+g_1+\cdots+g_q + k-q =g$.
\end{itemize}

Then by the discussion above, we have 
\begin{equation}\label{Y-Y* leq}
Y(X,L)-Y^*(X,L) \leq 2(Q_1+Q_2+Q_3)
\end{equation}
where the coefficient $2$ comes from the reason that we have assumed $\alpha\in \sN_{1,1}(X,L)\setminus\sN^*_{1,1}(X,L)$; indeed if $\beta\in \sN_{1,1}(X,L)\setminus\sN^*_{1,1}(X,L)$ and $\alpha \in \sN^*_{1,1}(X,L)$ one may have the same upper bound $(Q_1+Q_2+Q_3)$. \\

Then we have the following estimations of $\E[Q_1], \E[Q_2]$ and $\E[Q_3]$, whose proofs are exactly the same as the proofs of Lemma \ref{sum E[N Gamma] for chi=m} and \ref{sum E[N Gamma] for chi geq m}, and we omit the details. For $L=L(g)=2\log g -4\log \log g +\omega(g)$, we have that as $g\to \infty$, 
\begin{equation}\label{E[Q1]}
\E[Q_1] \leq c L^{8} e^L \frac{1}{g^3} = O(\frac{(\log g)^4}{g}e^{\omega(g)}) \to 0,
\end{equation}
\begin{equation}\label{E[Q2]}
\E[Q_2] \leq c L^{10} e^\frac{3L}{2} \frac{1}{g^4} = O(\frac{(\log g)^4}{g}e^{\frac{3}{2}\omega(g)}) \to 0,
\end{equation}
and
\begin{equation}\label{E[Q3]}
\E[Q_3] \leq c L^{11} e^\frac{3L}{2} \frac{1}{g^4} = O(\frac{(\log g)^5}{g}e^{\frac{3}{2}\omega(g)}) \to 0.
\end{equation}

Therefore we have that as $g\to \infty$, 
\begin{equation}\label{Y-Y*-0}
0\leq \E[Y(X,L)]-\E[Y^*(X,L)] = O(\frac{(\log g)^5}{g}e^{\frac{3}{2}\omega(g)}) \to 0.
\end{equation}

For \eqref{tend to 0 (2)}, first we rewrite
\begin{eqnarray*}
&& \frac{\E[Y^*(X,L)] - \E[N^*_{1,1}(X,L)]^2}{\E[N^*_{1,1}(X,L)]^2} \\
&&= \frac{\E[N_{1,1}(X,L)]^2}{\E[N^*_{1,1}(X,L)]^2} \times \frac{\E[Y(X,L)]}{\E[N_{1,1}(X,L)]^2}\\
&&\times \big( 1-\frac{\E[Y(X,L)]-\E[Y^*(X,L)]}{\E[Y(X,L)]} \big) -1.
\end{eqnarray*}

\noindent As $g\to \infty$, it follows by Lemma \ref{E[N]}, \ref{N-N*} and \ref{E[Y]} that
\[\lim \limits_{g\to\infty}\frac{\E[N_{1,1}(X,L)]^2}{\E[N^*_{1,1}(X,L)]^2}=1 \ \text{and} \ \lim \limits_{g\to\infty} \frac{\E[Y(X,L)]}{\E[N_{1,1}(X,L)]^2}=1.\]
Which together with \eqref{Y-Y*-0} imply that
\begin{eqnarray*}
\lim \limits_{g\to \infty} \frac{\E[Y^*(X,L)] - \E[N^*_{1,1}(X,L)]^2}{\E[N^*_{1,1}(X,L)]^2}=0.
\end{eqnarray*}

The proof of \eqref{tend to 0 (2)} is complete.
\end{proof}

\subsection{Proof of \eqref{tend to 0 (3)}} \label{proof of third tend to 0}
In this subsection we show \eqref{tend to 0 (3)}, where we will apply Mirzakhani's generalized McShane identity for certain counting problem for $S_{0,4}$ and $S_{1,2}$.
Our main result in the subsection is as follows.
\begin{proposition}\label{E[Z]/E[N]}
With the notations as above, there exists a universal constant $c>0$ such that
\begin{equation*}
\E[Z^*(X,L)]\leq c L^3 e^L \frac{1}{g^2}.
\end{equation*}
Moreover, Equation \eqref{tend to 0 (3)} holds.
\end{proposition}

Consider an ordered pair $(\alpha,\beta)\in \mathcal Z^*(X,L)$, that is, $\alpha,\beta\in \sN^*_{1,1}(X,L)$ with $\alpha\neq \beta$ and $\alpha\cap\beta\neq\emptyset$. By definition of $\sN^*_{1,1}$, we know that $X_{\alpha\beta}$ is of type $S_{1,2}$. Unfortunately, one can not apply Mirzakhani's Integration Formula (see Theorem \ref{Mirz int formula}) to the pair $(\alpha,\beta,\partial X_{\alpha\beta})$ because $\alpha\cap\beta\neq\emptyset$. We consider the following map
\begin{equation*}
(\alpha,\beta) \mapsto (\alpha,\partial X_{\alpha\beta}).
\end{equation*}
This map may not be injective. However, one can control the multiplicity of this map. To do this, it is sufficient to count the number of such $\beta's$ with lengths $\leq L$ in a given $S_{1,2}$ with geodesic boundary. We also need a similar counting result in a given $S_{0,4}$ for some technical reason. More precisely, we have the following lemma.

\begin{lemma}\label{multi in S12 S04}
Let $Y$ be a hyperbolic surface with geodesic boundary belonging to either of the following cases:
\begin{enumerate}
  \item[(a)] $Y$ is of type $S_{1,2}$, with boundary components $\gamma_1,\gamma_2$;
  \item[(b)] $Y$ is of type $S_{0,4}$, with boundary components $\gamma_1, \gamma_2, \gamma_3, \gamma_4$.
\end{enumerate}
For any $L>0$, let $N(Y,\gamma_1,\gamma_2,L)$ be the number of simple closed geodesics on $Y$ of length $\leq L$ that form a pair of pants together with $\gamma_1$ and $\gamma_2$. Then
$$
N(Y,\gamma_1,\gamma_2,L)\leq \frac{\ell(\gamma_1)}{\mathcal R(\ell(\gamma_1),\ell(\gamma_2),L)},
$$
where $\mathcal R(x,y,z)$ is the function given in Mirzakhani's generalized McShane identity (see Theorem \ref{McShane id}).
\end{lemma}

\begin{proof}
We only treat Case (a) here. The proof for (b) is similar.

By Lemma \ref{estimation R,D} we know that $\mathcal D>0$ and $\mathcal R>0$. So by Mirzakhani's generalized McShane identity (see Theorem \ref{McShane id}) we have
\begin{eqnarray*}
\ell(\gamma_1)
&=& \sum_{\{\alpha_1,\alpha_2\}} \mathcal D(\ell(\gamma_1), \ell(\alpha_1), \ell(\alpha_2)) +
\sum_{i=2}^n \sum_\alpha \mathcal R(\ell(\gamma_1),\ell(\gamma_i),\ell(\alpha)) \\
&\geq& \sum_{\alpha'} \mathcal R(\ell(\gamma_1),\ell(\gamma_2),\ell(\alpha'))
\end{eqnarray*}
where $\alpha'$ runs over all simple closed geodesics of lengths $\leq L$ and bounding a pair of pants together with the union $\gamma_1\cup\gamma_2$. By Lemma \ref{estimation R,D} we know that $\mathcal R(x,y,z)$ is decreasing with respect to $z$. Thus,
\begin{equation*}
\ell(\gamma_1) \geq N(Y,\gamma_1,\gamma_2,L) \cdot \mathcal R(\ell(\gamma_1),\ell(\gamma_2),L),
\end{equation*}
as required.
\end{proof}

\begin{rem*}
If one considers the number of closed geodesics that are not necessarily simple, it follows from \cite[Lemma 6.6.4]{Buser10} that
\begin{equation*}
N(Y,\gamma_1,\gamma_2,L)\leq ce^L
\end{equation*}
for a universal constant  $c$. In view of Lemma \ref{estimation R,D}, the lemma above implies the better estimate
\begin{equation*}
N(Y,\gamma_1,\gamma_2,L) \leq c(\ell(\gamma_1),\ell(\gamma_2)) e^{\frac{L}{2}}.
\end{equation*}
In contrast, by \cite{Mirz08}, one may expect that
$$N(Y,\gamma_1,\gamma_2,L) \leq c_1(Y) L^4 \quad \text{if } Y\cong S_{1,2},$$
$$N(Y,\gamma_1,\gamma_2,L) \leq c_2(Y) L^2 \quad \text{if } Y\cong S_{0,4}$$
for $L$ large enough, where the constants depend on the hyperbolic surface $Y$. It is not that easy to give explicit expression for $c_1$ and $c_2$ which only depend on the lengths of the boundary geodesics.
\end{rem*}

Now we return to the proof of Proposition \ref{E[Z]/E[N]}. Note that the complement of $S_{1,2}$ in $X$ may have several possibilities: $$X\setminus S_{1,2} = S_{g-2,2} \ \text{or}\ S_{k,1}\cup S_{g-k-1,1}$$
for some $1\leq k\leq \frac{1}{2}(g-1)$. We divide $\mathcal Z^*(X,L)$ into several parts as follows.

\begin{definition}
	\begin{equation*}
	\mathcal Z^{*0}(X,L):= \left\{ (\alpha,\beta)\in\mathcal Z^*(X,L) \,;\,  X\setminus X_{\alpha\beta} \ \text{is of type}\ S_{g-2,2} \right\}
	\end{equation*}
	and for any $1\leq k\leq \frac{1}{2}(g-1)$, 
	\begin{equation*}
	\mathcal Z^{*k}(X,L):= \left\{ (\alpha,\beta)\in\mathcal Z^*(X,L) \,;\, X\setminus X_{\alpha\beta} \ \text{is of type}\ S_{k,1}\cup S_{g-k-1,1} \right\}.
	\end{equation*}
\end{definition}

On the other hand, recall that for an ordered pair $(\alpha,\beta)\in \mathcal Z^*(X,L)$, since $\ell(\alpha)\leq L$ and $\ell(\beta)\leq L$, we have $$\ell(\partial X_{\alpha\beta}) \leq 2L.$$ We divide $Z^*(X,L)$ into two parts
\begin{equation*}
Z^*(X,L) = Z_1^*(X,L) +Z_2^*(X,L)
\end{equation*}
where $Z_1^*(X,L)$ and $Z_2^*(X,L)$ are defined as follow.

\begin{definition}
\begin{equation*}
\mathcal Z_1^*(X,L):= \left\{ (\alpha,\beta)\in\mathcal Z^*(X,L) \,;\,  \ell(\partial X_{\alpha\beta}) \leq 1.9L \right\},
\end{equation*}
\begin{equation*}
Z_1^*(X,L) := \#\mathcal Z_1^*(X,L).
\end{equation*}

\begin{equation*}
\mathcal Z_2^*(X,L):= \left\{ (\alpha,\beta)\in\mathcal Z^*(X,L) \,;\,  \ell(\partial X_{\alpha\beta}) > 1.9L \right\},
\end{equation*}
\begin{equation*}
Z_2^*(X,L) := \#\mathcal Z_2^*(X,L).
\end{equation*}

For $i=1,2$ and $1\leq k\leq \frac{1}{2}(g-1)$, we also define 
\begin{equation*}
\mathcal Z_i^{*0}(X,L):= \mathcal Z^{*0}(X,L)\cap \mathcal Z_i^*(X,L),
\end{equation*}
\begin{equation*}
Z_i^{*0}(X,L) := \#\mathcal Z_i^{*0}(X,L),
\end{equation*}
and
\begin{equation*}
\mathcal Z_i^{*k}(X,L):= \mathcal Z^{*k}(X,L)\cap \mathcal Z_i^*(X,L),
\end{equation*}
\begin{equation*}
Z_i^{*k}(X,L) := \#\mathcal Z_i^{*k}(X,L).
\end{equation*}
\end{definition}

\begin{rem*}
	The value 1.9 is not crucial and can be replaced by any number in the interval $(\frac{5}{3},2)$, where $\frac{5}{3}$ comes from Lemma \ref{lemma:4holded}.
\end{rem*}

We divide the proof of Proposition \ref{E[Z]/E[N]} into the following two lemmas.
\begin{lemma}\label{E[Z1*]}
Let $L=L(g)=2\log g -4\log \log g +\omega(g)$ as before. Then we have as $g\to \infty$,
\begin{equation*}
\E[Z_1^*(X,L)] \leq  c L^6 e^{0.95L}  \frac{1}{g^2}
\end{equation*}
for a universal constant $c>0$.
\end{lemma}

\begin{proof}
By Lemma \ref{multi in S12 S04} we have
\begin{eqnarray*}
	Z_1^{*0}(X,L)
	&=& \sum_{\mbox{\tiny
			$\begin{array}{c}
			\alpha\neq \beta, \alpha\cap\beta \neq \emptyset, \\
			\ell(\partial X_{\alpha\beta})\leq 1.9L, \\
			X\setminus X_{\alpha\beta} = S_{g-2,2} \end{array}$}
	} \mathbf 1_{\sN^*_{1,1}(X,L)}(\alpha) \mathbf 1_{\sN^*_{1,1}(X,L)}(\beta) \\
	&\leq& \sum_{(\alpha,\gamma_1,\gamma_2)} \mathbf 1_{[0,L]}(\ell(\alpha)) \mathbf 1_{[0,1.9L]}(\ell(\gamma_1)+\ell(\gamma_2))  N(Y,\gamma_1,\gamma_2,L) \\
	&\leq& \sum_{(\alpha,\gamma_1,\gamma_2)} \mathbf 1_{[0,L]}(\ell(\alpha)) \mathbf 1_{[0,1.9L]}(\ell(\gamma_1)+\ell(\gamma_2))  \frac{\ell(\gamma_1)}{\mathcal R (\ell(\gamma_1),\ell(\gamma_2),L)}
\end{eqnarray*}
where $(\alpha,\gamma_1,\gamma_2)$ runs over all ordered triples of simple closed geodesics such that 
$\gamma_1\cup\gamma_2$ splits off a subsurface $Y$ of type $S_{1,2}$ with complement $S_{g-2,2}$, while $\alpha$ splits off a one-holed torus in that $S_{1,2}$ (see the first picture in Figure \ref{figure:multicurves}). 

Similarly, for all $1\leq k\leq \frac{1}{2}(g-1)$ we have
\begin{equation*}
Z_1^{*k}(X,L) 
\leq \sum_{(\alpha,\gamma_1,\gamma_2)} \mathbf 1_{[0,L]}(\ell(\alpha)) \mathbf 1_{[0,1.9L]}(\ell(\gamma_1)+\ell(\gamma_2))  \frac{\ell(\gamma_1)}{\mathcal R (\ell(\gamma_1),\ell(\gamma_2),L)}
\end{equation*}
where $(\alpha,\gamma_1,\gamma_2)$ runs over all ordered triples of simple closed geodesics such that $\gamma_1\cup\gamma_2$ splits off an $S_{1,2}$ with complement $S_{k,1}\cup S_{g-k-1,1}$, and $\alpha$ splits off a one-holed torus in that $S_{1,2}$ (see the second picture in Figure \ref{figure:multicurves}).

Then one may apply Mirzakhani's Integration Formula (see Theorem \ref{Mirz int formula}) to get
\begin{eqnarray*}
\int_{\M_g} Z_1^{*0}(X,L) dX
&\leq& \int_{0\leq z\leq L} \int_{0\leq x+y\leq 1.9L;\, x,y\geq 0} \frac{x}{\mathcal R (x,y,L)} \\
& & V_{1,1}(z) V_{0,3}(x,y,z) V_{g-2,2}(x,y) xyz dxdydz 
\end{eqnarray*}
and
\begin{eqnarray*}
	\int_{\M_g} Z_1^{*k}(X,L) dX
	&\leq& \int_{0\leq z\leq L} \int_{0\leq x+y\leq 1.9L;\, x,y\geq 0} \frac{x}{\mathcal R (x,y,L)} \\
	& & V_{1,1}(z) V_{0,3}(x,y,z) V_{k,1}(x) V_{g-k-1,1}(y) xyz dxdydz 
\end{eqnarray*}
for all $1\leq k\leq \frac{1}{2}(g-1)$.

\noindent By Theorem \ref{Mirz vol lemma 0} we know that
\begin{equation*}
V_{1,1}(z) = \frac{1}{48}(z^2+4\pi^2)\quad \text{and} \quad V_{0,3}(x,y,z)=1.
\end{equation*}
\noindent By Lemma \ref{MP vol lemma} we know that
\begin{equation*}
V_{g-2,2}(x,y) \leq \frac{\sinh(\frac{x}{2})\sinh(\frac{y}{2})}{\frac{x}{2}\frac{y}{2}} V_{g-2,2} \leq \frac{e^{\frac{x+y}{2}}}{xy} V_{g-2,2},
\end{equation*}
\begin{equation*}
V_{k,1}(x) \leq \frac{\sinh(\frac{x}{2})}{\frac{x}{2}} V_{k,1} \leq \frac{e^{\frac{x}{2}}}{x} V_{k,1},
\end{equation*}
\begin{equation*}
V_{g-k-1,1}(y) \leq \frac{\sinh(\frac{y}{2})}{\frac{y}{2}} V_{g-k-1,1} \leq \frac{e^{\frac{y}{2}}}{y} V_{g-k-1,1}.
\end{equation*}
\noindent By Lemma \ref{estimation R,D} we know that
\begin{equation*}
\frac{x}{\mathcal R (x,y,L)} \leq 100(1+x)(1+e^{-\frac{x+y}{2}}e^{\frac{L}{2}}).
\end{equation*}
Put all the inequalities above together we have
\begin{eqnarray*}
\int_{\M_g} Z_1^*(X,L) dX
&=& \int_{\M_g} \big( Z_1^{*0}(X,L) + \sum_{1\leq k\leq\frac{1}{2}(g-1)} Z_1^{*k}(X,L) \big) dX\\
&\leq& \frac{100}{48} \big( V_{g-2,2} + \sum_{1\leq k\leq\frac{1}{2}(g-1)} V_{k,1}V_{g-k-1,1} \big) \\
& & \int_{0\leq z\leq L} \int_{0\leq x+y\leq 1.9L;\,x,y\geq 0} \\
& & z(z^2+4\pi^2) e^{\frac{x+y}{2}} (1+x)(1+e^{-\frac{x+y}{2}}e^{\frac{L}{2}}) dxdydz \\
&\leq& c\cdot \big((1+L^6) e^{0.95L} + (1+L^7) e^{\frac{L}{2}}\big) \\
& & \big( V_{g-2,2} + \sum_{1\leq k\leq\frac{1}{2}(g-1)} V_{k,1}V_{g-k-1,1} \big)
\end{eqnarray*}
for some universal constant $c>0$. And by Lemma \ref{Mirz vol lemma 1} and \ref{Mirz vol lemma 2} we know that
\begin{equation*}
V_{g-2,2}=\frac{1}{(8\pi^2 g)^2}V_g (1+O\left(\frac{1}{g}\right))
\end{equation*}
and
\begin{equation*}
\sum_{1\leq k\leq\frac{1}{2}(g-1)} V_{k,1}V_{g-k-1,1} =O\left( \frac{V_g}{g^3}\right).
\end{equation*}
So as $g\to \infty$, we have
\begin{equation*}
\E[Z_1^*(X,L)] \leq c L^6 e^{0.95L}  \frac{1}{g^2}
\end{equation*}
for some universal constant $c>0$, as required.
\end{proof}

\begin{figure}[h]
	\centering
	\includegraphics[width=12.5cm]{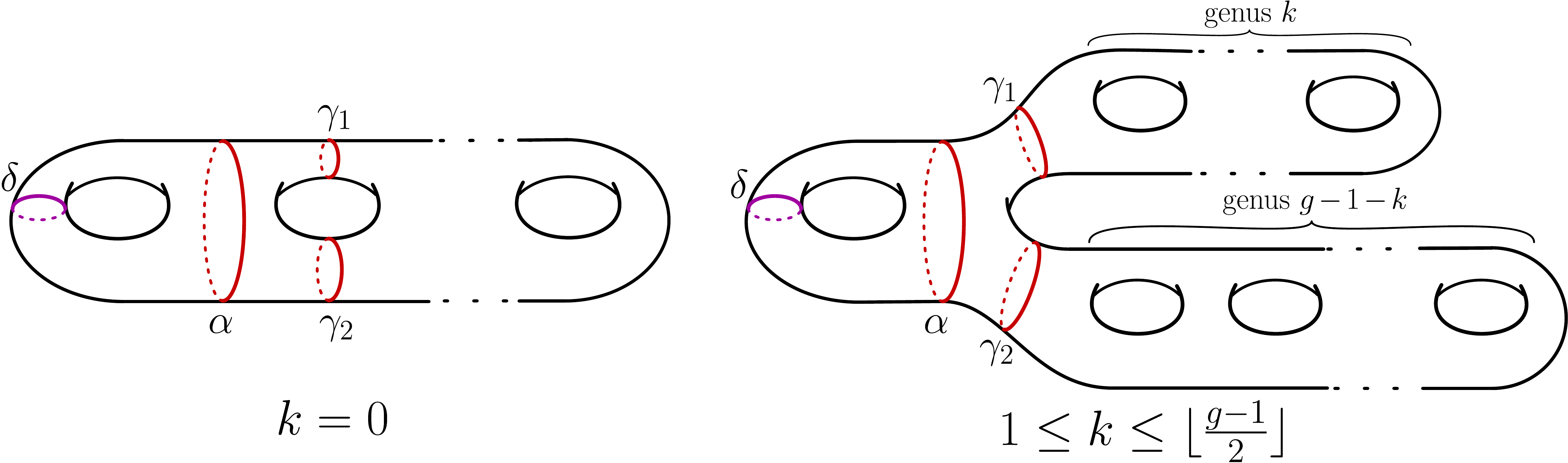}
	\caption{}
	\label{figure:multicurves}
\end{figure}

Now we estimate the expectation of $Z_2^*(X,L)$.
\begin{lemma}\label{E[Z2*]}
Let $L=L(g)=2\log g -4\log \log g +\omega(g)$ as before. Then we have as $g\to \infty$,
\begin{equation*}
\E[Z_2^*(X,L)] \leq c L^3 e^L \frac{1}{g^2}
\end{equation*}
for a universal constant $c>0$.
\end{lemma}

\begin{proof}
Assume that $(\alpha,\beta)\in \mathcal Z_2^*(X,L)$. Denote the boundary of $X_{\alpha\beta}$ to be the two simple closed geodesics $\gamma_1$ and $\gamma_2$. Then we have
\begin{equation*}
1.9L<\ell(\gamma_1)+\ell(\gamma_2)\leq 2L.
\end{equation*}

\noindent By Lemma \ref{lemma:4holded} we know that $\alpha$ and $\beta$ have exactly 4 intersection points, and the intersection $X_\alpha \cap X_\beta$ contains a simple closed geodesic $\delta$ which is disjoint with $\alpha,\beta,\gamma_1,\gamma_2$. (see Figure \ref{a new geodesic in S12}.) Since $\alpha\cup\beta$ is homotopic to $\gamma_1\cup\gamma_2\cup2\delta$ (see the remark after Lemma \ref{lemma:4holded}), we have
\begin{equation*}
\ell(\gamma_1)+\ell(\gamma_2)+2\ell(\delta) \leq \ell(\alpha)+\ell(\beta) \leq 2L.
\end{equation*}

\noindent Now similar to how we deal with $Z_1^*(X,L)$, by Lemma \ref{multi in S12 S04} we have
\begin{eqnarray*}
Z_2^{*0}(X,L)
&=& \sum_{\mbox{\tiny
		$\begin{array}{c}
		\alpha\neq \beta, \alpha\cap\beta \neq \emptyset, \\ 
		1.9L<\ell(\partial X_{\alpha\beta})\leq 2L, \\
		X\setminus X_{\alpha\beta} = S_{g-2,2} \end{array}$}
} \mathbf 1_{\sN^*_{1,1}(X,L)}(\alpha) \mathbf 1_{\sN^*_{1,1}(X,L)}(\beta) \\
&\leq& \sum_{(\alpha,\gamma_1,\gamma_2,\delta)} \mathbf 1_{[1.9L,2L]}(\ell(\gamma_1)+\ell(\gamma_2)) \mathbf 1_{[0,L]}(\ell(\alpha)) \\
& & \mathbf 1_{[0,2L]}(\ell(\gamma_1)+\ell(\gamma_2)+2\ell(\delta)) \cdot N(Y_{\gamma_1\gamma_2\delta},\gamma_1,\gamma_2,L) \\
&\leq& \sum_{(\alpha,\gamma_1,\gamma_2,\delta)} \mathbf 1_{[1.9L,2L]}(\ell(\gamma_1)+\ell(\gamma_2)) \mathbf 1_{[0,L]}(\ell(\alpha)) \\
& & \mathbf 1_{[0,2L]}(\ell(\gamma_1)+\ell(\gamma_2)+2\ell(\delta)) \frac{\ell(\gamma_1)}{\mathcal R (\ell(\gamma_1),\ell(\gamma_2),L)} 
\end{eqnarray*}
where $(\alpha,\gamma_1,\gamma_2,\delta)$ runs over all ordered quadruples of simple closed geodesics such that $\gamma_1\cup\gamma_2$ splits off an $S_{1,2}$ with complement $S_{g-2,2}$, $\alpha$ splits off a one-holed torus from that $S_{1,2}$, and $\delta$ is in that one-holed torus (see the first picture in Figure \ref{figure:multicurves}). We let $Y_{\gamma_1\gamma_2\delta}$ denote the hyperbolic surface of type $S_{0,4}$ split off by $\gamma_1$, $\gamma_2$ and $\delta$.

Similarly, for all $1\leq k\leq \frac{1}{2}(g-1)$ we have
\begin{eqnarray*}
	Z_2^{*k}(X,L)
	&\leq& \sum_{(\alpha,\gamma_1,\gamma_2,\delta)} \mathbf 1_{[1.9L,2L]}(\ell(\gamma_1)+\ell(\gamma_2)) \mathbf 1_{[0,L]}(\ell(\alpha)) \\
	& & \mathbf 1_{[0,2L]}(\ell(\gamma_1)+\ell(\gamma_2)+2\ell(\delta)) \frac{\ell(\gamma_1)}{\mathcal R (\ell(\gamma_1),\ell(\gamma_2),L)} 
\end{eqnarray*}
where $(\alpha,\gamma_1,\gamma_2,\delta)$ runs over all ordered quadruples of simple closed geodesics such that $\gamma_1\cup\gamma_2$ splits off an $S_{1,2}$ with complement $S_{k,1}\cup S_{g-k-1,1}$, $\alpha$ splits off a one-holed torus from that $S_{1,2}$, and $\delta$ is in that one-holed torus (see the second picture in Figure \ref{figure:multicurves}).

When $$L<1.9L<\ell(\gamma_1)+\ell(\gamma_2)\leq 2L,$$ it follows by Lemma \ref{estimation R,D} that
\begin{equation*}
\frac{\ell(\gamma_1)}{\mathcal R (\ell(\gamma_1),\ell(\gamma_2),L)} \leq 500+500\frac{\ell(\gamma_1)}{0.9L} < 2000.
\end{equation*}

So we have
\begin{equation*}
Z_2^{*0}(X,L) \leq 2000\sum_{(\alpha,\gamma_1,\gamma_2,\delta)} \mathbf 1_{[0,L]}(\ell(\alpha)) \mathbf 1_{[1.9L,2L]}(\ell(\gamma_1)+\ell(\gamma_2)+2\ell(\delta))
\end{equation*}
and
\begin{equation*}
Z_2^{*k}(X,L) \leq 2000\sum_{(\alpha,\gamma_1,\gamma_2,\delta)} \mathbf 1_{[0,L]}(\ell(\alpha)) \mathbf 1_{[1.9L,2L]}(\ell(\gamma_1)+\ell(\gamma_2)+2\ell(\delta)).
\end{equation*}

\noindent By Theorem \ref{Mirz vol lemma 0} we know that
\begin{equation*}
\quad V_{0,3}(x,y,z)=1.
\end{equation*}
\noindent By Lemma \ref{MP vol lemma} we know that
\begin{equation*}
V_{g-2,2}(x,y) \leq \frac{\sinh(\frac{x}{2})\sinh(\frac{y}{2})}{\frac{x}{2}\frac{y}{2}} V_{g-2,2} \leq \frac{e^{\frac{x+y}{2}}}{xy} V_{g-2,2}.
\end{equation*}

Then one may apply Mirzakhani's Integration Formula (see Theorem \ref{Mirz int formula}) to get
\begin{eqnarray*}
&& \int_{\M_g} Z_2^{*0}(X,L) dX \\
&\leq&  \int_{0\leq z\leq L} \int_{1.9L\leq x+y+2w\leq 2L;\,x,y,w\geq 0} 2000 \\
& & V_{0,3}(z,w,w) V_{0,3}(x,y,z) V_{g-2,2}(x,y) xyzw dxdydzdw \\
&\leq& 2000 V_{g-2,2} \int_{0\leq z\leq L} \int_{1.9L\leq x+y+2w\leq 2L;\,x,y,w\geq 0} e^{\frac{x+y}{2}}zwdxdydzdw \\
&\leq& c V_{g-2,2} L^3 e^{L}
\end{eqnarray*}
for some universal constant $c>0$.

Similarly, for all $1\leq k \leq \frac{1}{2}(g-1)$ we have
\begin{eqnarray*}
	&& \int_{\M_g} Z_2^{*k}(X,L) dX \\
	&\leq&  \int_{0\leq z\leq L} \int_{1.9L\leq x+y+2w\leq 2L;\,x,y,w\geq 0} 2000 \\
	& & V_{0,3}(z,w,w) V_{0,3}(x,y,z) V_{k,1}(x) V_{g-k-1,1}(y) xyzw dxdydzdw \\
	&\leq& 2000 V_{k,1}V_{g-k-1,1} \int_{0\leq z\leq L} \int_{1.9L\leq x+y+2w\leq 2L;\,x,y,w\geq 0} e^{\frac{x+y}{2}}zwdxdydzdw \\
	&\leq& c V_{k,1}V_{g-k-1,1} L^3 e^{L}
\end{eqnarray*}
for some universal constant $c>0$. 

And by Lemma \ref{Mirz vol lemma 1} and \ref{Mirz vol lemma 2} we know that
\begin{equation*}
V_{g-2,2}=\frac{1}{(8\pi^2 g)^2}V_g (1+O\left(\frac{1}{g}\right))
\end{equation*}
and
\begin{equation*}
\sum_{1\leq k\leq\frac{1}{2}(g-1)} V_{k,1}V_{g-k-1,1} =O\left( \frac{V_g}{g^3}\right).
\end{equation*}

Therefore we have
\begin{eqnarray*}
\E[Z_2^*(X,L)] 
&=& \E[Z_2^{*0}(X,L)] + \sum_{1\leq k\leq\frac{1}{2}(g-1)} \E[Z_2^{*k}(X,L)] \\ 
&\leq& c L^3 e^L \cdot \frac{V_{g-2,2} + \sum \limits_{1\leq k\leq\frac{1}{2}(g-1)} V_{k,1}V_{g-k-1,1}}{V_g} \\ 
&\leq& c L^3 e^L \frac{1}{g^2}
\end{eqnarray*}
for some universal constant $c>0$, as required.
\end{proof}

Now are are ready to prove Proposition \ref{E[Z]/E[N]}.

\bp [Proof of Proposition \ref{E[Z]/E[N]}]
Recall that $L=L(g)=2\log g -4\log \log g +\omega(g)$. Thus, it follows by Lemma \ref{E[Z1*]} and \ref{E[Z2*]} that there exists a universal constant $c>0$ such that
\begin{eqnarray*}
\E[Z^*(X,L)]
&=& \E[Z_1^*(X,L)]+\E[Z_2^*(X,L)] \\
&\leq& c \left(L^6e^{0.95L} + L^3 e^L\right) \frac{1}{g^2} \\
&\leq& c L^3 e^L \frac{1}{g^2}.
\end{eqnarray*}

For \eqref{tend to 0 (3)}, by Lemma \ref{E[N]} we know that as $g\to \infty$,
\[\E[N_{1,1}(X,L)] \sim \frac{1}{384\pi^2} L^2 e^{\frac{L}{2}} \frac{1}{g}.\]
Thus, we have
\begin{equation*}
\frac{\E[Z^*(X,L)]}{\E[N^*_{1,1}(X,L)]^2}=O(\frac{1}{L}) \rightarrow 0
\end{equation*}
as $g\rightarrow\infty$, which proves \eqref{tend to 0 (3)}.
\ep

Now we finish the proof of Theorem \ref{prop upper bound}.
\bp [Proof of Theorem \ref{prop upper bound}]
By the definition of $N_{1,1}(X,L)$ we have
\begin{eqnarray*}
&&\limg\Prob\big(X\in \sM_g\,;\, X\in\mathcal A(\omega(g)) \big)=\limg\Prob\big(N_{1,1}(X,L)\geq 1 \big)\\
&&=1-\limg\Prob\big(N_{1,1}(X,L)=0 \big).
\end{eqnarray*}

\noindent By Proposition \ref{N-N*}, \ref{E[Y]-E[Y*]}, \ref{E[Z]/E[N]} and Equation \eqref{prob(N*=0) leq 3 parts} we have
\begin{equation*}
\limg\Prob\big(N^*_{1,1}(X,L)=0 \big)=0.
\end{equation*}
Since $N^*_{1,1}(X,L)\leq N_{1,1}(X,L)$, we have
\begin{equation*}
\limg\Prob\big(N_{1,1}(X,L)=0 \big)=0,
\end{equation*}
as required.
\ep

\section{Half-collars and separating extremal length systole}\label{sec:half collar}

In this section, we prove Theorem \ref{thm:half collar} and \ref{cor:extremal}.
\subsection{Half-collars}\label{subsection half collar}
While a \emph{collar} of a simple closed geodesic $\gamma$ on a complete hyperbolic surface means an equidistant neighborhood $U$ of $\gamma$ homeomorphic to a cylinder, we will mainly consider a half of $U$ cut out by $\gamma$, which we call a \emph{half-collar}. More precisely:
\begin{definition}
	Given $l,w>0$, let $C_{l,w}$ denote the hyperbolic cylinder such that one of the two boundary components is a simple closed geodesic of length $l$, while every point on the other component has distance $w$ from the geodesic component
		\begin{figure}[h]
		\centering	
		\includegraphics[width=3cm]{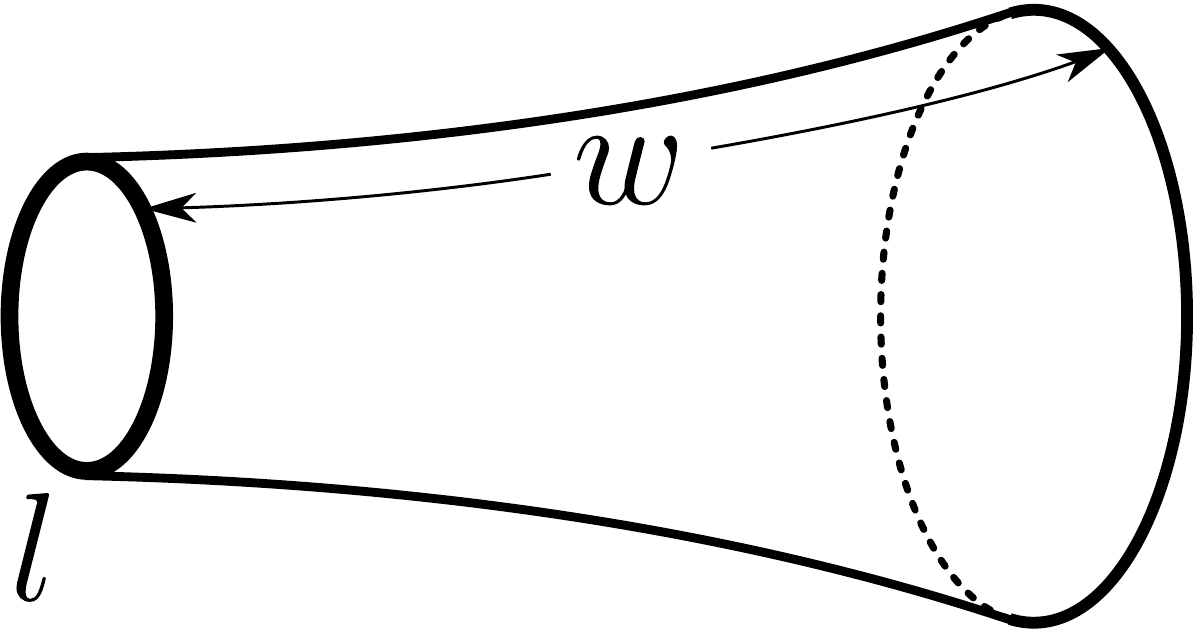}
		\caption{The cylinder $C_{l,w}$.}
		\label{figure:collar}
	\end{figure}
	 (see Figure \ref{figure:collar}).
	Given a hyperbolic surface $X$ and a simple closed geodesic $\gamma\subset X$, a \emph{half-collar} of width $w$ around  $\gamma$ is a subsurface $C\subset X$ isometric to $C_{\ell_\gamma(X),w}$ such that $\gamma$ is the geodesic boundary component of $C$.
\end{definition}

The following result is standard:
\begin{lemma}\label{prop:collar}
	Let $X$ be a compact hyperbolic surface of type $S_{g,1}$ with geodesic boundary $\gamma:=\partial X\cong\mathbb{S}^1$. Given $w>0$, the following conditions are equivalent:
	\begin{enumerate}[label=(\roman*)]
		\item\label{item:collar1} there is no half-collar of width $w$ around $\gamma$;
		\item\label{item:collar2} there is a simple geodesic arc $a\subset X$ of length $\leq2w$ with endpoints in $\gamma$. 
	\end{enumerate}
\end{lemma}

\begin{proof}
	We first prove the implication ``\ref{item:collar2}$\Rightarrow$\ref{item:collar1}'' by showing that if \ref{item:collar1} fails then \ref{item:collar2} fails as well.
	So suppose there is a half-collar $C$ of width $w$ around $\gamma$ and let $a$ be any geodesic arc with endpoints in $\gamma$. Since $a$ cannot be entirely contained in $C$ (otherwise it would give rise to a geodesic bigon, which is impossible), there are disjoint sub-arcs $a_1,a_2\subset a$, each of which joins a point of $\gamma$ with a point in the non-geodesic boundary component of $C$. It follows that 
	$\ell(a)>\ell(a_1)+\ell(a_2)\geq 2w$, hence \ref{item:collar2} fails.
	We have thus shown ``\ref{item:collar2}$\Rightarrow$\ref{item:collar1}''.
	
	As for the implication ``\ref{item:collar1}$\Rightarrow$\ref{item:collar2}'', consider the $\epsilon$-neighborhood $U_\epsilon:=\{x\in X\,;\,d(x,\gamma)<\epsilon\}$ of $\gamma$ in $X$. When $\epsilon$ is small enough, the closure $\overline{U}_\epsilon$ is homeomorphic to a cylinder with boundary, hence is a half-collar of width $\epsilon$. As $\epsilon$ grows larger, there is a critical value $\epsilon_0$ such that $\overline{U}_{\epsilon_0}$ stops to be a cylinder for the first time, which is characterized by the existence of a point $x_0\in X$ with $d(x,\gamma)=\epsilon_0$ such that $\overline{U}_{\epsilon_0}$ touches itself at $x_0$. One can then draw two geodesic segments of length $\epsilon_0$ from $x_0$ to $\gamma$ which fit together to form a simple geodesic arc of length $2\epsilon_0$.
	
	Now, Condition \ref{item:collar1} just means $\epsilon_0\leq w$. In this case, the arc that we just constructed implies that \ref{item:collar2} holds. This shows ``\ref{item:collar1}$\Rightarrow$\ref{item:collar2}''.
\end{proof}

We can now prove Theorem \ref{thm:half collar} by using Theorem \ref{main} and Proposition \ref{lower bound for chi geq 2}.
\begin{theorem}[=Theorem \ref{thm:half collar}]\label{thm:half collar-1}
Given $\epsilon>0$,  consider the following conditions defined for all $X\in \M_g$:
\begin{itemize}
\item[(c).] There is a half-collar around $\gamma$ in the $S_{g-1,1}$-part of $X$ with width $\frac{1}{2}\log g-\left(\frac{3}{2}+\epsilon\right)\log\log g$;
\item[(d).] $\ell_{\sys}^{\rm sep}(X)$ is achieved by a simple closed geodesic $\gamma$ separating $X$ into $S_{1,1}\cup S_{g-1,1}$;
\end{itemize}
Then we have
$$
\lim \limits_{g\to \infty} \Prob\left(X\in \M_g\,;\, \textit{$X$ satisfies $(c)$ and $(d)$} \right)=1.
$$ 
\end{theorem}

\begin{proof}
	Fix the function $\omega(g)$ as in Theorem \ref{main} and let $\mathcal{A}_g$ 
	and $\mathcal{B}_g$ denote the subsets of $\M_g$ as follow. 
	$$
	\mathcal{A}_g:=\left\{X\in\M_g\,;\ \parbox[l]{7cm}{$|\ell_{\sys}^{\rm sep}(X)-(2\log g - 4\log \log g)| \leq \omega(g)$ \\
	and	$\ell_{\sys}^{\rm sep}(X)$ is achieved only by simple\\ closed geodesics bounding one-holed torus}\right\}
	$$
	$$
	\mathcal{B}_g:=\{X\in\M_g\,;\, \sL_{1,2}(X)> 4\log g-10\log\log g-\omega(g)\}.
	$$
	
	Fix $\epsilon>0$ and suppose $g\geq3$ satisfies
	\begin{equation}\label{eqn:proof half collar}
	\omega(g)< 2\epsilon\log\log g.
	\end{equation}
	We claim that every $X\in\mathcal{A}_g\cap\mathcal{B}_g$ satisfies the condition stated in Theorem \ref{thm:half collar}. That is, for any simple closed geodesic $\gamma$ achieving $\lss(X)$, which separates $X$ into $S_{1,1}\cup S_{g-1,1}$ because $X\in\mathcal{A}_g$, there is a half-collar around $\gamma$ in the $S_{g-1,1}$-part of $X$ with width $\frac{1}{2}\log g-\left(\frac{3}{2}+\epsilon\right)\log\log g$.
	
	Suppose by contradiction that the claim is false. Then by Lemma \ref{prop:collar}, there exists an $X\in\mathcal{A}_g\cap\mathcal{B}_g$, a simple closed geodesic $\gamma\subset X$ achieving $\lss(X)$, and a simple geodesic arc $a$ in the $S_{g-1,1}$-part of $X$ with endpoints on $\gamma$, such that 
	$$
	\ell(a)\leq\log g-\left(3+2\epsilon\right)\log\log g.
	$$
	In this situation, there are simple closed geodesics $\gamma_1$ and $\gamma_2$ homotopic to the two closed piecewise geodesics formed by $a$ and the two arcs of $\gamma$ split out by $a$, respectively (see Figure \ref{figure:arc}), 
	\begin{figure}[h]
		\includegraphics[width=4.3cm]{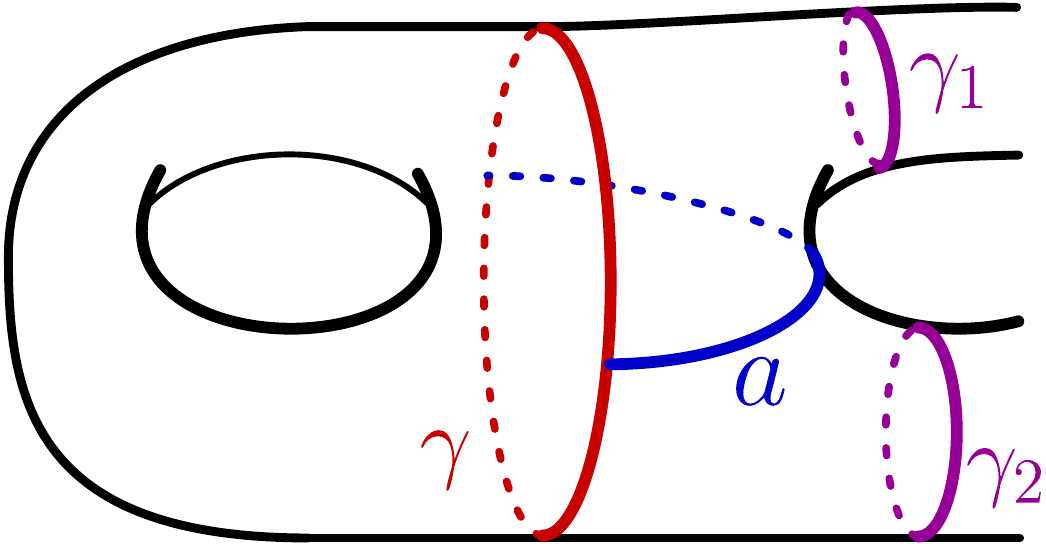}
		\caption{From an arc to a pair-of-pants}
		\label{figure:arc}
	\end{figure} 
	such that $\gamma_1$, $\gamma_2$ and $\gamma$ together bound a pair of pants outside of the one-holed torus $X_\gamma$.
	
	Since each of $\gamma_1$ and $\gamma_2$ is shorter than the corresponding closed piecewise geodesic, we have
	\begin{align}
	&\ell(\gamma_1)+\ell(\gamma_2)\leq \ell(\gamma)+2\ell(a)=\lss(X)+2\ell(a)\label{eqn:proof half collar 2}\\
	&\leq 2\log g-4\log\log g+\omega(g)+2\big(\log g-(3+2\epsilon)\log\log g\big)\nonumber\\
	&=4\log g-(10+4\epsilon)\log\log g+\omega(g).\nonumber
	\end{align}
	But on the other hand, by definition of $\sL_{1,2}(X)$ (see the definition in the Introduction) and the assumption $X\in\mathcal{B}_g$, we have
	$$
	\ell(\gamma_1)+\ell(\gamma_2)\geq\sL_{1,2}(X)> 4\log g-10\log\log g-\omega(g).
	$$
	This leads to a contradiction because  by \eqref{eqn:proof half collar}, the lower bound of $\ell(\gamma_1)+\ell(\gamma_2)$ here is greater than the upper bound in \eqref{eqn:proof half collar 2}. We have thus shown the claim.
	
	As $g\to\infty$, since $\Prob(\mathcal{A}_g)$ and $\Prob(\mathcal{B}_g)$ both tend to $1$ by Theorem \ref{main} and Proposition \ref{lower bound for chi geq 2}, we have
	$\Prob(\mathcal{A}_g\cap\mathcal{B}_g)\to 1$ as well. 
	In view of the above claim, this implies the required statement.
\end{proof}

\subsection{Extremal length}\label{subsec-el}
Given a Riemann surface $U$ and a set $\Gamma$ of rectifiable curves on $U$, the \emph{extremal length} $\exl_\Gamma(U)$ of $\Gamma$ is defined as (\eg see \cite[Chapter 4]{Ahlfors-ci} and \cite[Section 3]{Kerck80})
$$
\exl_\Gamma(U):=\sup_{\sigma}\frac{\inf_{\alpha\in\Gamma}\ell_\sigma(\alpha)^2}{A_\sigma(U)},
$$
where the supremum is over all Borel-measurable conformal metrics $\sigma$ on $U$, and $\ell_\sigma(\alpha)$ and $A_\sigma(U)$ denote the length of $\alpha$ and the area of $U$ under $\sigma$, respectively.
In particular, given a closed hyperbolic surface $X\in\M_g$ and a simple closed geodesic $\gamma\subset X$, we denote 
$$
\exl_\gamma(X):=\exl_{\Gamma_\gamma}(X)
$$
for the set $\Gamma_\gamma$ of all rectifiable closed curves on $X$ homotopic to $\gamma$. We then define the \emph{separating extremal length systole} $\exlss(X)$ of $X$ as 
$$
\exlss(X):=\inf_\gamma\exl_\gamma(X),
$$
where the infimum is over all separating simple closed geodesics on $X$.

Maskit \cite{Maskit} established some basic relations between the extremal length $\exl_\gamma(X)$ and the hyperbolic length $\ell_\gamma(X)$. The following lemma is a reformulation of \cite[Prop.\@ 1]{Maskit}:
\begin{lemma}\label{lemma:extremal}
	Let $X\in\M_g$. For any simple closed geodesic $\gamma\subset X$, we have
	$$
	\ell_\gamma(X)\leq \pi \exl_\gamma(X).
	$$
	Conversely, if there exists a half-collar around $\gamma$ with width $w$, then
	$$
	\ell_\gamma(X)\geq2\big(\arctan(e^w)-\tfrac{\pi}{4}\big)\exl_\gamma(X).
	$$
\end{lemma}
\begin{proof}
	The first inequality is exactly Inequality (2) in \cite[Prop.\@ 1]{Maskit}. On the other hand, Inequality (1) in \cite[Prop.\@ 1]{Maskit} implies that if we identify the universal cover of $X$ with the upper half-plane $\mathbb{H}^2$ in such a way that $\gamma$ lifts to $\boldsymbol{i}\R_+$, and assume that $\gamma$ has a half-collar $C$  which lifts to 
	$$
	\left\{z\,;\,\tfrac{\pi}{2}-\theta\leq \arg(z)\leq \tfrac{\pi}{2}\right\}
	$$
	for some $\theta\in(0,\frac{\pi}{2})$, then $\ell_\gamma(X)\geq \theta\exl_\gamma(X)$. By an elementary hyperbolic-geometric calculation, the width $w$ of $C$ is related to $\theta$ by 
	$\cosh(w)=\frac{1}{\cos\theta}$, which is equivalent to
	$2\big(\arctan(e^w)-\tfrac{\pi}{4}\big)=\theta$ (this can be seen by using the trigonometric identity $\tan(\phi)+\cot(\phi)=2\csc(2\phi)$). The second required inequality follows.
\end{proof}

We can now deduce Theorem \ref{cor:extremal} from Theorem \ref{main}.
\begin{theorem}[=Theorem \ref{cor:extremal}]\label{cor:extremal-1}
Given any $\epsilon>0$, then we have
$$
\lim \limits_{g\to \infty} \Prob\left(X\in \M_g\,;\, \frac{\exlss(X)}{\lss(X)}< \frac{2+\epsilon}{\pi} \right)=1.
$$
As a consequence of Theorem \ref{main}, we have
$$
\lim \limits_{g\to \infty} \Prob\left(X\in \M_g\,;\, \frac{(2-\epsilon)}{\pi}\log g< \exlss(X)< \frac{(4+\epsilon)}{\pi}\log g \right)=1.
$$
\end{theorem}

\begin{proof}
	Let $\mathcal{A}_g$ denote the subset of $\M_g$ consisting of those $X\in\M_g$ satisfying the conditions in Theorem \ref{main} and \ref{thm:half collar}. The sequence $(\mathcal{A}_g)$ has the property that given any $w>0$, every $X\in\mathcal{A}_g$ with $g$ large enough contains a half-collar of width $w$ around some separating simple closed geodesic $\gamma$ with $\ell_\gamma(X)=\lss(X)$. 
	
	Now fix $\epsilon>0$ and let $w_\epsilon>0$ be large enough such that
	$$
	\frac{1}{2(\arctan(e^{w_\epsilon})-\frac{\pi}{4})}\leq\frac{2+\epsilon}{\pi}.
	$$
	For every $X\in\mathcal{A}_g$ with $g$ large enough, letting $\gamma\subset X$ be the separating simple closed geodesic described in the property above, which achieves $\lss(X)$ and has a half-collar of width $w_\epsilon$. By Lemma \ref{lemma:extremal}, we have
	\begin{eqnarray*}
	\frac{\exlss(X)}{\lss(X)}&=&\frac{\exlss(X)}{\ell_\gamma(X)}\\
&\leq&\frac{\exl_\gamma(X)}{\ell_\gamma(X)}\\
&\leq& \frac{1}{2(\arctan(e^{w_\epsilon})-\frac{\pi}{4})}\\
&\leq&\frac{2+\epsilon}{\pi}.
	\end{eqnarray*}
	Therefore, the first statement of the theorem follows from Theorem \ref{main}. The second statement is then a consequence of the first statement, the fact that
	$$
	\lss(X)\leq\pi\exlss(X),
	$$
	which follows from Lemma \ref{lemma:extremal}, and the fact that 
	$$
	\lim_{g\to\infty}\Prob\big(X\in\M_g\,;\,(2-\epsilon)\log g< \lss(X)< (2+\epsilon)\log g\big)=0,
	$$
	which follows from Theorem \ref{main}. The proof is now complete.
\end{proof}

\section{Non-simple systole and expected value of $\sL_1$}\label{section exp-L1}
In this section consider the non-simple systole and the expected value of $\mathcal{L}_1$ over $\sM_g$ for large genus. First we provide the following elementary property needed in the proofs of Theorem \ref{cor:ns systole} and \ref{cor E[L1]}. 
\begin{proposition}\label{ns-sub}
	Let $X\in \sM_g$ and $\gamma\subset X$ be a self-intersecting closed geodesic. Then there exists a connected subsurface $X(\gamma)$ of $X$ such that 
	\begin{itemize}
		\item[(1)]  $\gamma\subset X(\gamma)$;  
		
		\item[(2)]  the boundary $\partial X(\gamma)$ is a simple closed multi-geodesic with
		\[\ell\left(\partial X(\gamma)\right) \leq 2\ell(\gamma);\]
		
		\item[(3)] $\area(X(\gamma)) \leq 3\ell(\gamma)$.
	\end{itemize}
\end{proposition}
\bp
The construction for $X(\gamma)$ is the same as the one for $X_3$ at the beginning of Section \ref{section union}, only with the role of $X_1\cup X_2$ replaced by $\gamma$. Namely, we get $X(\gamma)$ by deforming the boundary components of $\gamma$ (or more precisely, boundary components of an $\varepsilon$-neighborhood of $\gamma$ for small enough $\varepsilon$) into simple closed geodesics in the way described there. Property (2) is clear from the construction.

By construction $\gamma$ (or its small $\varepsilon$-neighborhood) is freely homotopic to a subset of $X(\gamma)$. Since $X(\gamma)$ is a surface with geodesic boundary, the unique closed geodesic representing this free homotopy class should be contained in $X(\gamma)$. So property (1) holds.

%

For (3), we write 
$$
X(\gamma)\setminus \gamma= (\sqcup D_i) \sqcup (\sqcup C_j)
$$ 
where the $D_i$'s and the $C_j$'s are topological discs and cylinders, respectively, all disjoint from each other. By the classical Isoperimetric Inequality (\eg see \cite{Buser10, WX18}), we know that
\[\area(D_i)\leq \ell(\partial D_i) \quad \text{and} \quad \area(C_j)\leq \ell(\partial C_j).\]
Thus, we have
\begin{eqnarray*}
	\area(X(\gamma))&=&\area (X(\gamma)\setminus \gamma)=\textstyle\sum_i\area(D_i)+\textstyle\sum_j\area(C_j)\\
	&\leq & \textstyle\sum_i\ell(\partial D_i)+\textstyle\sum_j\ell(\partial C_j)\\
	&\leq & \ell(\partial X(\gamma))+\ell(\gamma) \\
	&\leq & 3\ell_{\gamma}(X),
\end{eqnarray*}
as required.
\ep

\begin{rem*}
The multi-geodesic $\partial X(\gamma)$ is empty if $\gamma$ is filling in $X$.
\end{rem*}

Now we are ready to prove Theorem \ref{cor:ns systole}.
\bt[=Theorem \ref{cor:ns systole}]
Given any $\epsilon>0$, then we have
$$
\lim \limits_{g\to \infty} \Prob\left(X\in \M_g\,;\, (1-\epsilon)\log {\textit{g}}< \lns(X)< 2\log {\textit{g}} \right)=1.
$$
\et
\bp
	Fix the function $\omega(g)$ as in Theorem \ref{main} and let $\mathcal{C}_g$ 
	and $\mathcal{D}_g$ denote the subsets of $\M_g$ as follow. 
	$$
	\mathcal{C}_g:=\left\{X\in\M_g\,;\ \parbox[l]{7cm}{$|\ell_{\sys}^{\rm sep}(X)-(2\log g - 4\log \log g)| \leq \omega(g)$\\ 
	and $\ell_{\sys}^{\rm sep}(X)$ is achieved only by simple closed geodesics bounding one-holed torus}\right\}
	$$
	$$
	\mathcal{D}_g:=\{X\in\M_g\,;\,\sL_{1}(X) \geq 2\log g-4\log\log g-\omega(g)\}.
	$$
By Theorem \ref{main} and \ref{cor L1} we know that
$$
\lim \limits_{g\to \infty} \Prob\left(X \in \mathcal{C}_g\cap \mathcal{D}_g\right)=1.
$$
So it suffices to show that given any $\epsilon>0$, as $g\to \infty$, for any $X\in \mathcal{C}_g\cap \mathcal{D}_g$ we have
\[(1-\epsilon)\log {\textit{g}}< \lns(X)< 2\log {\textit{g}}.\]

We first show the upper bound: $\lns(X)< 2\log {\textit{g}}$. Since $X\in \mathcal{C}_g$, one may let $\gamma\subset X$ be a shortest simple separating closed geodesic bounding a one-holed torus $X_{1,1}$. Let $w(\gamma)>0$ be the maximal value such that the half collar $C(w(\gamma))$ is embedded in $X_{1,1}$ where \[C(w(\gamma)):=\{z\in X_{1,1}\,;\, \dist(z, \gamma)<w(\gamma)\}.\]
It is known that the hyperbolic metric $ds^2$ on $C(w(\gamma))$ satisfies that
\[ds^2=d\rho^2+ \left(\ell_{\sys}^{\rm sep}(X)\cosh \rho\right)^2 dt^2\]
where $\rho \in [0,w(\gamma))$ is the distance to $\gamma$ and $t \in [0,1]$. Since $\area(X_{1,1})=2\pi$ and $C(w(\gamma))$ is embedded in $X_{1,1}$, we have $$\area(C(w(\gamma)))=\ell_{\sys}^{\rm sep}(X)\sinh \left(w(\gamma)\right)<2\pi.$$
By the choice of $w(\gamma)$ we know that one component $\alpha$ of the boundary of the half collar $C(w(\gamma))$ is a non-simple closed curve. By continuity we know that the length $\ell(\alpha)$ satisfies that
\beqar
\ell(\alpha)&=&\ell_{\sys}^{\rm sep}(X)\cosh \left(w(\gamma)\right)\\
&=&\sqrt{\ell_{\sys}^{\rm sep}(X)^2+\ell_{\sys}^{\rm sep}(X)^2\sinh^2 \left(w(\gamma)\right)}\\&<&\sqrt{\ell_{\sys}^{\rm sep}(X)^2+(2\pi)^2}\\
&<& 2\log g
\eeqar
for large $g$ because $X\in \mathcal{C}_g$. Then the unique closed geodesic in $X_{1,1}$ representing $\alpha$ gives the upper bound.

Now we show the lower bound: $(1-\epsilon)\log {\textit{g}}< \lns(X)$. First by above one may assume that $\gamma' \subset X$ is a shortest non-simple closed geodesic with $\ell_{\gamma'}(X)=\lns(X)<2 \log g$ for large $g$. Then by Proposition \ref{ns-sub} there exists a subsurface $X(\gamma')$ of $X$ such that 
\[\ell(\partial (X(\gamma'))) \leq 2 \lns(X)<4\log g \ \text{and} \ \area(X(\gamma'))\leq 3\lns(X)<6\log g.\]
Recall that by Gauss-Bonnet we know that $\area(X)=4\pi(g-1)$. So for large $g$, we have that $X(\gamma')$ is a proper subsurface of $X$. Clearly the boundary $\partial X(\gamma')$ consists of multi simple closed geodesics which separate $X$. Hence, we have
$$\mathcal{L}_1(X)\leq \ell(\partial (X(\gamma')))\leq 2 \lns(X).$$ 
Then the lower bound follows because $X\in \mathcal{D}_g$.
\ep

Before proving Theorem \ref{cor E[L1]}, unlike the unboundness of $\ell_{\sys}^{\rm sep}$ we first show that 
\begin{proposition}\label{L1-upp}
There exists a universal constant $C>0$ independent of $g$ such that
$$\sup_{X\in \sM_g}\mathcal{L}_1(X)\leq C \log g.$$
\end{proposition}

\begin{proof}
Since each $X_g\in \sM_g$ admits a pants decomposition such that each geodesic has length no more than the Bers' constant depending on $g$ (see \eg \cite[Theorem 5.1.2]{Buser10}), $\sup_{X\in \sM_g} \mathcal{L}_1(X) <\infty$ has a upper bound only depending on $g$. So it suffices to consider the case that $g$ is large. For any $X\in \sM_g$, by \cite[Theorem 1.3]{Sabo08} of Sabourau we know that there exists a separating closed geodesic $\gamma \subset X$ which may not be simple such that
\[\ell_{\gamma}(X)\leq C' \log g\]
for some universal constant $C'>0$. If $\gamma$ is simple, we are done. Now we assume that $\gamma$ is non-simple, by Proposition \ref{ns-sub} there exists a subsurface $X(\gamma)$ of $X$ such that
\[\ell(\partial X(\gamma))\leq 2C' \log g \quad \text{and} \quad  \area(X(\gamma))\leq 3C'\log g.\]
Recall that by Gauss-Bonnet we know that $\area(X)=4\pi(g-1)$. So for large $g$, we have that $X(\gamma)$ is a proper subsurface of $X$. Clearly the boundary $\partial X(\gamma)$ consists of multi simple closed geodesics which separate $X$. Hence, we have
\[\mathcal{L}_1(X)\leq \ell(\partial X(\gamma))\leq 2C'\log g\]
which completes the proof by setting $C=2C'>0$. 
\end{proof}

\begin{rem*}
It is known by work of Buser-Sarnak \cite{BS94} that for all $g\geq 2$, there exists a hyperbolic surface $X_g\in \M_g$ such that $\ell_{sys}(X_g)\geq K\log g$ for some uniform constant $K>0$ independent of $g$ (one may also see \cite{BMP14} for more details). This together with the proposition above imply that
$$\sup_{X\in \sM_g}\mathcal{L}_1(X)\asymp \log g.$$
\end{rem*}

Now we are ready to prove Theorem \ref{cor E[L1]}.
\begin{theorem}[=Theorem \ref{cor E[L1]}]\label{cor E[L1]-1}
The expected value $\E[\sL_1]$ of $\sL_1\bracedcdot$ on $\M_g$ satisfies
\begin{equation*}
\limg\frac{\E[\sL_1]}{\log g} = 2.
\end{equation*}
\end{theorem}

\begin{proof}
First we set  
\[\mathcal{B}(\omega(g))=\{X\in \sM_g\,;\, |\mathcal{L}_1(X)-(2\log g-4 \log \log g)|\leq \omega(g)\}.\]
By Theorem \ref{cor L1} we know that
\[\limg \frac{\Vol(\mathcal{B}(\omega(g)))}{V_g}=1.\]

For the lower bound, we have
\begin{eqnarray*}
\frac{\E[\mathcal{L}_1]}{\log g} &\geq & \frac{1}{V_g}\int_{\mathcal{B}(\omega(g))}\frac{\mathcal{L}_1(X) }{\log g}dX \\
&\geq & \frac{2\log g-4 \log \log g-\omega(g)}{\log g}\cdot \frac{\Vol(\mathcal{B}(\omega(g)))}{V_g}
\end{eqnarray*}
which implies that
\[\liminf \limits_{g\to \infty}\frac{\E[\mathcal{L}_1]}{\log g}\geq 2.\]

For the upper bound, it follows by Proposition \ref{L1-upp} that
\begin{eqnarray*}
&& \frac{\E[\mathcal{L}_1]}{\log g} =  \frac{1}{V_g}\int_{\mathcal{B}(\omega(g))}\frac{\mathcal{L}_1(X) }{\log g}dX+ \frac{1}{V_g}\int_{\sM_g\setminus \mathcal{B}(\omega(g))}\frac{\mathcal{L}_1(X) }{\log g}dX \\
&&\leq  \frac{2\log g-4 \log \log g+\omega(g)}{\log g}\cdot \frac{\Vol(\mathcal{B}(\omega(g)))}{V_g}+ C \cdot \frac{\Vol(\sM_g\setminus \mathcal{B}(\omega(g)))}{V_g}. 
\end{eqnarray*}
Letting $g\to \infty$, we get
\[\limsup \limits_{g\to \infty}\frac{\E[\mathcal{L}_1]}{\log g}\leq 2,\]
as required.
\end{proof}

\section{Further questions}\label{section questions}
In this last section we propose several questions related to the results in this article.

\subsection{Shortest separating simple closed multi-geodesics} By Theorem \ref{main} we know that on a generic $X\in \M_g$, a separating systolic closed geodesic of $X$ separates $X$ into $S_{1,1}\cup S_{g-1,1}$. However, by Theorem \ref{cor L1} we only know that on a generic $X\in \M_g$, the shortest separating simple closed multi-geodesics of $X$ separates $X$ into either $S_{1,1}\cup S_{g-1,1}$ or $S_{0,3}\cup S_{g-2,3}$. A natural question is to determine the weights of these two cases. Or more precisely, 
\begin{question}
On a generic $X\in \M_g$, is $\sL_1(X)$ achieved by a separating systole as $g\to \infty$?
\end{question}

\subsection{Expectation of $\lss$}
Theorem \ref{cor E[L1]} tells that as $g\to \infty$, the expected value $\E[\sL_1]$ behaves like $2\log g$. The two ingredients in the proof are Theorem \ref{cor L1} and Proposition \ref{L1-upp}, the latter one says that $\sup_{X\in\mathcal{M}_g}\mathcal{L}_1(X)\leq C \log g$ for some universal constant $C>0$. For $\lss$, although we still have the first ingredient, namely Theorem 1, the second is missing because it is known that $\sup_{X\in \mathcal{M}_g}\lss(X)=\infty$. So we raise the following question:
\begin{question}\label{q-exp-ss}
Does the following limit hold: $$\lim \limits_{g\to \infty}\frac{\E[\lss]}{\log g} = 2?$$
\end{question}

\begin{rem*}
Very recently joint with H. Parlier, the second and third named authors of this article in \cite{PWX20} give an affirmative answer to Question \ref{q-exp-ss} above.
\end{rem*}

\subsection{Geometric Cheeger constants} Recall as in the Introduction, for all $1\leq m\leq g-1$ the \emph{$m$-th geometric Cheeger constant} $H_m(X)$ of $X$ is defined as
\[H_m(X):= \inf \limits_{\gamma}\frac{\ell_{\gamma}(X)}{2\pi m}\]
where $\gamma$ is a simple closed multi-geodesics on $X$ with $X\setminus \gamma=X_1\cup X_2$, and $X_1$ and $X_2$ are connected subsurfaces of $X$ such that $|\chi(X_1)|=m\leq |\chi(X_2)|$. As a direct consequence of Theorem \ref{cor L1}, the first geometric Cheeger constant $H_1\bracedcdot$ on $\M_g$ asymptotically behaves as
\[\lim \limits_{g\to \infty}\Prob \left(X\in \M_g\,;\, (1-\epsilon)\cdot \frac{\log g}{\pi}< H_1(X)< \frac{\log g}{\pi}\right)=1\]
for any $\epsilon>0$. A natural question is to study general $H_m$.
\begin{question}\label{ques-hm}
For $m\in [1,g-1]$, what is the asymptotic behavior of $H_m\bracedcdot$ on $\M_g$ as $g\to \infty$?
\end{question} 
\noindent This question is related to \cite[Problem 10.5]{Wright-tour} of Wright on the asymptotic behavior of the classical Cheeger constant  $h(X)$ of $X$ (which is recently solved by Budzinski-Curien-Petri \cite{BCP22}), because $H(X):=\min_{1\leq m \leq g-1}H_m(X)$ serves as a natural upper bound for $h(X)$. For fixed $m>0$ independent of $g$, the question above may be reduced to the following explicit one:
\begin{question}\label{ques-L1m}
Let $\omega(g)$ be a function as \eqref{eq-omega} and $m>0$ be fixed. Then does the following limit hold: as $g\to \infty$,
\begin{equation*}
\Prob\left(X\in\M_g\,;\, |\sL_{1,m}(X) - (2m\log g - (6m-2)\log\log g)| \leq \omega(g)\right)\to 1?
\end{equation*}
\end{question}

\noindent Theorem \ref{cor L1} answers Question \ref{ques-hm} and \ref{ques-L1m} for $m=1$, By Proposition \ref{lower bound for chi geq 2} it suffices to study the upper bound.

\bibliographystyle{amsalpha}
\bibliography{ref}

\providecommand{\bysame}{\leavevmode\hbox to3em{\hrulefill}\thinspace}
\providecommand{\MR}{\relax\ifhmode\unskip\space\fi MR }
\providecommand{\MRhref}[2]{%
  \href{http://www.ams.org/mathscinet-getitem?mr=#1}{#2}
}
\providecommand{\href}[2]{#2}
\begin{thebibliography}{GMST21}

\bibitem[Agg21]{Agg21}
Amol Aggarwal, \emph{Large genus asymptotics for intersection numbers and
  principal strata volumes of quadratic differentials}, Invent. Math.
  \textbf{226} (2021), no.~3, 897--1010.

\bibitem[Ahl10]{Ahlfors-ci}
Lars~V. Ahlfors, \emph{Conformal invariants}, AMS Chelsea Publishing,
  Providence, RI, 2010, Topics in geometric function theory, Reprint of the
  1973 original, With a foreword by Peter Duren, F. W. Gehring and Brad Osgood.

\bibitem[AM22]{AM22}
Nalini Anantharaman and Laura Monk, \emph{A high-genus asymptotic expansion of
  {W}eil-{P}etersson volume polynomials}, J. Math. Phys. \textbf{63} (2022),
  no.~4, Paper No. 043502, 26.

\bibitem[BCP22]{BCP22}
Thomas {Budzinski}, Nicolas {Curien}, and Bram {Petri}, \emph{{On Cheeger
  constants of hyperbolic surfaces}}, arXiv e-prints (2022), arXiv:2207.00469.

\bibitem[BMP14]{BMP14}
Florent Balacheff, Eran Makover, and Hugo Parlier, \emph{Systole growth for
  finite area hyperbolic surfaces}, Ann. Fac. Sci. Toulouse Math. (6)
  \textbf{23} (2014), no.~1, 175--180.

\bibitem[BS94]{BS94}
P.~Buser and P.~Sarnak, \emph{On the period matrix of a {R}iemann surface of
  large genus}, Invent. Math. \textbf{117} (1994), no.~1, 27--56, With an
  appendix by J. H. Conway and N. J. A. Sloane.

\bibitem[Bus10]{Buser10}
Peter Buser, \emph{Geometry and spectra of compact {R}iemann surfaces}, Modern
  Birkh\"{a}user Classics, Birkh\"{a}user Boston, Ltd., Boston, MA, 2010,
  Reprint of the 1992 edition.

\bibitem[DGZZ21]{DGZZ21}
Vincent Delecroix, \'{E}lise Goujard, Peter Zograf, and Anton Zorich,
  \emph{Masur-{V}eech volumes, frequencies of simple closed geodesics, and
  intersection numbers of moduli spaces of curves}, Duke Math. J. \textbf{170}
  (2021), no.~12, 2633--2718.

\bibitem[DGZZ22]{DGZZ22}
\bysame, \emph{Large genus asymptotic geometry of random square-tiled surfaces
  and of random multicurves}, Invent. Math. \textbf{230} (2022), no.~1,
  123--224.

\bibitem[GMST21]{GMST19}
Clifford Gilmore, Etienne~Le Masson, Tuomas Sahlsten, and Joe Thomas,
  \emph{Short geodesic loops and $l^p$ norms of eigenfunctions on large genus
  random surfaces}, Geom. Funct. Anal. \textbf{31} (2021), 62--110.

\bibitem[GPY11]{GPY11}
Larry Guth, Hugo Parlier, and Robert Young, \emph{Pants decompositions of
  random surfaces}, Geom. Funct. Anal. \textbf{21} (2011), no.~5, 1069--1090.

\bibitem[Gru01]{Grus01}
Samuel Grushevsky, \emph{An explicit upper bound for {W}eil-{P}etersson volumes
  of the moduli spaces of punctured {R}iemann surfaces}, Math. Ann.
  \textbf{321} (2001), no.~1, 1--13.

\bibitem[Kee74]{Kee74}
Linda Keen, \emph{Collars on {R}iemann surfaces}, 263--268. Ann. of Math.
  Studies, No. 79.

\bibitem[Ker80]{Kerck80}
Steven~P. Kerckhoff, \emph{The asymptotic geometry of {T}eichm\"{u}ller space},
  Topology \textbf{19} (1980), no.~1, 23--41.

\bibitem[LX14]{LX14}
Kefeng Liu and Hao Xu, \emph{A remark on {M}irzakhani's asymptotic formulae},
  Asian J. Math. \textbf{18} (2014), no.~1, 29--52.

\bibitem[Mas85]{Maskit}
Bernard Maskit, \emph{Comparison of hyperbolic and extremal lengths}, Ann.
  Acad. Sci. Fenn. Ser. A I Math. \textbf{10} (1985), 381--386.

\bibitem[McS98]{McS98}
Greg McShane, \emph{Simple geodesics and a series constant over {T}eichmuller
  space}, Invent. Math. \textbf{132} (1998), no.~3, 607--632.

\bibitem[Mir07a]{Mirz07}
Maryam Mirzakhani, \emph{Simple geodesics and {W}eil-{P}etersson volumes of
  moduli spaces of bordered {R}iemann surfaces}, Invent. Math. \textbf{167}
  (2007), no.~1, 179--222.

\bibitem[Mir07b]{Mirz07-int}
\bysame, \emph{Weil-{P}etersson volumes and intersection theory on the moduli
  space of curves}, J. Amer. Math. Soc. \textbf{20} (2007), no.~1, 1--23.

\bibitem[Mir08]{Mirz08}
\bysame, \emph{Growth of the number of simple closed geodesics on hyperbolic
  surfaces}, Ann. of Math. (2) \textbf{168} (2008), no.~1, 97--125.

\bibitem[Mir10]{Mirz10}
\bysame, \emph{On {W}eil-{P}etersson volumes and geometry of random hyperbolic
  surfaces}, Proceedings of the {I}nternational {C}ongress of {M}athematicians.
  {V}olume {II}, Hindustan Book Agency, New Delhi, 2010, pp.~1126--1145.

\bibitem[Mir13]{Mirz13}
\bysame, \emph{Growth of {W}eil-{P}etersson volumes and random hyperbolic
  surfaces of large genus}, J. Differential Geom. \textbf{94} (2013), no.~2,
  267--300.

\bibitem[MP19]{MP19}
Maryam Mirzakhani and Bram Petri, \emph{Lengths of closed geodesics on random
  surfaces of large genus}, Comment. Math. Helv. \textbf{94} (2019), no.~4,
  869--889.

\bibitem[MT21]{MT20}
Laura Monk and Joe Thomas, \emph{{The Tangle-Free Hypothesis on Random
  Hyperbolic Surfaces}}, International Mathematics Research Notices (2021), to
  appear.

\bibitem[MZ15]{MZ15}
Maryam Mirzakhani and Peter Zograf, \emph{Towards large genus asymptotics of
  intersection numbers on moduli spaces of curves}, Geom. Funct. Anal.
  \textbf{25} (2015), no.~4, 1258--1289.

\bibitem[Pen92]{Penner92}
R.~C. Penner, \emph{Weil-{P}etersson volumes}, J. Differential Geom.
  \textbf{35} (1992), no.~3, 559--608.

\bibitem[PWX21]{PWX20}
Hugo {Parlier}, Yunhui {Wu}, and Yuhao {Xue}, \emph{{The simple separating
  systole for hyperbolic surfaces of large genus}}, Journal of the Institute of
  Mathematics of Jussieu (2021), to appear.

\bibitem[Sab08]{Sabo08}
St\'{e}phane Sabourau, \emph{Asymptotic bounds for separating systoles on
  surfaces}, Comment. Math. Helv. \textbf{83} (2008), no.~1, 35--54.

\bibitem[ST01]{ST01}
Georg Schumacher and Stefano Trapani, \emph{Estimates of {W}eil-{P}etersson
  volumes via effective divisors}, Comm. Math. Phys. \textbf{222} (2001),
  no.~1, 1--7.

\bibitem[SWY80]{SWY80}
R.~Schoen, S.~Wolpert, and S.~T. Yau, \emph{Geometric bounds on the low
  eigenvalues of a compact surface}, Geometry of the {L}aplace operator
  ({P}roc. {S}ympos. {P}ure {M}ath., {U}niv. {H}awaii, {H}onolulu, {H}awaii,
  1979), Proc. Sympos. Pure Math., XXXVI, Amer. Math. Soc., Providence, R.I.,
  1980, pp.~279--285.

\bibitem[Wol82]{Wolpert82}
Scott Wolpert, \emph{The {F}enchel-{N}ielsen deformation}, Ann. of Math. (2)
  \textbf{115} (1982), no.~3, 501--528.

\bibitem[Wol10]{Wolpert-book}
Scott~A. Wolpert, \emph{Families of {R}iemann surfaces and {W}eil-{P}etersson
  geometry}, CBMS Regional Conference Series in Mathematics, vol. 113,
  Published for the Conference Board of the Mathematical Sciences, Washington,
  DC; by the American Mathematical Society, Providence, RI, 2010.

\bibitem[Wri20]{Wright-tour}
Alex Wright, \emph{A tour through {M}irzakhani's work on moduli spaces of
  {R}iemann surfaces}, Bull. Amer. Math. Soc. (N.S.) \textbf{57} (2020), no.~3,
  359--408.

\bibitem[Wu19]{Wu19}
Yunhui Wu, \emph{Growth of the {W}eil-{P}etersson inradius of moduli space},
  Ann. Inst. Fourier (Grenoble) \textbf{69} (2019), no.~3, 1309--1346.

\bibitem[WX22]{WX18}
Yunhui Wu and Yuhao Xue, \emph{Small eigenvalues of closed {R}iemann surfaces
  for large genus}, Trans. Amer. Math. Soc. \textbf{375} (2022), no.~5,
  3641--3663.

\bibitem[{Zog}08]{Zograf08}
Peter {Zograf}, \emph{{On the large genus asymptotics of Weil-Petersson
  volumes}}, arXiv e-prints (2008), arXiv:0812.0544.

\end{thebibliography}

\end{document}